\documentclass[10pt,a4paper]{article}
\usepackage[utf8]{inputenc}
\usepackage{amsmath, amsthm}
\usepackage{esint}
\usepackage{amsfonts}
\usepackage{amssymb}
\usepackage{hyperref}
\usepackage{cleveref}
\usepackage{graphicx}
\usepackage[T1]{fontenc}
\usepackage[francais,english]{babel}
\usepackage{listings}
\usepackage{vmargin} 
\usepackage{dsfont}
\usepackage{color}

\usepackage[final]{pdfpages}
\numberwithin{equation}{section}

\newtheorem{de}{Definition}[section]
\newtheorem{thm}[de]{Theorem}

\newtheorem{prop}[de]{Proposition}
\newtheorem{lem}[de]{Lemma}
\newtheorem{rem}[de]{Remark}

\newcommand\eps{\varepsilon}
\newcommand\E{\mathcal{E}}
\newcommand\F{\mathcal{F}}
\newcommand\V{\mathcal{V}}
\newcommand\A{\mathcal{A}}

\newcommand\G{\mathcal{G}}
\newcommand\C{\mathcal{C}}
\newcommand\W{\mathcal{W}}

\newcommand\M{\mathcal{M}}

\newcommand\vv{\mathbf{v}}
\newcommand\rd{\mathrm{d}}
\newcommand\KFR{\mathtt{WFR}}
\newcommand\FR{\mathtt{FR}}
\newcommand\HK{\mathtt{HK}}
\renewcommand\W{\mathtt{W}}
\newcommand\grad{\operatorname{grad}}

\newcommand\trho{\tilde{\rho}}

\newcommand\supp{\mathop{\rm supp}}

\renewcommand{\textbf}[1]{\begingroup\bfseries\mathversion{bold}#1\endgroup}

\DeclareMathOperator*{\dive}{div}
\DeclareMathOperator*{\argmin}{argmin}
\newcommand\Rn{\mathbb{R}^d}
\DeclareMathOperator*{\R}{\mathbb{R}}

\DeclareMathOperator*{\Paa}{\mathcal{P}^{ac}}

\title{An unbalanced Optimal Transport splitting scheme for general advection-reaction-diffusion problems}
\author{T.O. Gallou\"et, M. Laborde, L. Monsaingeon}

\begin{document}

\maketitle

\begin{abstract}
In this paper, we show that unbalanced optimal transport provides a convenient framework to handle reaction and diffusion processes in a unified metric framework.
We use a constructive method, alternating minimizing movements for the Wasserstein distance and for the Fisher-Rao distance, and prove existence of weak solutions for general scalar reaction-diffusion-advection equations.
We extend the approach to systems of multiple interacting species, and also consider an application to a very degenerate diffusion problem involving a Gamma-limit.
Moreover, some numerical simulations are included.
\end{abstract}

\section{Introduction}
Since the seminal works of Jordan-Kinderlehrer-Otto \cite{JKO}, it is well known that certain diffusion equations can be interpreted as gradient flows in the space of probability measures, endowed with the quadratic Wasserstein distance $\W$.
The well-known JKO scheme (a.k.a. minimizing movement), which is a natural implicit Euler scheme for such gradient flows, naturally leads to constructive proofs of existence for weak solutions to equations or systems with mass conservation such as, for instance, Fokker-Planck equations \cite{JKO}, Porous Media Equations \cite{O1}, aggregation equation \cite{CDFLS}, double degenerate diffusion equations \cite{O}, general degenerate parabolic equation \cite{A} etc.
We refer to the classical textbooks of Ambrosio, Gigli and Savar\'e \cite{AGS} and to the books of Villani \cite{V1,V2} for a detailed account of the theory and extended bibliography. Recently, this theory has been extended to study the evolution of interacting species with mass-conservation, see for examples \cite{DFF,Zinsl,L,KMX,CL1}.  

 Nevertheless in biology, for example for diffusive prey-predator models, the conservation of mass may not hold, and the classical optimal transport theory does not apply.
 An unbalanced optimal transport theory was recently introduced simultaneously in \cite{CPSV,CPSV1,KMV,LMS_big_15,LMS_small_15}, and the resulting Wasserstein-Fisher-Rao ($\KFR$) metrics (also referred to as the Hellinger-Kantorovich distance $\HK$) allows to compute distances between measures with variable masses while retaining a convenient Riemannian structure.
 See section \ref{part4-section2WFR} for the definition and a short discussions on this $\KFR$ metric.
 We also refer to \cite{piccoli_rossi_generalized_14,figalli2010new} for earlier attempts to account for mass variations within the framework of optimal transport.
 
The $\KFR$ metrics can be seen as an \emph{inf-convolution} between Wasserstein/transport and Fisher-Rao/reaction processes, and is therefore extremely convenient to control both in a unified metric setting.
This allows to deal with non-conservative models of population dynamics, see e.g. \cite{KMV,Kondratyev20162784}.
In \cite{GM}, the first and third authors proposed a variant of the JKO scheme for $\KFR$-gradient flows corresponding to some particular class of reaction-diffusion PDEs: roughly speaking, the reaction and diffusion were handled separately in two separate $\FR,\W$ metrics, and then patched together using a particular uncoupling of the inf-convolution, namely $\KFR^2\approx \W^2+\FR^2$ in some sense (see \cite[section 3]{GM} for a thorough discussion).
However, the analysis was restricted to very particular structures for the PDE, corresponding to pure $\KFR$ gradient-flows.

In this work we aim at extending this splitting scheme in order to handle more general reaction-diffusion problems, not necessarily corresponding to gradient flows.
Roughly speaking, the structure of our splitting scheme is the following: the transport/diffusion part of the PDE is treated by a single Wasserstein JKO step
$$
\rho^k\xrightarrow[\mbox{transport}]{\W}\rho^{k+1/2},
$$
and the next Fisher-Rao JKO step
$$
\rho^{k+1/2}\xrightarrow[\mbox{reaction}]{\FR}\rho^{k+1}
$$
handles the reaction part of the evolution.
As already mentioned, the $\KFR$ metric will allow to suitable control both steps in a unified metric framework.
We will first state a general convergence result for scalar reaction-diffusion equations, and then illustrate on a few particular examples how the general idea can be adapted to treat e.g. prey-predator systems or very degenerate Hele-Shaw diffusion problems.
In this work we do not focus on optimal results and do not seek full generality, but rather wish to illustrate the efficiency of the general approach.

Another advantage of the splitting scheme is that is well adapted to existing Monge/Kantorovich/Wasserstein numerical solvers, and the Fisher-Rao step turns out to be a simple pointwise convex problem which can be implemented in a very simple way.
See also \cite{simone_lenaic,chizat2016scaling} for a more direct numerical approach by entropic regularization.
Throughout the paper we will illustrate the theoretical results with a few numerical tests.
All the numerical simulations were implemented with the augmented Lagrangian {\rm ALG2-JKO} scheme from \cite{BCL} for the Wasserstein step, and we used a classical Newton algorithm for the Fisher-Rao step.
\\

The paper is organized as follows.
In section~\ref{part4-section2WFR} we recall the basic definitions and useful properties of the Wasserstein-Fisher-Rao distance $\KFR$.
Section~\ref{sec:KFRsplitting} contains the precise description of the splitting scheme and a detailed convergence analysis for a broad class of reaction-diffusion equations.
In section~\ref{section:systems} we present an extension to some prey-predator multicomponent systems with nonlocal interactions.
In section~\ref{part4-section2HS} we extend the general result from section~\ref{sec:KFRsplitting} to a very degenerate tumor growth model studied in \cite{PQV}, corresponding to a pure $\KFR$ gradient flow: we show that the splitting scheme captures fine properties of the model, particularly the $\Gamma$-convergence of discrete gradient flows as the degenerate diffusion parameter of Porous Medium type $m\to\infty$.
The last section~\ref{part4-section2HSN} contains an extension to a tumor-growth model coupled with an evolution equation for the nutrients.

\section{Preliminaries}
\label{part4-section2WFR}

Let us first fix some notations.
Throughout the whole paper, $\Omega$ denotes a possibly unbounded convex subset of $\R^d$, $Q_T$ represents the product space $[0,T] \times \Omega$, for $T>0$, and we write $\M^+=\M^+(\Omega)$ for the set of nonnegative finite Radon measures on $\Omega$.
We say that a curve of measures $t\mapsto \rho_t\in\mathcal C_w([0,1];\M^+)$ is narrowly continuous if it is continuous with respect to the narrow convergence of measures, namely for the duality with $\C_b(\Omega)$ test-functions.
\begin{de}
\label{def:FR}
The Fisher-Rao distance between $\rho_0,\rho_1\in \M^+$ is
 $$
 \FR(\rho_0,\rho_1):= \min_{(\rho_t,r_t) \in \A_{\FR}[\rho_0,\rho_1]} \int_0^1\int_\Omega |r_t|^2 \,d\rho_t(x)dt,
 $$
 where the admissible set $A_{\FR}[\rho_0,\rho_1]$ consists in curves $[0,1]\ni t\mapsto (\rho_t,r_t)\in\M^+\times\M$ such that $t\mapsto \rho_t\in\mathcal C_w([0,1];\M^+)$ is narrowly continuous with endpoints $\rho_t(0)=\rho_0,\rho_t(1)=\rho_1$, and
 $$
 \partial_t\rho_t=\rho_tr_t
 $$
 in the sense of distributions $\mathcal D'((0,1)\times\Omega)$.
\end{de}
The Monge-Kantorovich-Wasserstein admits several equivalent definitions and formulations, and we refer e.g. to \cite{V1,V2,AGS,S} for a complete description.
For our purpose we shall only need the dynamical Benamou-Brenier formula:
\begin{thm}[Benamou-Brenier formula, \cite{BB,AGS}]
 \label{theo:benamou_brenier}
There holds
\begin{equation}
\label{eq:benamou_brenier_formula}
\W^2(\rho_0,\rho_1)=\min\limits_{(\rho,\vv)\in\mathcal A_\W[\rho_0,\rho_1]}\int_0^1\int_\Omega|\vv_t|^2\rd\rho_t\rd t,
\end{equation}
where the admissible set $\mathcal A_\W[\rho_0,\rho_1]$ consists in curves $(0,1)\ni t\mapsto (\rho_t,\vv_t)\in \M^+\times \M(\Omega;\R^d)$ such that $t\mapsto\rho_t$ is narrowly continuous with endpoints $\rho_t(0)=\rho_0$, $\rho_t(1)=\rho_1$ and solving the continuity equation
$$
\partial_t\rho_t+\dive(\rho_t\vv_t)=0
$$
in the sense of distributions $\mathcal D'((0,1)\times\Omega)$.
\end{thm}

According to the original definition in \cite{CPSV} we have
\begin{de}
\label{def:KFR}
The Wasserstein-Fisher-Rao distance between $\rho_0, \rho_1 \in \M^+(\Omega)$ is 
\begin{equation}
\label{eq:def_KFR}
\KFR^2(\rho_0,\rho_1):=\inf_{(\rho,\vv,r) \in \A_{\KFR}[\rho_0,\rho_1]} \int_0^1 \int_\Omega (|\vv_t(x)|^2+|r_t|^2)\,d\rho_t(x)dt, 
\end{equation}
where the admissible set $\A_{\KFR}[\rho_0,\rho_1]$ is the set of curves $t \in [0,1] \mapsto (\rho_t,v_t,r_t) \in \M^+\times \M(\Omega;\R^d) \times \M$ such that $t\mapsto \rho_t \in \mathcal{C}_w([0,1],\M^+)$ is narrowly continuous with endpoints $\rho_{|t=0}=\rho_0$, $\rho_{|t=1}=\rho_1$ and solves the continuity equation with source
$$ \partial_t \rho_t +\dive(\rho_t v_t) = \rho_t r_t.$$
\end{de}

Comparing definition \ref{def:KFR} with definition \ref{def:FR} and Theorem~\ref{theo:benamou_brenier}, this dynamical formulation {\it \`a la Benamou-Brenier} shows that the $\KFR$ distance can be viewed as an inf-convolution of the Wasserstein and Fisher-Rao distances $\W,\FR$.
From \cite{CPSV,CPSV1,KMV,LMS_big_15} the infimum in \eqref{eq:def_KFR} is always a minimum, and the corresponding minimizing curves $t\mapsto\rho_t$ are of course constant-speed geodesics $\KFR(\rho_t,\rho_s)=|t-s|\KFR(\rho_0,\rho_1)$.
Then $(\M^+,\KFR)$ is a complete metric space, and $\KFR$ metrizes the narrow convergences of measures (see again \cite{CPSV,CPSV1,KMV,LMS_big_15}).
Interestingly, there are other possible formulations of the distance in terms of static unbalanced optimal transportation, primal-dual characterizations with relaxed marginals, lifting to probability measures on a cone over $\Omega$, duality with subsolutions of Hamilton-Jacobi equations, and we refer to \cite{CPSV,CPSV1,KMV,LMS_small_15,LMS_big_15} for more details.

As a first useful interplay between the distances $\KFR,\W,\FR$ we have
\begin{prop}[\cite{GM}]
\label{prop:comparison_d_W_H}
Let $\rho_0,\rho_1\in \M^+_2$ such that $|\rho_0|=|\rho_1|$. Then
$$
\KFR^2(\rho_0,\rho_1)\leqslant \W^2(\rho_0,\rho_1).
$$
Similarly for all $\mu_0,\mu_1\in \M^+$ (with possibly different masses) there holds
$$
\KFR^2(\mu_0,\mu_1)\leqslant \FR^2(\mu_0,\mu_1).
$$
Finally, for all $\nu_0,\nu_1\in \M^+_2$ such that $|\nu_0|=|\nu_1|$ and all $\nu\in \M^+$, there holds
$$
\KFR^2(\nu_0,\nu)\leqslant 2(\W^2(\nu_0,\nu_1)+\FR^2(\nu_1,\nu)).
$$
\end{prop}
Moreover, we have the following link between the reaction and the velocity in \eqref{eq:def_KFR}, which was the original definition in \cite{KMV}:
\begin{prop}[\cite{GM}]
\label{prop:coupling=uncoupling_d}
The definition \eqref{def:KFR} of the $\KFR$ distance can be restricted to the subclass of admissible paths $(\vv_t,r_t)=(\nabla u_t,u_t)$  for potentials $u_t\in H^1(\rd\rho_t)$ and continuity equations
$$
\partial_t\rho_t+\dive(\rho_t\nabla u_t)=\rho_t u_t.
$$
\end{prop}
This shows that $(\M^+,\KFR)$ can be endowed with the formal Riemannian structure constructed as follow: any two tangent vectors $\xi^1=\partial_t\rho^1,\xi^2=\partial_t\rho^2$ can be uniquely identified with potentials $u^i$ by solving the elliptic equations
$$
\xi^i=-\dive(\rho\nabla u^i)+\rho u^i.
$$
Then the Riemaniann tensor is naturally constructed on the $H^1(\rd\rho)$ scalar product, i-e
$$
g_\rho(\xi^1,\xi^2):=\langle u^1,u^2\rangle_{H^1(\rd\rho)}=\int_{\Omega}(\nabla u^1\cdot\nabla u^2+ u^1 u^2)\rd\rho.
$$
This is purely formal, and we refer again to \cite{GM} for discussions.
Given a functional
$$
\F(\rho):=\int_\Omega F(\rho)+\int_\Omega \rho V +\frac{1}{2}\int_{\Omega}(K\ast \rho)\rho,
$$
this Riemannian structure also allows to compute $\KFR$ gradients as
$$
\grad_{\KFR}\F(\rho)
=-\dive\left(\rho\nabla \frac{\delta\F}{\delta\rho}\right)+\rho \frac{\delta\F}{\delta\rho}
=\grad_{\W}\F(\rho)+\grad_\FR\F(\rho),
$$
where $\frac{\delta\F}{\delta\rho}=F'(\rho)+V+K\ast\rho$ denotes the Euclidean first variation of $\F$ with respect to $\rho$.
In other words, the Riemannian tangent vector $\grad_\KFR\F(\rho)$ is represented in the previous $H^1(\rd\rho)$ duality by the scalar potential $u=\frac{\delta\F}{\delta\rho}$.

\section{An existence result for general parabolic equations}
\label{sec:KFRsplitting}
In this section, we propose to solve scalar parabolic equations of the form
\begin{equation}
\label{eq:KFR-general equation}
\left\{
\begin{array}{l}
\partial_t \rho = \dive(\rho \nabla (F'_1(\rho) +V_1)) -\rho( F'_2(\rho)+V_2) \\
\rho|_{t=0}=\rho_0\\
\left.{}_{}\rho \nabla (F'_1(\rho) +V_1)\right|_{\partial\Omega}\cdot \nu =0
\end{array}
\right.
\end{equation}
in a bounded domain $\Omega \subset \Rn$ with Neumann boundary condition and suitable initial conditions. 
Our goal is to extend to the case $F_1\neq F_2,V_1\neq V_2$ the method initially introduced in \cite{GM} for variational $\KFR$-gradient flows, i-e \eqref{eq:KFR-general equation} with $F_1=F_2$ and $V_1=V_2$.

We assume for simplicity that $F_1 \, : \,  \R \rightarrow \R$ is given by
\begin{eqnarray}
\label{assumption Diffusion}
F_1(z)=\left\{
\begin{array}{ll}
z\log z -z & \mbox{(linear diffusion)}\\
\mbox{or}\\
\frac{1}{m_1-1}z^{m_1} & \mbox{(Porous Media diffusion)}
\end{array}
\right.,
\end{eqnarray}

and $F_2\, : \, \R \rightarrow \R$ is given by
\begin{eqnarray}
\label{assumption reaction}
F_2(z) = \frac{1}{m_2-1}z^{m_2}, \qquad\mbox{for some } m_2>1.
\end{eqnarray}
Note that we cannot take $F_2(z)=z\log z-z$ because the Boltzmann entropy is not well behaved (neither regular nor convex) with respect to the Fisher-Rao metric in the reaction step, see \cite{GM,LMS_small_15,LMS_big_15} for discussions.
In addition, we assume that
$$
V_1 \in W^{1,\infty}(\Omega)
\qquad\mbox{and}\qquad
V_2 \in L^\infty(\Omega).
$$
We denote $\E_1, \E_2 \, :\, \M^+ \rightarrow \R$ the energy functionals
$$
\E_i(\rho) := \F_i(\rho) + \V_i(\rho),
$$
where
$$
\F_i(\rho):= \left\{ \begin{array}{ll}
\int_\Omega F_i(\rho) & \text{    if } \rho \ll \mathcal{L}_{|\Omega}\\
+\infty & \text{    otherwise, }
\end{array}\right.
\qquad
\text{and}
\qquad
\V_i(\rho) := \int_\Omega V_i \rho.
$$
 Although more general statements with suitable structural assumptions could certainly be proved, we do not seek full generality here and choose to restrict from the beginning to the above simple (but nontrivial) setting for the sake of exposition.

\begin{de}
\label{def:weak_solutions}
A weak solution of \eqref{eq:KFR-general equation} is a curve $ [0, +\infty) \ni t\mapsto \rho(t,\cdot) \in L^1_+\cap L^\infty(\Omega)$ such that for all $T<\infty$ the pressure $P_1(\rho):=\rho F_1'(\rho)-F_1(\rho)$ satisfies $\nabla P_1(\rho) \in L^2([0,T] \times \Omega)$, and
$$
 \int_0^{+\infty}\left( \int_\Omega (\rho \partial_t \phi -\nabla V_1 \cdot \nabla \phi \rho - \nabla P_1(\rho) \cdot \nabla \phi - \rho (F_2'(\rho) +V_2) \phi) \,dx \right)\, dt = -\int_\Omega \phi(0,x) \rho_0(x) \,dx
$$
for every $\phi \in \C^\infty_c([0,+\infty) \times \Rn)$.
\end{de}
\noindent
Note that the pressure $P_1$ is defined so that the diffusion term $\dive(\rho\nabla F_1'(\rho))=\Delta P_1(\rho)$, at least for smooth solutions.\\

The starting point of our analysis is that \eqref{eq:KFR-general equation} can be written, at least formally as,
$$
\partial_t \rho = \dive(\rho \nabla (F'_1(\rho) +V_1)) -\rho( F'_2(\rho)+V_2)
\quad \leftrightarrow\quad 
\partial_t\rho=-\grad_{\W}\mathcal E_1(\rho)-\grad_{\FR}\mathcal E_2(\rho).
$$
Our splitting scheme is a variant of that originally introduced in \cite{GM}, and can be viewed as an operator splitting method: each part of the PDE above is discretized (in time) in its own $\W,\FR$ metric, and corresponds respectively to a $\W$/transport/diffusion step and to a $\FR$/reaction step.
More precisely, let $h>0$ be a small time step.
Starting from the initial datum $ \rho_h^0 :=\rho_0$, we construct two recursive sequences $(\rho_h^k)_k$ and $(\rho_h^{k+1/2})_k$ such that
\begin{eqnarray}
 \label{splitting scheme}
\left\{\begin{array}{l}
\rho_h^{k+1/2} \in \argmin\limits_{\rho \in \M^+, |\rho|=|\rho_h^k|} \left\{ \frac{1}{2h} \W^2(\rho, \rho_h^k) + \E_1(\rho ) \right\},\\
\\
\rho_h^{k+1} \in \argmin\limits_{\rho \in \M^+} \left\{ \frac{1}{2h} \FR_2^2(\rho, \rho_h^{k+1/2}) + \E_2(\rho ) \right\}.
\end{array}\right.
\end{eqnarray}

With our structural assumptions on $F_i,V_i$ and arguing as in \cite{GM}, the direct method shows that this scheme is well-posed, i-e that each minimizing problem in \eqref{splitting scheme} admits a unique minimizer.
We construct next two piecewise-constant interpolating curves
\begin{eqnarray}
\label{piecewise-constant curves}
\left\{\begin{array}{l}
\rho_h(t) = \rho_h^{k+1},\\
\trho_h(t) = \rho_h^{k + 1/2},
\end{array}\right.
 \text{  for all }t\in (kh, (k+1)h].
\end{eqnarray}

Our main results in this section is the constructive existence of weak solutions to \eqref{eq:KFR-general equation}:
\begin{thm}
\label{theo:existence differente energies}
Assume that $\rho_0 \in L^1_+\cap L^\infty(\Omega)$.
Then, up to a discrete subsequence (still denoted $h\to 0$ and not relabeled here), $\rho_h$ and $\trho_h$ converge strongly in $L^1((0,T)\times \Omega)$ to a 
weak solution $\rho$ of \eqref{eq:KFR-general equation}.
\end{thm}
Note that any uniqueness for \eqref{eq:KFR-general equation} would imply convergence of the whole (continuous) sequence $\rho_h,\trho_h\to \rho$ as $h\to 0$, but for the sake of simplicity we shall not address this issue here.

The main technical obstacle in the proof of Theorem \ref{theo:existence differente energies} is to retrieve compactness in time. 
For the classical minimizing scheme of any energy $\E$ on any metric space $(X,d)$, suitable time compactness is usually retrieved in the form of the \emph{total-square distance estimate}
$\frac{1}{2h}\sum\limits_{k\geq 0}d^2(x^k,x^{k+1})\leqslant \E(x_0)-\inf \E$.
This usually works because only one functional is involved, and $\E(x_0)-\inf\E$ is obtained as a telescopic sum of one-step energy dissipations $\E(x^{k+1})-\E(x^k)$.
Here each of our elementary step in \eqref{eq:KFR-general equation} involves one of the $\W,\FR$ metrics, and we will use the $\KFR$ distance to control both simultaneously: this strongly leverages the inf-convolution structure, the $\KFR$ distance being precisely built on a compromise between $\W$/transport and $\FR$/reaction. 
On the other hand we also have two different functionals $\E_1,\E_2$, and we will have to carefully estimate the dissipation of $\E_1$ during the $\FR$ reaction step (driven by $\E_2$) as well as the dissipation of $\E_2$ during the $\W$ transport/diffusion step (driven by $\E_1$).

We start by collecting one-step estimates, exploiting the optimality conditions for each elementary minimization procedure, and postpone the proof of Theorem~\ref{theo:existence differente energies} to the end of the section.

\subsection{Optimality conditions and pointwise $L^\infty$ estimates}

The optimality conditions for the first Wasserstein step $\rho^k\to\rho^{k+1/2}$ in \eqref{splitting scheme} are by now classical, and can be written for example

\begin{eqnarray}
\label{opt Wasserstein}
\frac{-\nabla \varphi_h^{k+1/2}}{h} \rho_h^{k+1/2} = \nabla P_1(\rho_h^{k+1/2}) + \rho_h^{k+1/2} \nabla V_1\qquad \mbox{a.e.}
\end{eqnarray}
Here $\varphi_h^{k+1/2}$ is an optimal (backward) Kantorovich potential from $\rho_h^{k+1/2}$ to $\rho_h^{k}$.

\begin{lem}
\label{lem:maximum_principle_W_step}
For all $k \geqslant 0$,
 \begin{equation}
 \label{eq: mass-conservation W}
 \|\rho_h^{k+1/2}\|_{L^1}=\|\rho_h^k\|_{L^1}
 \end{equation}
 and for all constant $C$ such that $V_1 \leqslant C$, 
 \begin{equation}
 \label{eq:PMAX_W}
\rho_h^k (x) \leqslant (F_1')^{-1}(C -V_1(x))\, \mbox{a.e.} 
\qquad
\Rightarrow \qquad
\rho_h^{k+1/2} (x) \leqslant (F_1')^{-1}(C -V_1(x)) \, \mbox{a.e.}
 \end{equation}

\end{lem}
\begin{proof}
The Wasserstein step is mass conservative by construction, so the first part is obvious.
The second part is a direct consequence of a generalization \cite[lemma 2]{PT} of Otto's maximum principle \cite{O1}. 
\end{proof}

\begin{rem}
Note that if $\rho_h^k \leqslant M$, we may take $C=F_1'(M)+\|V_1\|_{L^\infty}$ in \eqref{eq:PMAX_W}.
Formally, this corresponds to taking $\overline{\rho}(x):=(F_1')^{-1}(C -V_1(x))$ as a stationary Barenblatt supersolution for $\partial_t\rho =\dive(\rho\nabla (F_1'(\rho)+V_1))$ at the continuous level.
In addition, if $V_1\equiv 0$ we recover Otto's maximum principle \cite{O1} in the form $\|\rho^{k+1/2}\|_{L^\infty}\leqslant \|\rho^k\|_{L^\infty}$.

\end{rem}

For the second Fisher-Rao reaction step, the optimality condition has been obtained in \cite[section 4.2]{GM} in the form
\begin{eqnarray}
\label{opt Fisher-Rao}
 \left( \sqrt{\rho_h^{k+1}} - \sqrt{\rho_h^{k+1/2}}\right) \sqrt{\rho_h^{k+1}}= - \frac{h}{2}\rho_h^{k+1} \left( F_2'(\rho_h^{k+1}) + V_2 \right) \qquad \mbox{a.e.}
\end{eqnarray}
As a consequence we have

\begin{lem}
\label{lem encadrement}
There is $C\equiv C(V_2)>0$ such that for $h\leqslant h_0(V_2)$ small enough we have

\begin{equation}
\label{encadrement FR_sup}
\rho_h^{k+1}(x) \leqslant (1+Ch)\rho_h^{k+1/2}(x) \qquad \mbox{a.e.},
\end{equation}
and for all $M>0$ there is $c\equiv c(M,V_2)$ such that if $\|\rho_h^{k+1/2}\|_{\infty}\leqslant M$ then
\begin{equation}
\label{encadrement FR_inf}
(1-ch)\rho_h^{k+1/2}(x) \leqslant \rho_h^{k+1}(x) \qquad \mbox{a.e.}
\end{equation}
\end{lem}
Note in particular that this immediately implies
\begin{equation}
\label{eq:supports}
\supp\,\rho_h^{k+1} = \supp\,\rho_h^{k+1/2},
\end{equation}
which was to be expected since the reaction part $\partial_t\rho =-\rho (F_2'(\rho)+V_2)$ of the PDE \eqref{eq:KFR-general equation} preserves strict positivity.
\begin{proof}
We start with the upper bound: inside $\supp \rho_h^{k+1}$, \eqref{opt Fisher-Rao} and $F_2'\geqslant 0$ give
\begin{eqnarray*}
\sqrt{\rho_h^{k+1}(x)} - \sqrt{\rho_h^{k+1/2}(x)} &=& -h \sqrt{\rho_h^{k+1}(x)}(F_2'(\rho_h^{k+1}(x)) +V_2(x))\\
&\leqslant & -hV_2(x) \sqrt{\rho_h^{k+1}(x)}
\leqslant h\|V_2\|_\infty\sqrt{\rho_h^{k+1}(x)}
\end{eqnarray*} 
whence
$$
\sqrt{\rho_h^{k+1}(x)} \leqslant \frac{1}{1-h\|V_2\|_\infty}\sqrt{\rho_h^{k+1/2}(x)}.
$$
Taking squares and using
$$
\frac{1}{(1-h\|V_2\|_\infty)^2} =1+2\|V_2\|_{L^\infty}h+\mathcal O(h^2)\leqslant 1+3\|V_2\|_{L^\infty}h
$$
for small $h$ gives the desired inequality.

For the lower bound \eqref{encadrement FR_inf}, we first observe that since $F_2''\geqslant 0$ and from \eqref{encadrement FR_sup} we have $F_2'(\rho^{k+1}_h)\leqslant F_2'((1+Ch)\rho^{k+1/2}_h)\leqslant F_2'(2M)$ if $h$ is small enough.
Then \eqref{opt Fisher-Rao} gives inside $\supp \rho^{k+1}$
\begin{eqnarray*}
\sqrt{\rho_h^{k+1}(x)} - \sqrt{\rho_h^{k+1/2}(x))}&=& -h\sqrt{\rho_h^{k+1}(x)} (F_2'(\rho_h^{k+1}(x)) +V_2(x))\\
&\geqslant & -h(F_2'(2M) +\|V_2\|_\infty)\sqrt{\rho_h^{k+1}(x)},
\end{eqnarray*} 
hence
$$
\rho_h^{k+1}(x)\geqslant \frac{1}{(1+h(F_2'(2M) +\|V_2\|_\infty))^2}\rho_h^{k+1/2}(x)\geqslant (1-ch)\rho_h^{k+1/2}(x)
$$
for small $h$.
\end{proof}

Combining Lemma~\ref{lem:maximum_principle_W_step} and Lemma~\ref{lem encadrement}, we obtain at the continuous level

\begin{prop}
\label{prop:maximum principle}
For all $T>0$ there exist constants $M_T,M_T'$ such that for all $t \in [0,T]$,
$$
\|\rho_h(t)\|_{L^1\cap L^\infty} , \|\trho_h(t)\|_{L^1\cap L^\infty} \leqslant M_T
$$
and
$$
\|\rho_h(t)-\trho_h(t)\|_{L^1}\leqslant h M_T'
$$
uniformly in $h\geqslant 0$.
\end{prop}
Note from the second estimate that strong $L^1((0,T)\times \Omega)$ convergence of $\rho_h$ will immediately imply convergence of $\trho_h$ to the same limit.
\begin{proof}
By induction combining \eqref{eq:PMAX_W} and \eqref{encadrement FR_sup}, we obtain, for all $t \in [0,T]$,
$$\|\rho_h(t)\|_{L^\infty} , \|\trho_h(t)\|_{L^\infty} \leqslant C_T, $$
 where $C_T$ is a constant depending on $\| V_1 \|_{L^\infty}$, see \cite[lemma 2]{PT}.
The $L^1$ bound is even easier: since the Wasserstein step is mass preserving, we can integrate \eqref{encadrement FR_sup} in space to get
$$
\|\rho^{k+1}_h\|_{L^1}\leqslant (1+Ch)\|\rho_h^{k+1/2}\|_{L^1}=(1+Ch)\|\rho_h^{k+1}\|_{L^1}.
$$
For $t\leqslant T\Leftrightarrow k\leqslant \lfloor T/h\rfloor$ the $L^{1}$ bounds immediately follow by induction, with $(1+Ch)^{\lfloor T/h\rfloor}\lesssim e^{CT}$.
and we conclude again by induction.

In order to compare now $\rho_h$ and $\trho_h$, we take advantage of the above upper bound to write $\rho^{k+1/2}_h\leqslant M_T$ as long as $kh\leqslant T$.
Taking $c=c(M_T)$ in \eqref{encadrement FR_inf} and combining with \eqref{encadrement FR_sup}, we have
$$
-ch \rho^{k+1/2}_h\leqslant \rho^{k+1/2}_h-\rho^{k+1}_h\leqslant Ch \rho^{k+1/2}_h\qquad \mbox{a.e.}
$$
Integrating in $\Omega$ we conclude that
$$
\|\rho_h(t)-\trho_h(t)\|_1=\|\rho^{k+1}_h-\rho_h^{k+1/2}\|_1\leqslant h\max\{c,C\}\|\rho_h^{k+1/2}\|_1
\leqslant h\max\{c,C\}M_T=hM'_T
$$
and the proof is complete.
\end{proof}

\subsection{Energy dissipation}
Our goal is here to estimate the crossed dissipation along each elementary $\W,\FR$ step.

Testing $\rho=\rho_h^k$ in the first Wasserstein step in \eqref{splitting scheme}, we get as usual
\begin{eqnarray}
\label{sum telescopique MK 1}
\frac{1}{2h}\W^2(\rho_h^{k+1/2}, \rho_h^k) \leqslant \F_1(\rho_h^k)-\F_1(\rho_h^{k+1/2})+\int_\Omega V_1(\rho_h^{k}-\rho_h^{k+1/2}).
\end{eqnarray}
Since $V_1$ is Globally Lipschitz we can first use standard methods from \cite{DFF,L} to control $\int_\Omega V_1(\rho_h^{k}-\rho_h^{k+1/2})$ in terms of $\W^2(\rho_h^{k+1/2},\rho_h^{k})$, and suitably reabsorb in the left-hand side to obtain
\begin{eqnarray}
\label{sum telescopique MK 2}
\frac{1}{4h}\W^2(\rho_h^{k+1/2}, \rho_h^k) \leqslant \F_1(\rho_h^k)-\F_1(\rho_h^{k+1/2})+C_Th.
\end{eqnarray}

The dissipation of $\F_1$ along the Fisher-Rao step is controlled as
\begin{prop}
\label{prop:decr F1}
For all $T>0$ there exists a constant $C_T>0$ such that, for all $k \geqslant 0$ and $k\leq \lfloor T/h\rfloor$,
\begin{equation}
\label{décroissance F1}
\F_1(\rho_h^{k+1}) \leqslant \F_1(\rho_h^{k+1/2}) +C_Th.
\end{equation}
\end{prop}

\begin{proof}
We first treat the case of $F_1(z) = \frac{1}{m_1 -1}z^{m_1}$ with $m_1>1$.
Since $F_1$ is increasing, we use \eqref{encadrement FR_sup} to obtain 
\begin{eqnarray*}
\F_1(\rho_h^{k+1}) - \F_1(\rho_h^{k+1/2}) & \leqslant & \frac{( (1+Ch)^{m_1} -1)}{m_1 -1}\int_\Omega (\rho_h^{k+1/2})^{m_1}\\
&\leqslant & Ch \|\rho^{k+1/2}\|_{L^\infty}^{m_1-1}\,\|\rho^{k+1/2}\|_{L^1},
\end{eqnarray*}
and we conclude from Proposition~\ref{prop:maximum principle}.

In the second case $F_1(z) = z\log(z)-z$, we have
$$
\F_1(\rho_h^{k+1}) = \int_{\{\rho_h^{k+1} \leqslant e^{-1}\}} \rho_h^{k+1}\log(\rho_h^{n+1}) +\int_{\{\rho_h^{k+1} \geqslant e^{-1}\}} \rho_h^{k+1}\log(\rho_h^{k+1}) -\int_\Omega \rho_h^{k+1}.
$$
Note from Proposition~\ref{prop:maximum principle} that the $z$ contribution in $F_1(z)=z\log z-z$ is immediately controlled by $|\int \rho_h^{k+1}-\int\rho^{k+1/2}_h|\leqslant \|\rho^{k+1}_h-\rho_h^{k+1/2}\|_{L^1}\leqslant hM_T'$, so we only have to estimate the $z\log z$ contribution.
Since $z\mapsto z\log z$ is increasing on $\{z\geqslant e^{-1}\}$ and using \eqref{encadrement FR_sup}, the second term in the right hand side becomes

\begin{eqnarray*}
\int_{\{\rho_h^{k+1} \geqslant e^{-1}\}} \rho_h^{k+1}\log(\rho_h^{k+1}) &\leqslant & \int_{\{\rho_h^{k+1} \geqslant e^{-1}\}} (1+Ch)\rho_h^{k+1/2}\log((1+Ch)\rho_h^{k+1/2})\\
&\leqslant & \int_{\{\rho_h^{k+1} \geqslant e^{-1}\}} \rho_h^{k+1/2}\log(\rho_h^{k+1/2})
+ Ch\int_{\{\rho_h^{k+1} \geqslant e^{-1}\}} \rho_h^{k+1/2}\log(\rho_h^{k+1/2}) \\
& & \hspace{1.5cm} + (1+Ch)\int_{\{\rho_h^{k+1}\geqslant e^{-1}\}}\rho_h^{k+1/2}\log(1+Ch)\\
&\leqslant &  \int_{\{\rho_h^{k+1} \geqslant e^{-1}\}} \rho_h^{k+1/2}\log(\rho_h^{k+1/2}) +C_Th,
\end{eqnarray*}
where we used $\|\rho_h^{k+1/2}\|_{L^1} \leqslant M_T$ from Proposition~\ref{prop:maximum principle} as well as $\log(1+Ch) \leqslant Ch$ in the last inequality.
Using the same method with the bound from below \eqref{encadrement FR_inf} on $\{\rho^{k+1}_h\leqslant e^{-1}\}$ (where $z\mapsto z\log z$ is now decreasing), we obtain similarly
$$
\int_{\{\rho_h^{k+1} \leqslant e^{-1}\}} \rho_h^{k+1}\log(\rho_h^{k+1}) \leqslant  \int_{\{\rho_h^{k+1} \leqslant e^{-1}\}} \rho_h^{k+1/2}\log(\rho_h^{k+1/2}) +C_Th.
$$
Combining both inequalities gives
$$
\int_\Omega \rho_h^{k+1}\log(\rho_h^{k+1}) \leqslant \int_\Omega \rho_h^{k+1/2}\log(\rho_h^{k+1/2}) +C_Th
$$
and the proof is complete.

\end{proof}

Summing \eqref{sum telescopique MK 2} and \eqref{décroissance F1} over $k$ we obtain
\begin{equation}
\label{eq:sum telescopic MK}
\frac{1}{2h}\sum_{k=0}^{N-1}\W^2(\rho_h^{k+1/2}, \rho_h^k) \leqslant \F_1(\rho_0)-\F_1(\rho_h^{N})+C_T,
\end{equation}
where $N=\lfloor \frac{T}{h}\rfloor$.\\

In the above estimate we just controlled the dissipation of $\F_1$ along the $\FR$/reaction steps, and the goal is now to similarly estimate the dissipation of $\F_2$ along the Wasserstein step.
Testing $\rho=\rho_h^{k+1/2}$ in the second Fisher-Rao step in \eqref{splitting scheme}, we obtain
\begin{eqnarray}
\label{sum telescopique FR 1}
\frac{1}{2h}\FR_2(\rho_h^{k+1}, \rho_h^{k+1/2}) \leqslant \F_2(\rho_h^{k+1/2})-\F_2(\rho_h^{k+1})+\int_\Omega V_2(\rho_h^{k+1/2}-\rho_h^{k+1}).
\end{eqnarray}

Since we assumed $V_2 \in L^\infty(\Omega)$ and because $\rho_h(t)=\rho^{k+1}_h$ remains close to $\trho_h(t)=\rho_h^{k+1/2}$ in $L^1$ uniformly in $t,h$ by Proposition~\ref{prop:maximum principle}, we immediately control the potential part as
\begin{equation}
 \label{sum telescopique FR 2}
\int_\Omega V_2(\rho_h^{k+1/2}-\rho_h^{k+1}) \leqslant \|V_2\|_\infty  C_Th.
\end{equation}
For the internal energy we argue exactly as in the proof Proposition \ref{prop:decr F1} (for the Porous Media part, since we chose here $F_2(z)=\frac{1}{m_2-1}z^{m_2}$), and obtain
\begin{equation}
 \label{décroissance F2}
\F_2(\rho_h^{k+1/2})-\F_2(\rho_h^{k+1}) \leqslant C_Th.
\end{equation}

Combining \eqref{sum telescopique FR 1}, \eqref{sum telescopique FR 2} and \eqref{décroissance F2}, we immediately deduce that 
\begin{eqnarray}
\label{sum telescopic FR}
\frac{1}{2h}\sum_{k=0}^{N-1}\FR^2(\rho_h^{k+1/2}, \rho_h^{k+1}) \leqslant C_T,
\end{eqnarray}
where $N=\lfloor \frac{T}{h}\rfloor$ as before.\\

Finally, we recover an approximate compactness in time in the form
\begin{prop}
\label{prop:1/2_holder}
There exists a constant $C_T>0$ such that for all $h$ small enough and $k\leqslant N=\lfloor T/h\rfloor$,
\begin{equation}
 \label{estimate sum KFR}
\frac{1}{h}\sum_{k=0}^{N-1} \KFR^2(\rho_h^k,\rho_h^{k+1}) \leqslant 4\F_1(\rho_0) +C_T.
\end{equation}
\end{prop}

\begin{proof}
Adding \eqref{eq:sum telescopic MK} and \eqref{sum telescopic FR} gives
$$
\frac{1}{h}\sum_{k=0}^{N-1} \W^2(\rho_h^k,\rho_h^{k+1/2})+\FR^2(\rho_h^{k+1/2},\rho_h^{k+1}) \leqslant 2\left( \F_1(\rho_0)-\F_1(\rho_h^{N})+C_T\right )+2C_T\leqslant 2\F_1(\rho_0) +C_T,
$$
since in any case $F_1(z)=\frac{1}{m_1-1}z^{m_1}\geqslant 0$ and $F_{1}(z)=z\log z-z\geqslant -1$ is bounded from below on the bounded domain $\Omega$, hence $\F_1(\rho_h^{N})\geqslant -C_{\Omega}$ uniformly.
It then follows from Proposition \ref{prop:comparison_d_W_H} that $\W^2(\rho_h^k,\rho_h^{k+1/2})+\FR^2(\rho_h^{k+1/2},\rho_h^{k+1})\geqslant \frac{1}{2}\KFR^2 \rho_h^k,\rho_h^{k+1}$ in the left-hand side, and the result immediately follows.

\end{proof}

\subsection{Estimates and convergences}

From the total-square distance estimate \eqref{estimate sum KFR} we recover as usual the approximate $\frac{1}{2}$-H\"older estimate
\begin{eqnarray}
\label{estimate holder}
\KFR(\rho_h(t),\rho_h(s))+\KFR(\trho_h(t),\trho_h(s)) \leqslant C_T|t-s+h|^{1/2}
\end{eqnarray}
for all fixed $T>0$ and $t,s\in[0,T]$.
From \eqref{sum telescopic FR} and Proposition~\ref{prop:comparison_d_W_H} we have moreover
\begin{eqnarray}
\label{estimate diff}
\KFR(\rho_h(t),\trho_h(t)) \leqslant \FR(\rho_h(t),\trho_h(t)) \leqslant C\sqrt{h}.
\end{eqnarray}
Using a refined version of Ascoli-Arzel\`a theorem, \cite[prop. 3.3.1]{AGS} and arguing exactly as in \cite[prop. 4.1]{GM}, we see that for all $T>0$ and up to extraction of a discrete subsequence, $\rho_h$ and $\trho_h$ converge uniformly to the same $\KFR$-continuous curve $\rho \in \mathcal{C}^{1/2}([0,T], \M^+_{\KFR})$ as 
$$
\sup_{t\in [0,T]} (\KFR(\rho_h(t),\rho(t))+ \KFR(\trho_h(t),\rho(t)) ) \rightarrow 0.
$$

In order to pass to the limit in the nonlinear terms, we first strengthen this $\KFR$-convergence into a more tractable $L^1$ convergence.
The first step is to retrieve compactness in space:
\begin{prop}
For all $T>0$, $\rho_h$ and $\trho_h$ satisfies 
\begin{equation}
 \label{estimate P1}
\|P_1(\trho_h)\|_{L^2([0,T];H^1(\Omega))} \leqslant C_T.
\end{equation}
\end{prop}

\begin{proof}
From \eqref{opt Wasserstein} and the $L^1\cap L^\infty$ bounds from Proposition~\ref{prop:maximum principle} we see that
\begin{eqnarray*}
 \int_\Omega |\nabla P_1(\rho_h^{k+1/2})|^2
 &\leqslant & \frac{1}{2h^2}\int_\Omega|\nabla \varphi_h^{k+1/2}|^2 (\rho_h^{k+1/2})^2  + \frac{1}{2} \int_\Omega|\nabla V_1|^2(\rho_h^{k+1/2})^2\\
 &\leqslant & \frac{C_T}{2h^2}\int_\Omega|\nabla \varphi_h^{k+1/2}|^2 \rho_h^{k+1/2} + \frac{1}{2} \|\nabla V_1\|_{\infty}^2\int _\Omega(\rho_h^{k+1/2})^2\\
 &\leqslant & C_T\left(\frac{\W ^2(\rho_h^{k+1/2},\rho_h^{k})}{h^2} + 1\right)
\end{eqnarray*}
since $\varphi^{k+1/2}_h$ is the optimal (backward) Kantorovich potential from $\rho_h^{k+1/2}$ to $\rho_h^{k}$.
Multiplying by $h>0$, summing over $k$, and exploiting \eqref{eq:sum telescopic MK} gives
\begin{equation*}
\|P_1(\trho_h)\|^2_{L^2([0,T];H^1(\Omega))}\leqslant \sum _{k=0}^{N-1}h\|P_1(\rho^{k+1/2}_h)\|^2_{H^1}
\leqslant C_T (\F_1(\rho_0)-\F_1(\rho_h^{N})+1)\leqslant C_T,
\end{equation*}
where we used as before $\F_1(\rho_h^N)\geqslant -C_{\Omega}$ in the last inequality.
\end{proof}
We are now in position of proving our main result:
\begin{proof}[Proof of Theorem~\ref{theo:existence differente energies}]
Exploiting \eqref{estimate sum KFR} and \eqref{estimate P1}, we can apply the extension of the Aubin-Lions lemma established by Rossi and Savar\'e in \cite{RS} to obtain that $\trho_h$ converges to $\rho$ strongly in $L^1(Q_T)$ (see \cite{L}).
By diagonal extraction if needed, we can assume that the convergence holds in $L^1(Q_T)$ for all fixed $T>0$.
Then by Proposition~\ref{prop:maximum principle} we have

$$
\|\rho_h - \rho \|_{L^1(Q_T)} \leqslant \|\rho_h - \trho_h \|_{L^1(Q_T}+\|\trho_h - \rho \|_{L^1(Q_T)}\leqslant C_T h +\|\trho_h - \rho \|_{L^1(Q_T)}\to 0
$$
hence $\rho_h\to \rho$ as well.

Moreover, since $P_1(\trho_h)$ is bounded in $L^2((0,T),H^1(\Omega))$ we can assume that $\nabla P_1(\trho_h) \rightharpoonup \nabla P_1(\rho)$ in $L^2((0,T),H^1(\Omega))$ for all $T>0$. 
Exploiting the Euler-Lagrange equations \eqref{opt Wasserstein}\eqref{opt Fisher-Rao} and arguing exactly as in \cite[Theorem 4]{GM}, it is easy to pass to the limit to conclude that
$$
\int_\Omega \rho(t_2)\varphi -\rho(t_1)\varphi=-\int_{t_1}^{t_2}\int_\Omega \Big\{
\nabla P(\rho)\cdot \nabla\varphi 
+\rho \nabla V_1\cdot \nabla\varphi 
-\rho(F'_2(\rho)+V_2)\varphi\Big\}
$$
for all $0<t_1<t_2$ and $\varphi\in \mathcal C^1_b(\Omega)$.
Since $\rho\in\mathcal C([0,T];\M^+_{\KFR})$ takes the initial datum $\rho(0)=\rho_0$ and $\KFR$ metrizes the narrow convergence of measures, this is well-known to be equivalent to our weak formulation in Definition~\ref{def:weak_solutions}, and the proof is complete.
\end{proof}

\begin{rem}
In the above proofs one can check that Theorem \ref{theo:existence differente energies} extends in fact to all $\C ^1$ nonlinearities $F_2$ such that $F_2' \geqslant C$ for some $C \in \R$.
Likewise, we stated and proved our main result in bounded domains for convenience: all the above arguments immediately extend to $\Omega=\R^d$ at least for $F_1(z)=\frac{1}{m_1-1}z^{m_1}\geqslant 0$.
The only place where we actually used the boundedness of $\Omega$ was in the proof of Proposition~\ref{prop:1/2_holder}, when we bounded from below $\mathcal F_1(\rho^N_h)\geqslant -C_\Omega$ in order to retrieve the total-square distance estimate.
When $\Omega=\R^d$ and $F_1(z)=z\log z-z$ a lower bound $\mathcal F_1(\rho^N_h)\geqslant -C_T$ still holds, but the proof requires a tedious control of the second moments $\mathfrak m_2(\rho)=\int_{\R^d}|x|^2\rho$ hence we did not address this technical issue for the sake of brevity.
\end{rem}

\section{Application to systems}
\label{section:systems}
In this section we shall try to illustrate that the previous scheme is very tractable and allows to solve systems of the form

\begin{eqnarray}
\label{eq:KFR-general system}
\left\{\begin{array}{l}
\partial_t \rho_1 = \dive(\rho_1 \nabla(F_1'(\rho_1)+ V_1[\rho_1,\rho_2]))-\rho_1(G_1'(\rho_1) +U_1[\rho_1,\rho_2]),\\
\partial_t \rho_2 = \dive(\rho_2 \nabla(F_2'(\rho_2)+ V_2[\rho_1,\rho_2]))-\rho_2 (G_2'(\rho_2) +U_2[\rho_1,\rho_2]),\\
{\rho_1}_{|t=0}=\rho_{1,0}, \, {\rho_2}_{|t=0}=\rho_{2,0}.
\end{array}\right.
\end{eqnarray}
For simplicity we assume again that $\Omega$ is a smooth, bounded subset of $\Rn$.
Then the system \eqref{eq:KFR-general system} is endowed with Neumann boundary conditions,
$$ \rho_1 \nabla(F_1'(\rho_1)+ V_1[\rho_1,\rho_2]) \cdot \nu =0 \text{  and  } \rho_2 \nabla(F_2'(\rho_2)+ V_2[\rho_1,\rho_2]) \cdot \nu=0 \qquad \text{ on } \mathbb{R}^+ \times \partial \Omega,$$
where $\nu$ is the outward unit normal to $\partial \Omega$.
In system of the form \eqref{eq:KFR-general system}, we allow interactions between densities in the potential terms $V_i[\rho_1,\rho_2]$ and $U_i[\rho_1,\rho_2]$. In the mass-conservative case (without reaction terms), this system has already been studied in \cite{DFF,L,CL1}, using a semi-implicit JKO scheme introduced by Di Francesco and Fagioli, \cite{DFF}. This section combines the splitting scheme introduced in the previous section and semi-implicit schemes both for the Wasserstein JKO step and for the Fisher-Rao JKO step.

\smallskip
For the ease of exposition we keep the same assumptions for $F_i$ and $G_i$ as in the previous section, i.e the diffusion terms $F_i$ satisfy \eqref{assumption Diffusion} and the reaction terms $G_i$ satisfy \eqref{assumption reaction}.
Moreover, since the potentials depend now on the densities $\rho_1$ and $\rho_2$, we need stronger hypotheses:
we assume that $V_i \, :\, L^1(\Omega;\R^+)^2 \rightarrow \mathcal{C}^{1}(\Omega)$ are continuous and verify, uniformly in $\rho_1,\rho_2 \in L^1(\Omega;\R^+)$,
\begin{eqnarray}
\label{assumption V system}
\begin{array}{c}
 \|V_i[\rho_1,\rho_2]\|_{W^{1,\infty(\Omega)}} \leqslant K (1+ \|\rho_1\|_{L^1(\Omega)} +\|\rho_2\|_{L^1(\Omega)}),\\
 \|\nabla( V_i[\rho_1,\rho_2]) - \nabla (V_i[\mu_1,\mu_2]) \|_{L^\infty(\Omega)} \leqslant K(\|\rho_1 -\mu_1\|_{L^1(\Omega)} +\|\rho_2 -\mu_2\|_{L^1(\Omega)}).
\end{array}
\end{eqnarray}

The interacting potentials we have in mind are of the form $V_i[\rho_1,\rho_2]= K_{i,1} \ast \rho_1 + K_{i,2}\ast \rho_2$, where $K_{i,1},K_{i,2} \in W^{1,\infty}(\Omega)$ and then $V_i$ satisfies \eqref{assumption V system}. 
For the reaction, we assume that the potentials $U_i$ are continuous from $L^1(\Omega)_+^2$ to $L^1$ with moreover
 
\begin{equation}
\label{assumption U}
 U_i[\rho_1,\rho_2] \geqslant  -K, \qquad \forall\,\rho_1,\rho_2 \in L^1(\Omega;{\R}^+)
\end{equation}
for some $K\in\R$, and
\begin{equation}
 \label{assumption U max}
 \|  U_i[\rho_1,\rho_2] \|_{L^\infty(\Omega)} \leqslant  K_M,
 \qquad
 \forall \|\rho_1\|_{L^1(\Omega)},\|\rho_2\|_{L^1(\Omega)} \leqslant M
\end{equation}
for some nondecreasimg function $K_M\geqslant 0$ of $M$.
The examples we have in mind are of the form 
$$ U_1[\rho_1,\rho_2]=C_1\frac{\rho_2}{1+\rho_1}, 
\quad 
U_2[\rho_1,\rho_2]=-C_2\frac{\rho_1}{1+\rho_1}
$$
for some constants $C_i\geq 0$, or nonlocal reactions
$$
U_i[\rho_1, \rho_2](x) = \int_\Omega K_{i,1}(x,y) \rho_1(y) \,dy + \int_\Omega K_{i,2}(x,y)\rho_2(y) \,dy
$$
for some nonnegative kernels $K_{i,j}\in L^1\cap L^\infty$.
Such reaction models appear for example in biological adaptive dynamics \cite{Per}.

\begin{de}
\label{def:weak_solutions system}
We say that $(\rho_1,\rho_2)\, : \, \R^+ \rightarrow L^1_+ \cap L^\infty_+(\Omega)$ is a weak solution of \eqref{eq:KFR-general system} if, for $i\in \{1,2\}$ and all $T< +\infty$, the pressure $P_i(\rho_i):=\rho_i F_i'(\rho_i)-F_i(\rho_i)$ satisfies $\nabla P_i(\rho_i) \in L^2([0,T] \times \Omega)$, and
\begin{multline}
 \int_0^{+\infty}\left( \int_\Omega (\rho \partial_t \phi_i -\rho_i\nabla V_i[\rho_1,\rho_2] \cdot \nabla \phi_i  - \nabla P_i(\rho_i) \cdot \nabla \phi_i - \rho_i (G_i'(\rho_i) +U_i[\rho_1,\rho_2]) \phi_i) \,dx \right)\, dt\\
  = -\int_\Omega \phi_i(0,x) \rho_{i,0}(x) \,dx,
\end{multline}
for all $\phi_i \in \C^\infty_c([0,+\infty) \times \Rn)$.
\end{de}
Then, the following result holds,
\begin{thm}
\label{theo:existence system}
Assume that $\rho_{1,0},\rho_{2,0} \in L^1\cap L^\infty_+(\Omega)$ and that $V_i,U_i$ satisfy \eqref{assumption V system}\eqref{assumption U}\eqref{assumption U max}.
Then \eqref{eq:KFR-general system} admits at least one weak solution.
 \end{thm}

Note that this result can be easily adapted to systems with an arbitrary number of species $N \geqslant 2$, coupled by nonlocal terms $V_i[\rho_1,\dots,\rho_N]$ and $U_i[\rho_1,\dots,\rho_N]$.
\begin{rem}
A refined analysis shows that our approach would allow to handle systems of the form
$$\left\{\begin{array}{l}
\partial_t \rho_1 - \dive(\rho_1 \nabla(F_1'(\rho_1)+ V_1))=-\rho_1 h_1(\rho_1,\rho_2),\\
\partial_t \rho_2 - \dive(\rho_2 \nabla(F_2'(\rho_2)+ V_2))=+\rho_2 h_2(\rho_1),
\end{array}\right.$$
where $h_1$ is a nonnegative continuous function and $h_2$ is a continuous functions.

Indeed since $h_1\geq 0$ the reaction term is the first equation is nonpositive, hence
$
\|\rho_1(t) \|_{L^\infty(\Omega)} \leqslant C_T$.
Then it follows that $-h_2(\rho_1)$ satisfies assumptions \eqref{assumption U} and \eqref{assumption U max}.
A classical example is $ h_2(\rho_1)= \rho_1^\alpha$ and $h_1(\rho_1,\rho_2)= \rho_1^{\alpha-1} \rho_2$, where $\alpha \geqslant 1$, see for example \cite{MPsurvey} for more discussions.
\end{rem}

As already mentioned, the proof of theorem \ref{theo:existence system} is based on a semi-implicit splitting scheme.
More precisely, we construct four sequences $\rho_{1,h}^{k+1/2},\rho_{1,h}^{k+1},\rho_{2,h}^{k+1/2},\rho_{2,h}^{k+1}$ defined recursively as
\begin{eqnarray}
 \label{eq:splitting scheme Wass-FR syst}
\left\{\begin{array}{l}
\rho_{i,h}^{k+1/2} \in \argmin\limits_{\rho \in \M^+, |\rho|=|\rho_{i,h}^k|} \left\{ \frac{1}{2h} \W^2(\rho, \rho_{i,h}^k) + \F_i(\rho) +\V_i(\rho | \rho_{1,h}^k, \rho_{2,h}^k) \right\}\\
\\
\rho_{i,h}^{k+1} \in \argmin\limits_{\rho \in \M^+} \left\{ \frac{1}{2h} \FR^2(\rho, \rho_{i,h}^{k+1/2}) +  \G_i(\rho) +\mathcal{U}_i(\rho | \rho_{1,h}^k, \rho_{2,h}^k) \right\}
\end{array}\right.,
\end{eqnarray} 
where the fully implicit terms
$$ 
\F_i(\rho):= \left\{ \begin{array}{ll}
\int_\Omega F_i(\rho) & \text{    if } \rho \ll \mathcal{L}_{|\Omega}\\
+\infty & \text{    otherwise }
\end{array}\right.
\quad\text{and}\quad
\G_i(\rho):= \left\{ \begin{array}{ll}
\int_\Omega G_i(\rho) & \text{    if } \rho \ll \mathcal{L}_{|\Omega}\\
+\infty & \text{    otherwise }
\end{array}\right.,
$$
and the semi-implicit terms
$$
\V_i(\rho | \mu_1, \mu_2):= \int_\Omega V_i[\mu_1,\mu_2] \rho 
\quad\text{and}\quad
\mathcal{U}_i(\rho | \mu_1, \mu_2):= \int_\Omega U_i[\mu_1,\mu_2] \rho.
$$

In the previous section, the proof of theorem \ref{theo:existence differente energies} for scalar equations strongly leveraged the uniform $L^\infty(\Omega)$-bounds on the discrete solutions.
Here an additional difficulty arises due to the nonlocal terms $\nabla V_i[\rho_1,\rho_2]$ and $U_i[\rho_1,\rho_2]$, which are a priori not uniformly bounded in $L^\infty(\Omega)$.
Using assumption \eqref{assumption U} we will first obtain a uniform $L^1(\Omega)$-bound on $\rho_1,\rho_2$, and then extend proposition \ref{prop:maximum principle} to the system \eqref{eq:KFR-general system}.
This in turn will give a uniform $W^{1,\infty}$ control on $V_i[\rho_1,\rho_2]$ and $L^\infty$ control on $U_i[\rho_1,\rho_2]$ through our assumptions \eqref{assumption V system}-\eqref{assumption U}-\eqref{assumption U max}, which will finally allow to argue as in the previous section and give $L^\infty$ control on $\rho_1,\rho_2$.

Numerical simulations for a diffusive prey-predator system are presented at the end of this section.

\subsection{Properties of discrete solutions}
Arguing as in the case of one equation, the optimality conditions for the Wasserstein step and for the Fisher-Rao step first give

\begin{lem}
\label{lem:L1 bound}
For all $k \geqslant 0$ and $i\in \{1,2\}$, we have
\begin{equation}
\label{eq:PMAX_W system}
 \|\rho_{i,h}^{k+1/2} \|_{L^1} = \|\rho_{i,h}^k\|_{L^1}.
\end{equation}
Moreover, there exists $C_i\equiv C(U_i) >0$ (uniform in $k$) such that
\begin{equation}
\label{encadrement FR_sup system}
 \rho_{i,h}^{k+1}(x) \leqslant (1+ C_i h) \rho_{i,h}^{k+1/2}(x) \qquad a.e.
\end{equation}

\end{lem}

\begin{proof}
The first part is simply the mass conservation in the Wasserstein step, and the second part follows the lines of the proof of \eqref{encadrement FR_sup} in Lemma \ref{lem encadrement} using assumption \eqref{assumption U}.
\end{proof} 

As a direct consequence we have uniform control on the $L^1$-norms:

\begin{lem}
\label{lem:L1_bound_systems}
For all $T>0$ there exist constants $C_T,C_T'>0$ such that, for all $t\in [0,T]$,
$$ \| \rho_{i,h}(t)\|_{L^1}, \| \trho_{i,h}(t)\|_{L^1} \leqslant C_T$$
and
\begin{equation}
\label{encadrement potential V system}
\| V_i[\rho_{1,h}(t), \rho_{2,h}(t)] \|_{W^{1,\infty}},\| V_i[\trho_{1,h}(t), \trho_{2,h}(t)] \|_{W^{1,\infty}}  \leqslant C'_T.
\end{equation}
\end{lem}
\begin{proof}
Integrating \eqref{encadrement FR_sup system} and iterating with \eqref{eq:PMAX_W system}, we obtain for all $t\leqslant T$ and $k\leqslant \lfloor T/h\rfloor$
$$
\| \rho_{i,h}^{k+1} \|_{L^1} \leqslant (1+C_ih)\| \rho_{i,h}^{k} \|_{L^1}  \leqslant (1+C_ih)^k\| \rho_{i,0}\|_{L^1} \leqslant e^{C_iT}\| \rho_{i,0}\|_{L^1}
.$$
Then \eqref{encadrement potential V system} follows from our assumption \eqref{assumption V system} on the interactions.
\end{proof}

Combining \eqref{encadrement FR_sup system} and \eqref{encadrement potential V system}, we deduce
\begin{prop}
\label{prop:encadrment L1}
For all $T>0$, there exists $M_T$ such that for all $t\in [0,T]$,
$$
\| \rho_{i,h}(t)\|_{L^\infty}, \| \trho_{i,h}(t)\|_{L^\infty} \leqslant M_T.
$$
Then, there exists $c_i \equiv c(M_T,U_i) \geq 0$, such that, for all $k \leqslant \lfloor T/h\rfloor$ and $h \leqslant h_0(U_1,U_2)$,
$$ 
(1-c_ih) \rho_{i,h}^{k+1/2}   \leqslant \rho_{i,h}^{k+1}.
$$
In particular, there exist $M_T'>0$ such that for all $t\in [0,T]$,
$$
\| \rho_{i,h}(t) - \trho_{i,h}(t) \|_{L^1} \leqslant hM_T'.
$$
\end{prop}

\begin{proof}
The first $L^\infty$ estimate can be found in \cite[Lemma 2]{PT}, and the rest of our statement can be proved exactly as in Lemma \ref{lem encadrement} and Proposition \ref{prop:maximum principle}.
\end{proof}

\subsection{Estimates and convergences}

Since we proved that $V_1[\rho_{1,h},\rho_{2,h}]$ and $V_2[\rho_{1,h},\rho_{2,h}]$ are bounded in $L^\infty([0,T],W^{1,\infty}(\Omega))$, we can argue exactly as in the previous section for the Wasserstein step and obtain
\begin{eqnarray}
\label{eq:sum telescopic MK system}
\frac{1}{4h}\W^2(\rho_{i,h}^{k+1/2}, \rho_{i,h}^k) \leqslant \F_i(\rho_{i,h}^k)-\F_i(\rho_{i,h}^{k+1/2})+C_Th,
\end{eqnarray}
see \eqref{sum telescopique MK 1}-\eqref{sum telescopique MK 2} for details.
Since $\trho_{1,h}$ and $\trho_{2,h}$ are uniformly bounded in $L^1(\Omega)$ (Lemma~\ref{lem:L1_bound_systems}), our assumption \eqref{assumption U max} ensures that $U_1[\rho_{1,h}^{k+1/2},\rho_{2,h}^{k+1/2}]$ and $U_2[\rho_{1,h}^{k+1/2},\rho_{2,h}^{k+1/2}]$ are uniformly bounded in $L^\infty(\Omega)$.
Proposition \ref{prop:encadrment L1} then allows to argue exactly as in \eqref{sum telescopique FR 1}-\eqref{sum telescopique FR 2}-\eqref{décroissance F2} for the Fisher-Rao step, and we get
\begin{eqnarray}
\label{eq:sum telescopic FR system}
\frac{1}{2h}\FR^2(\rho_h^{k+1}, \rho_h^{k+1/2}) \leqslant \G_i(\rho_{i,h}^{k+1/2})-\G_i(\rho_{i,h}^{k+1}) +C_Th.
\end{eqnarray}

The dissipation of $\F_i$ along the Fisher-Rao step is obtained in the same way as Proposition \ref{prop:decr F1} and we omit the details:
\begin{prop}
For all $T>0$ and $i\in\{1,2\}$, there exist constants $C_T,C'_T>0$ such that, for all $k\geqslant 0$ with $hk \leqslant T$, 
$$\begin{array}{c}
\F_i(\rho_{i,h}^{k+1}) \leqslant \F_i(\rho_{i,h}^{k+1/2}) +C_Th,\\
\G_i(\rho_{i,h}^{k+1/2}) \leqslant \G_i(\rho_{i,h}^{k+1}) +C'_Th.
\end{array}$$
\end{prop}
From \eqref{eq:sum telescopic MK system} and \eqref{eq:sum telescopic FR system} this immediately gives a telescopic sum
$$
\frac{1}{2h}\left(\W^2(\rho_{i,h}^{k}, \rho_{i,h}^{k+1/2})+ \FR^2(\rho_h^{k+1/2}, \rho_h^{k})\right)
\leqslant 2[\F_i(\rho_{i,h}^k)-\F_i(\rho_{i,h}^{k+1})] + C_Th
$$
which in turn yields an approximate $\frac 12$-H\"older estimate (with respect to the $\KFR$ distance) as in Proposition \ref{prop:1/2_holder}.
The rest of the proof of Theorem \ref{theo:existence system} is then identical to section \ref{sec:KFRsplitting} and we omit the details.

\subsection{Numerical application: prey-predator systems}
Our constructive scheme can be implemented numerically, by simply discretizing \eqref{eq:splitting scheme Wass-FR syst} in space.
We use the augmented Lagrangian method ALG-JKO from \cite{BCL} to solve the Wasserstein step, and the Fisher-Rao step is just a convex pointwise minimization problem.
Indeed, it is known \cite{GM,LMS} that $\FR^2(\rho,\mu)=4\|\sqrt{\rho}-\sqrt{\mu}\|^2_{L^2}$, hence the Fisher-Rao step in \eqref{eq:splitting scheme Wass-FR syst} is a mere convex pointwise minimization problem of the form: for all $x \in \Omega$ (and omitting all indexes $\rho_{i,h}$),
\begin{equation*}
\label{eq:pointwise_convex_FR}
\rho^{k+1}(x) = \argmin\limits_{\rho \geq 0} \left\{ 4\left| \sqrt{\rho} - \sqrt{\rho^{k+1/2}(x)} \right|^2 + 2h F(\rho)\right\}.
\end{equation*}
This is easily solved using any simple Newton procedure.

Figure \eqref{alg2-figure prey-predator} shows the numerical solution of the following diffusive prey-predator system
$$
\left\{\begin{array}{l}
\partial_t \rho_1 - \Delta \rho_1 - \dive(\rho_1 \nabla V_1[\rho_1,\rho_2])=A\rho_1\left( 1-\rho_1 \right) - B\frac{\rho_1 \rho_2}{1+\rho_1},\\
\partial_t \rho_2 - \Delta \rho_2 - \dive(\rho_2 \nabla V_2[\rho_1,\rho_2])=\frac{B\rho_1 \rho_2}{1+\rho_1}-C\rho_2,
\end{array}\right..
$$ 
Here the $\rho_1$ species are preys and $\rho_2$ are predators, see for example \cite{M}, the parameters $A=10,C=5,B=70$, and the interactions are chosen as
$$
V_{1}[\rho_1,\rho_2]=|x|^2 \ast \rho_1 - |x|^2 \ast \rho_2,
\quad
V_2[\rho_1,\rho_2]=|x|^2 \ast \rho_1 + |x|^2 \ast \rho_2.
$$
In \eqref{eq:KFR-general system} this corresponds to
$$
G_1(\rho_1) =A\frac{\rho_1^2}{2}, \quad G_2(\rho_2) =0, 
\quad 
U_1[\rho_1,\rho_2]=\frac{B\rho_2}{1+\rho_1}-A, 
\quad 
U_2[\rho_1,\rho_2]=-\frac{B\rho_1}{1+\rho_1}+C.
$$
Of course, $U_1$ and $U_2$ satisfy assumptions \eqref{assumption U} and \eqref{assumption U max}, and then Theorem \ref{theo:existence system} gives a solution of the prey-predator system.
As before, we shall disregard the uniqueness issue for the sake of simplicity.
Figure \eqref{figure masse prey-predator} depicts the mass evolution of the prey and predator species: we observe the usual oscillations in time with phase opposition, a characteristic behaviour for Lotka-Volterra types of systems.

\begin{figure}[h!]

\begin{tabular}{c@{\hspace{0mm}}c@{\hspace{1mm}}c@{\hspace{1mm}}c@{\hspace{1mm}}c@{\hspace{1mm}}c@{\hspace{1mm}}c@{\hspace{1mm}}c@{\hspace{10mm}}}

\centering
$t=0$ & $t=0.15$ & $t=0.35$ & $t=0.5$ &$t=0.65$ & $t=0.85$ & $t=1$\\

\includegraphics[ scale=0.15]{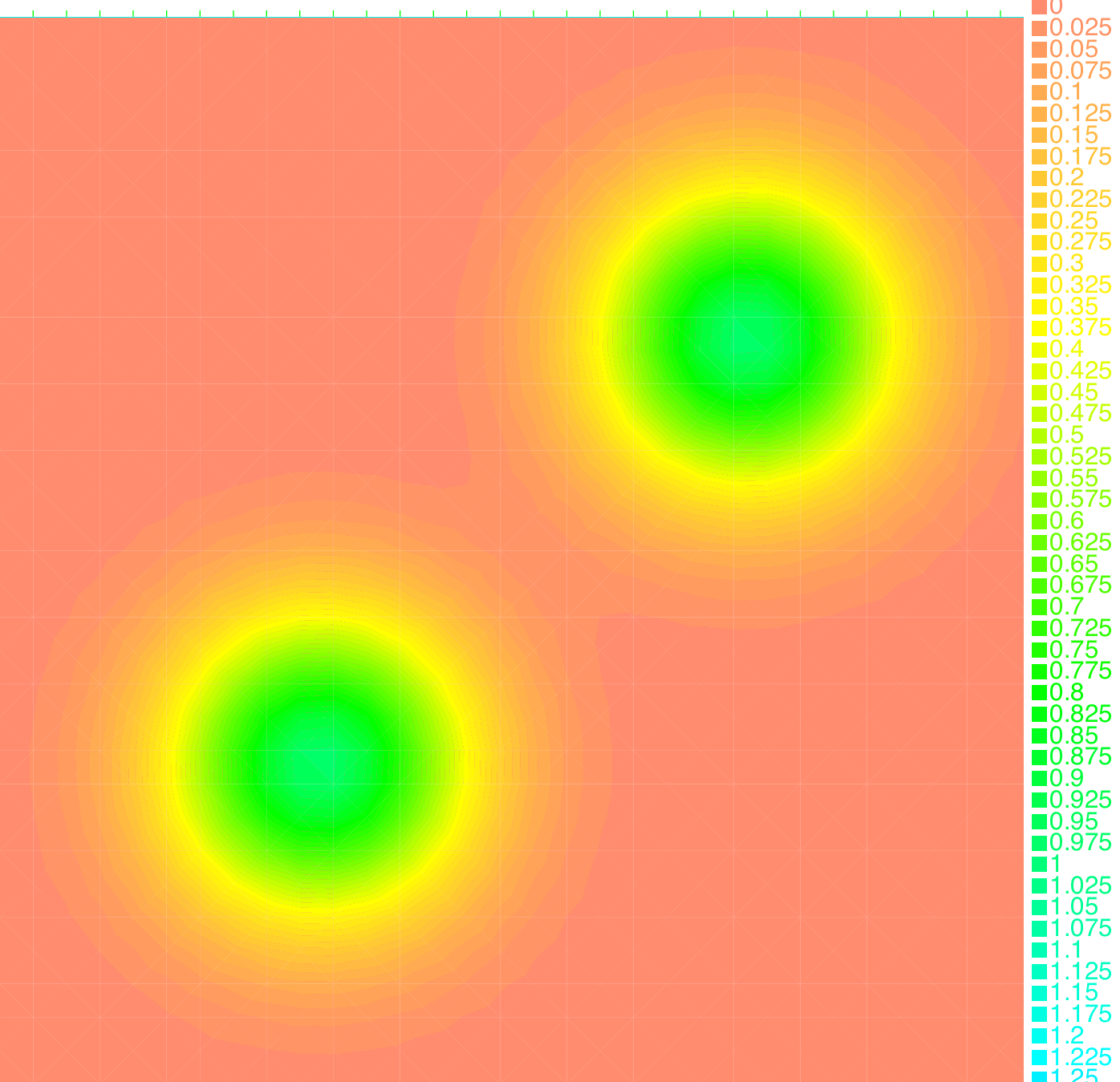}&
\includegraphics[ scale=0.15]{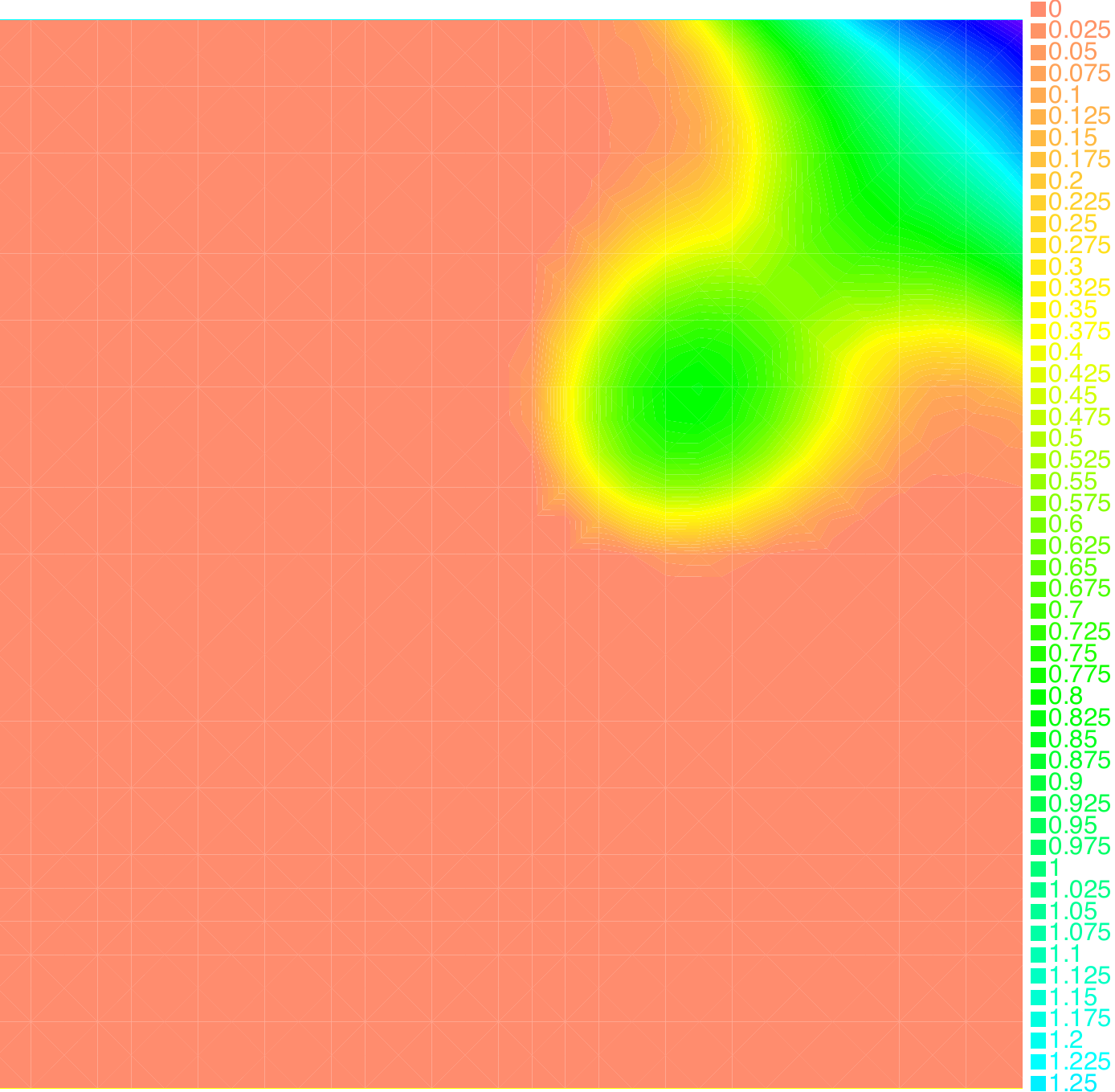}&
\includegraphics[ scale=0.15]{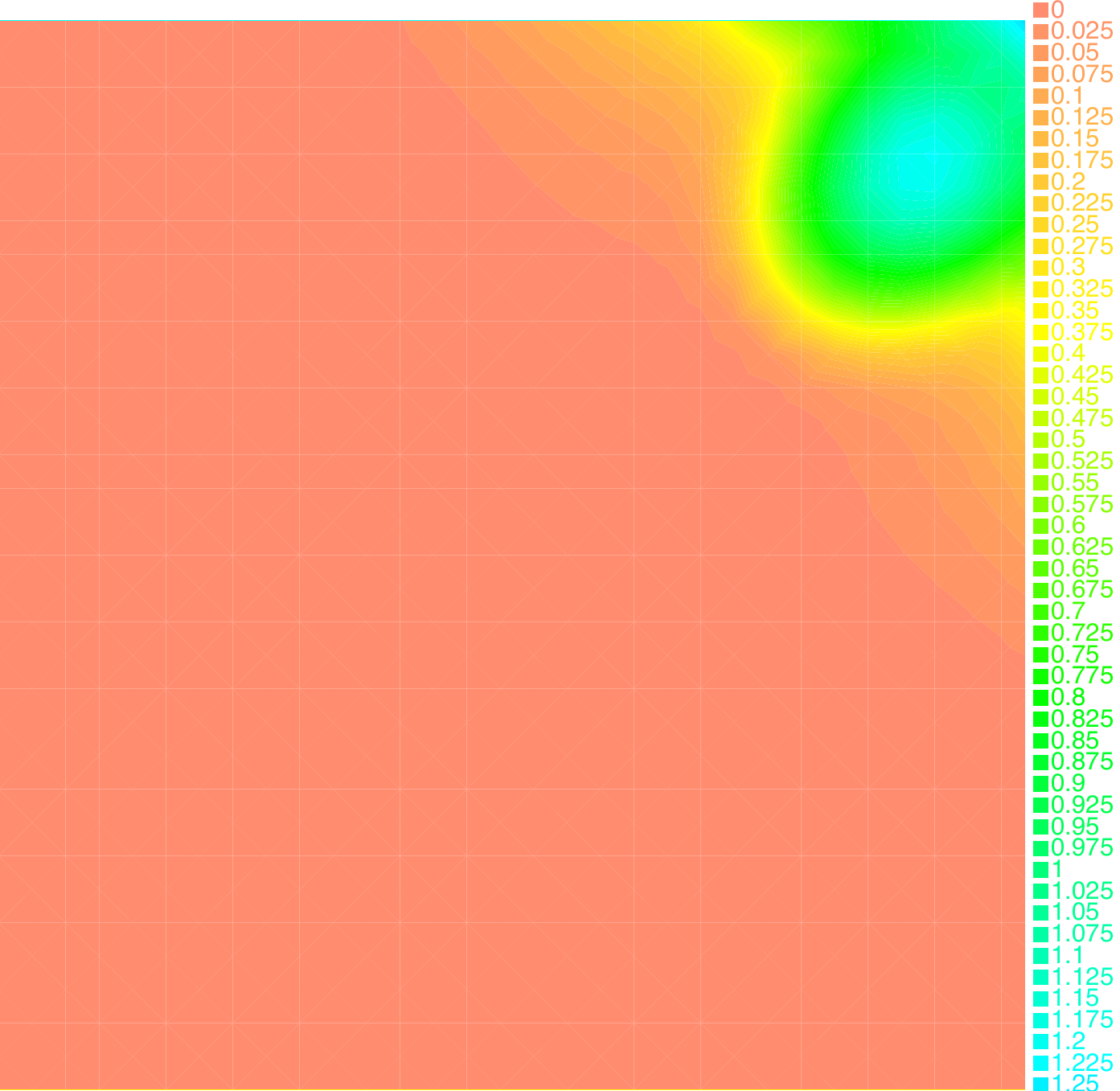}&
\includegraphics[ scale=0.15]{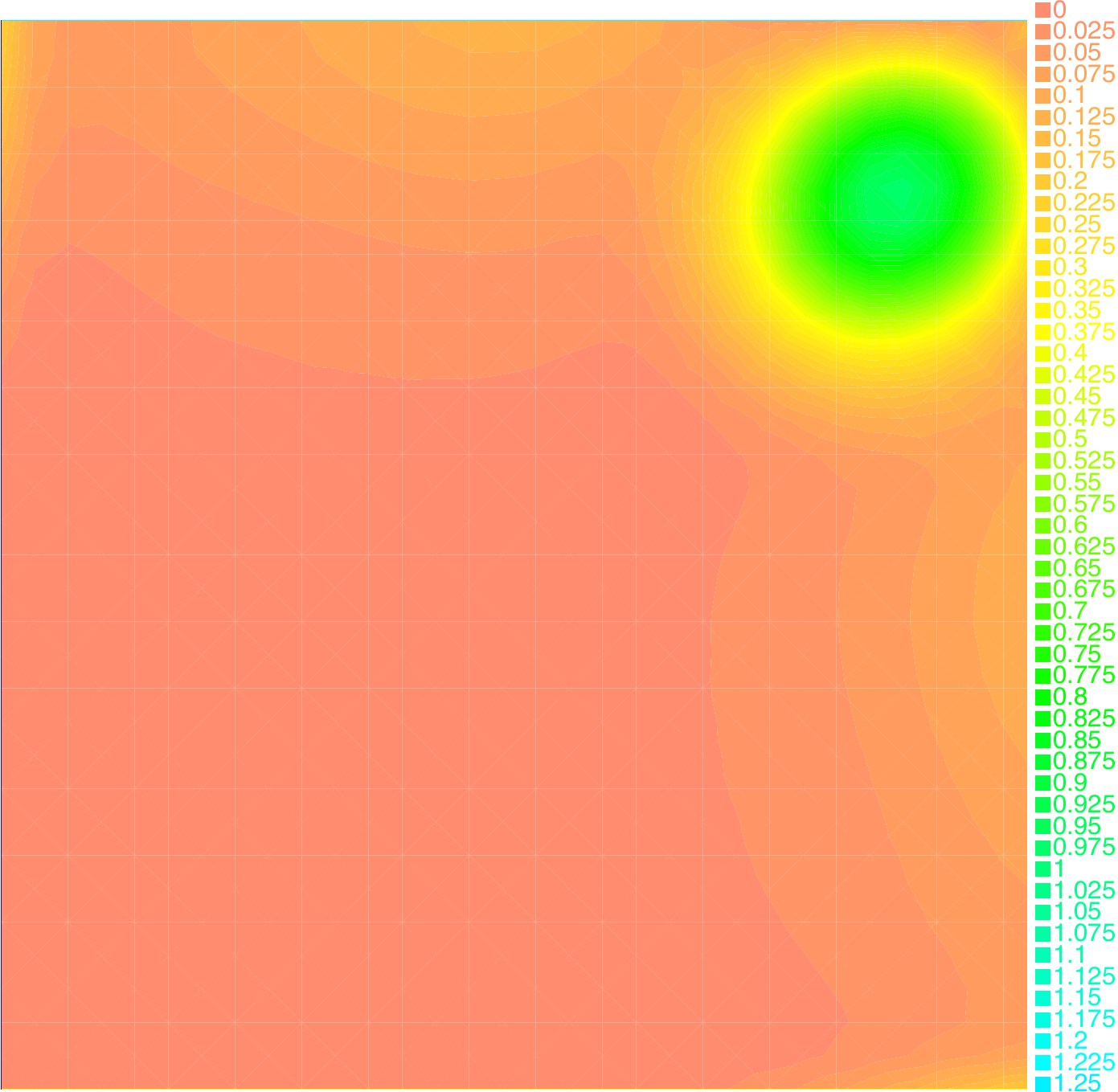}&
\includegraphics[ scale=0.15]{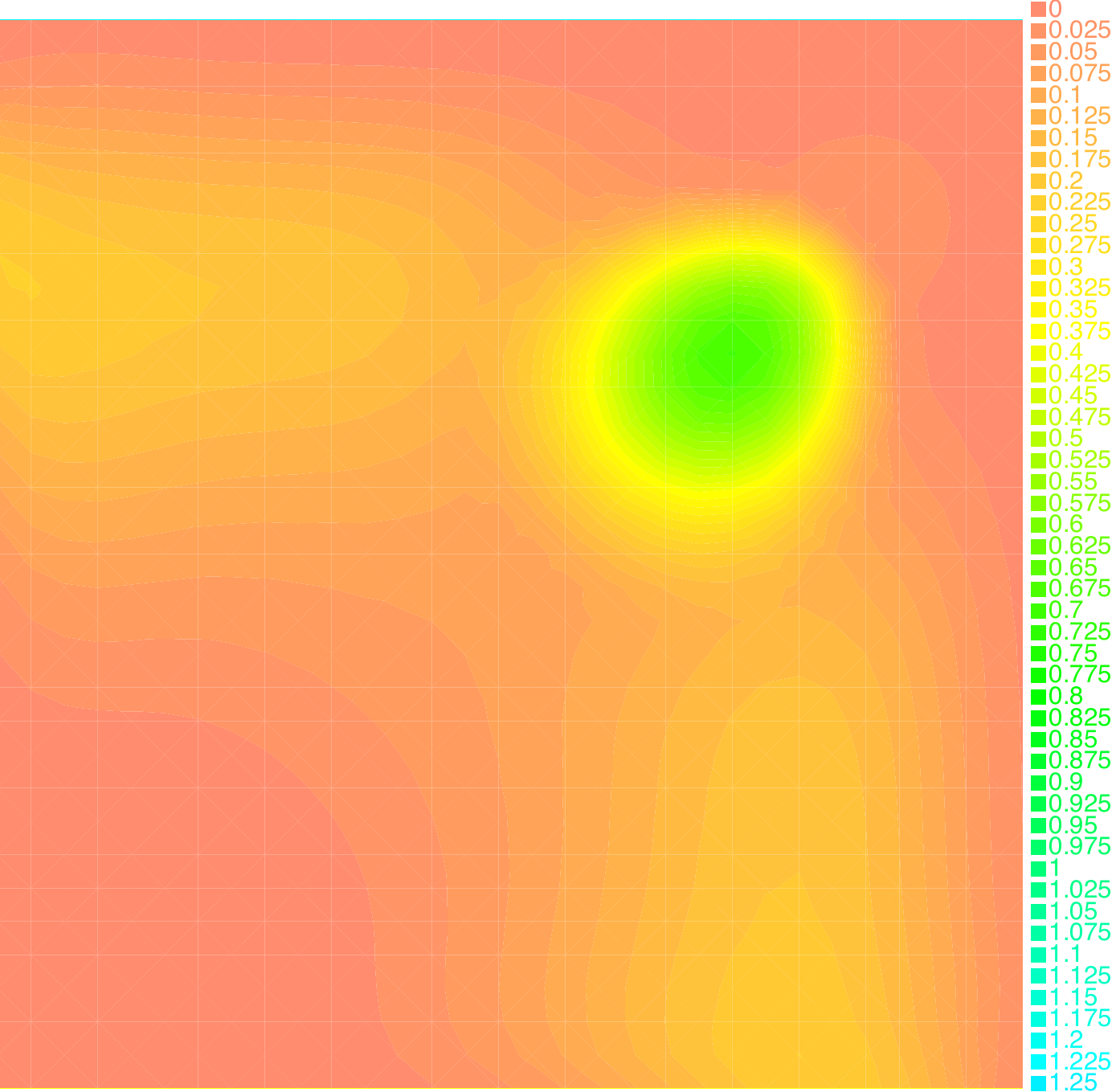}&
\includegraphics[ scale=0.15]{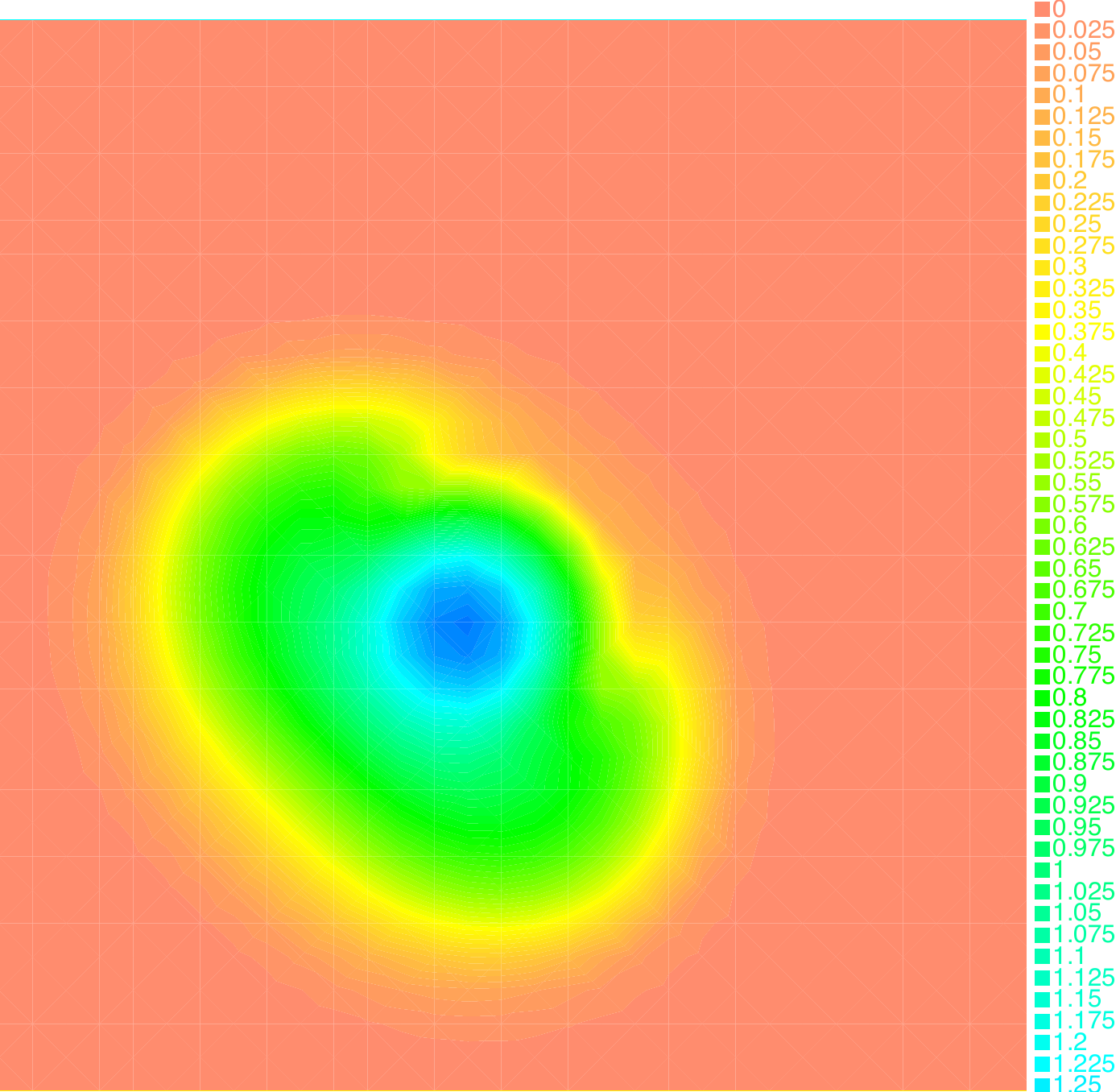}&
\includegraphics[ scale=0.15]{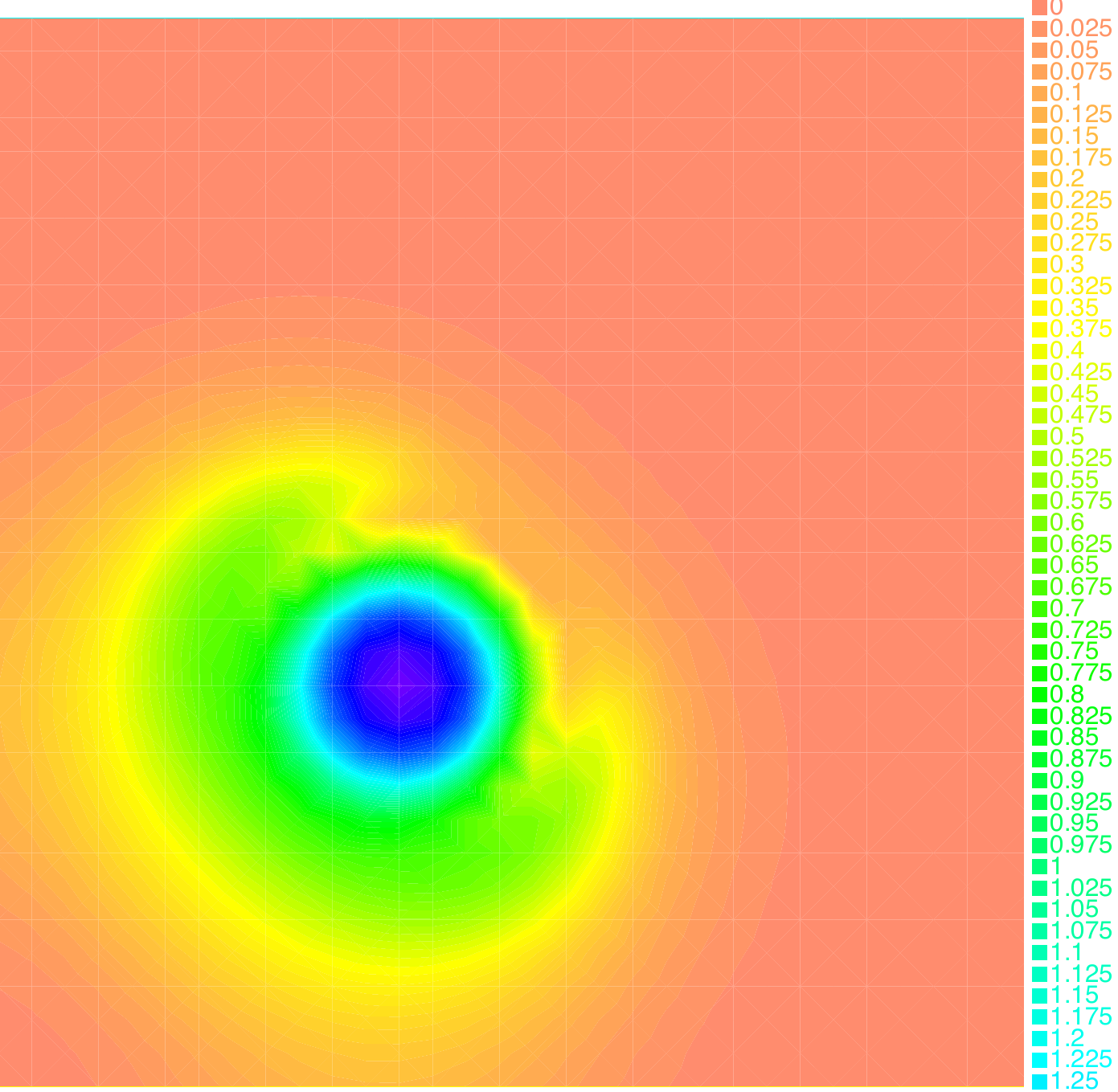}\\

\includegraphics[ scale=0.15]{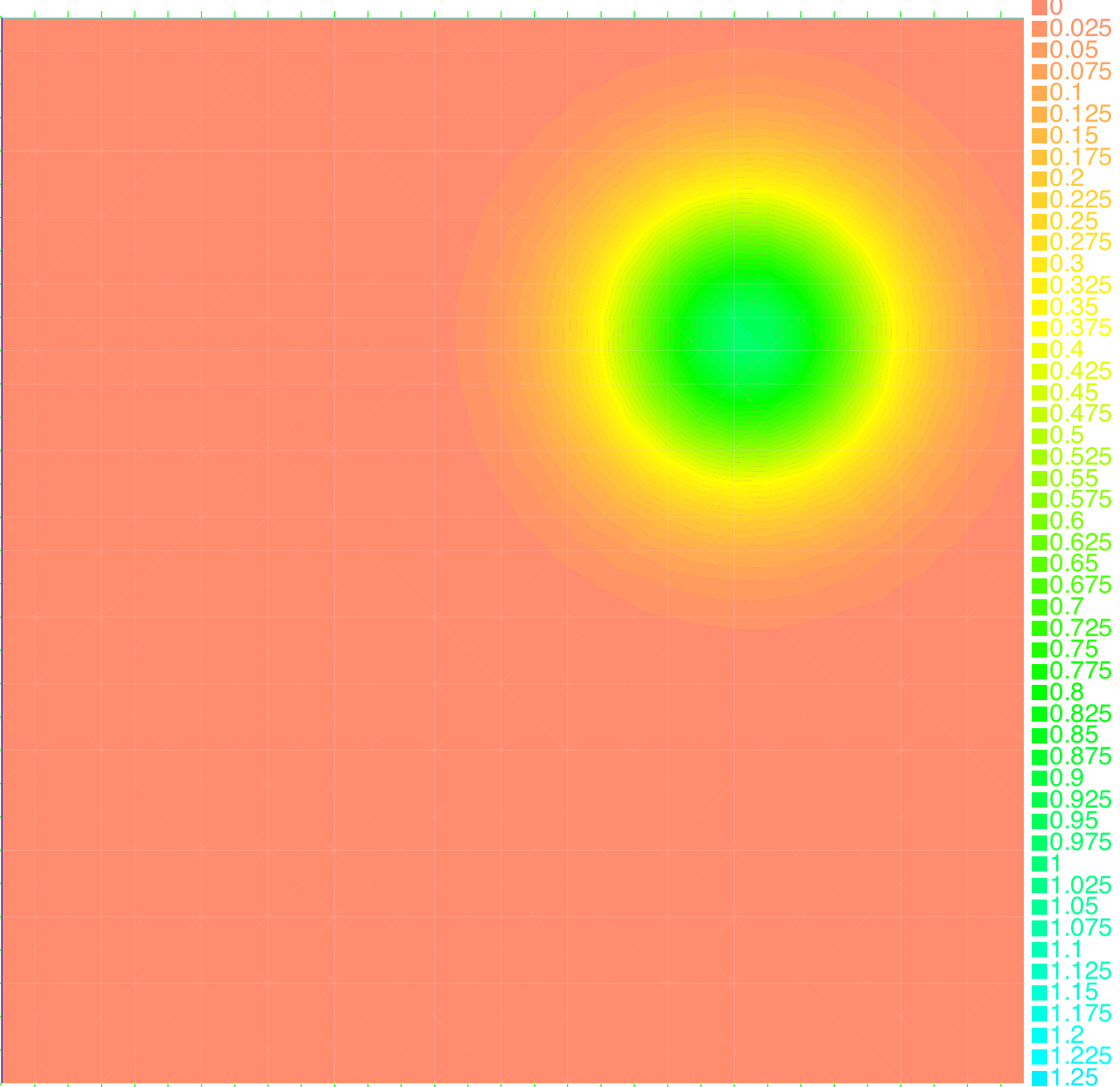}&
\includegraphics[ scale=0.15]{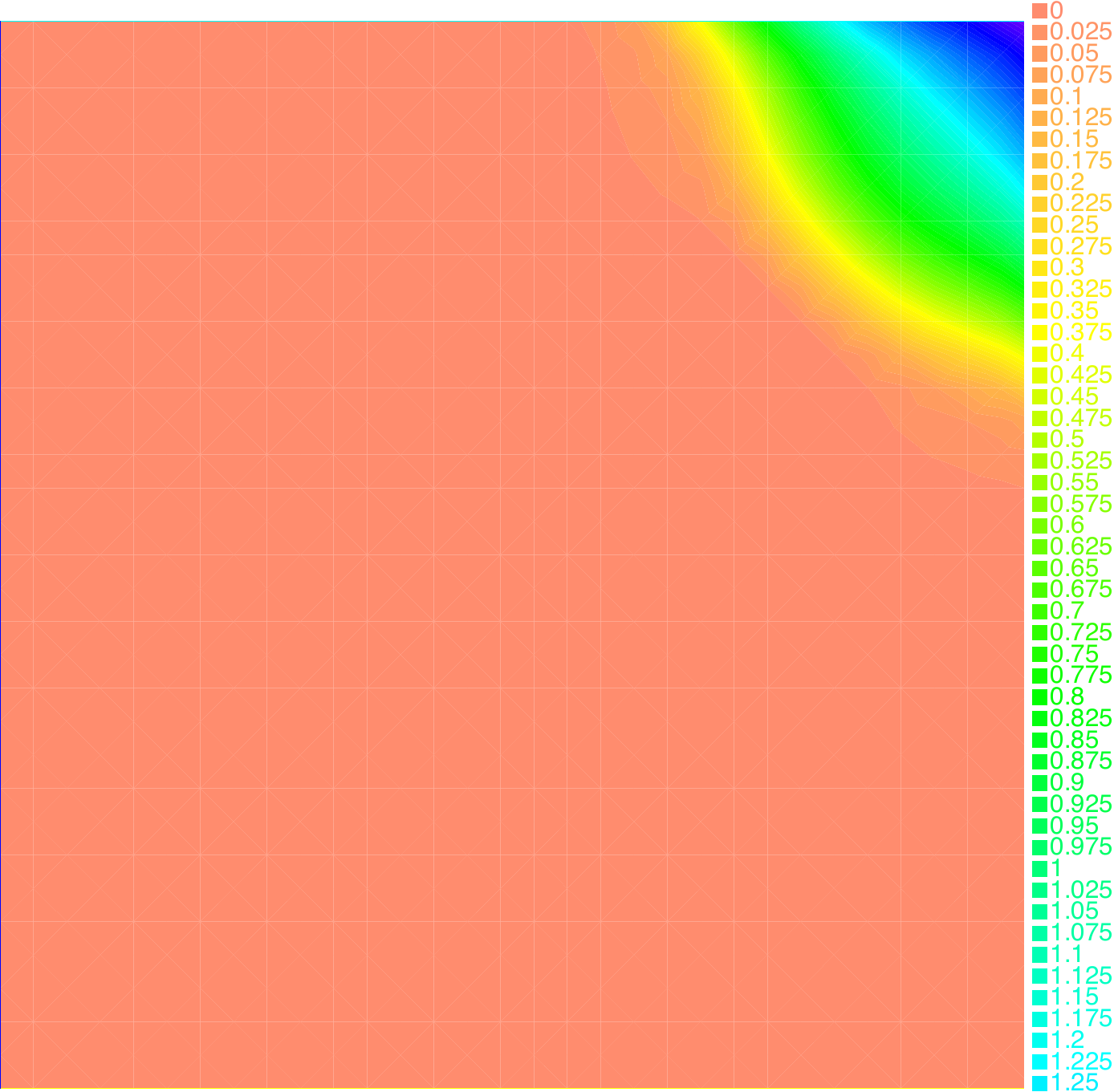}&
\includegraphics[ scale=0.15]{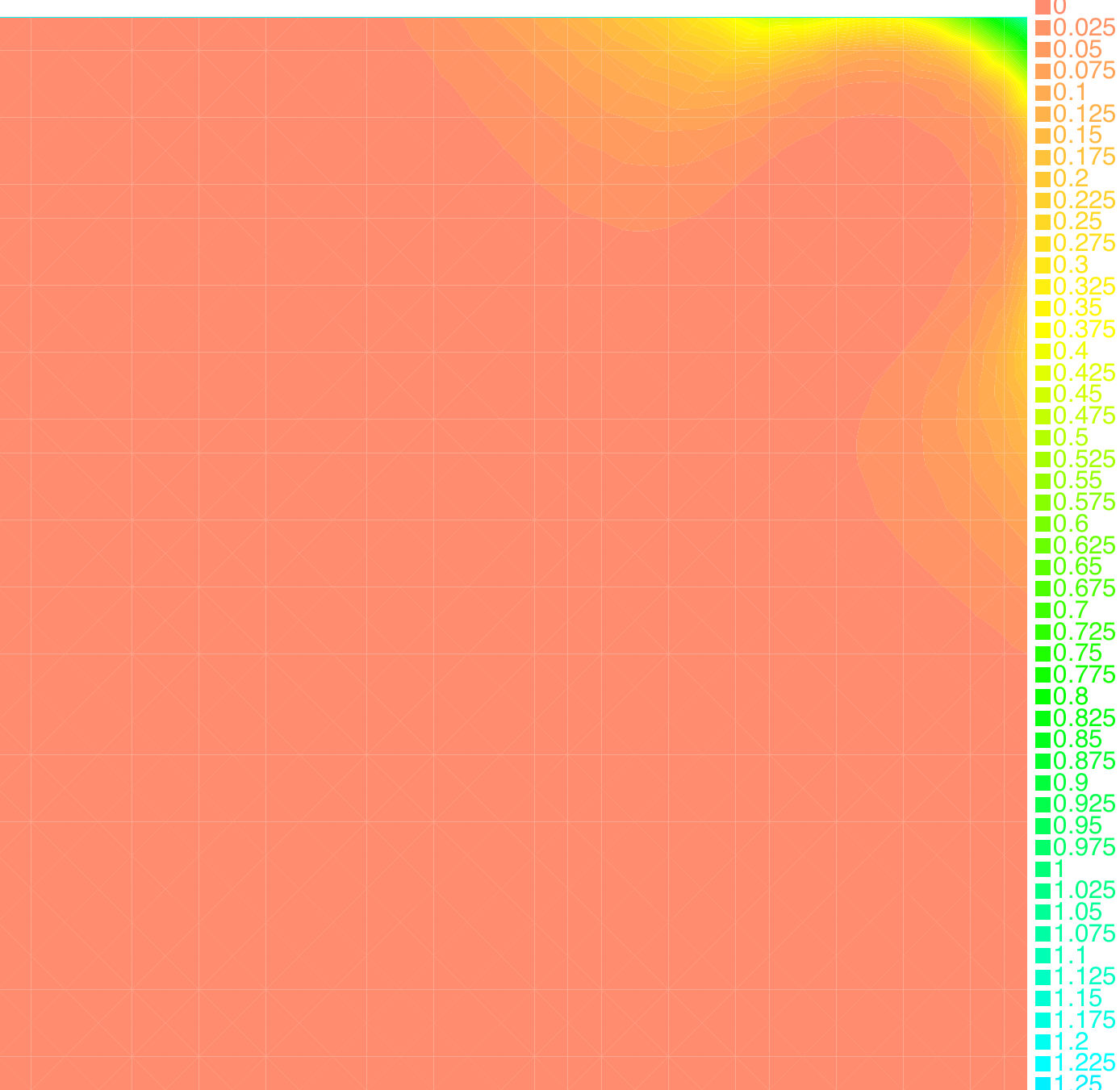}&
\includegraphics[ scale=0.15]{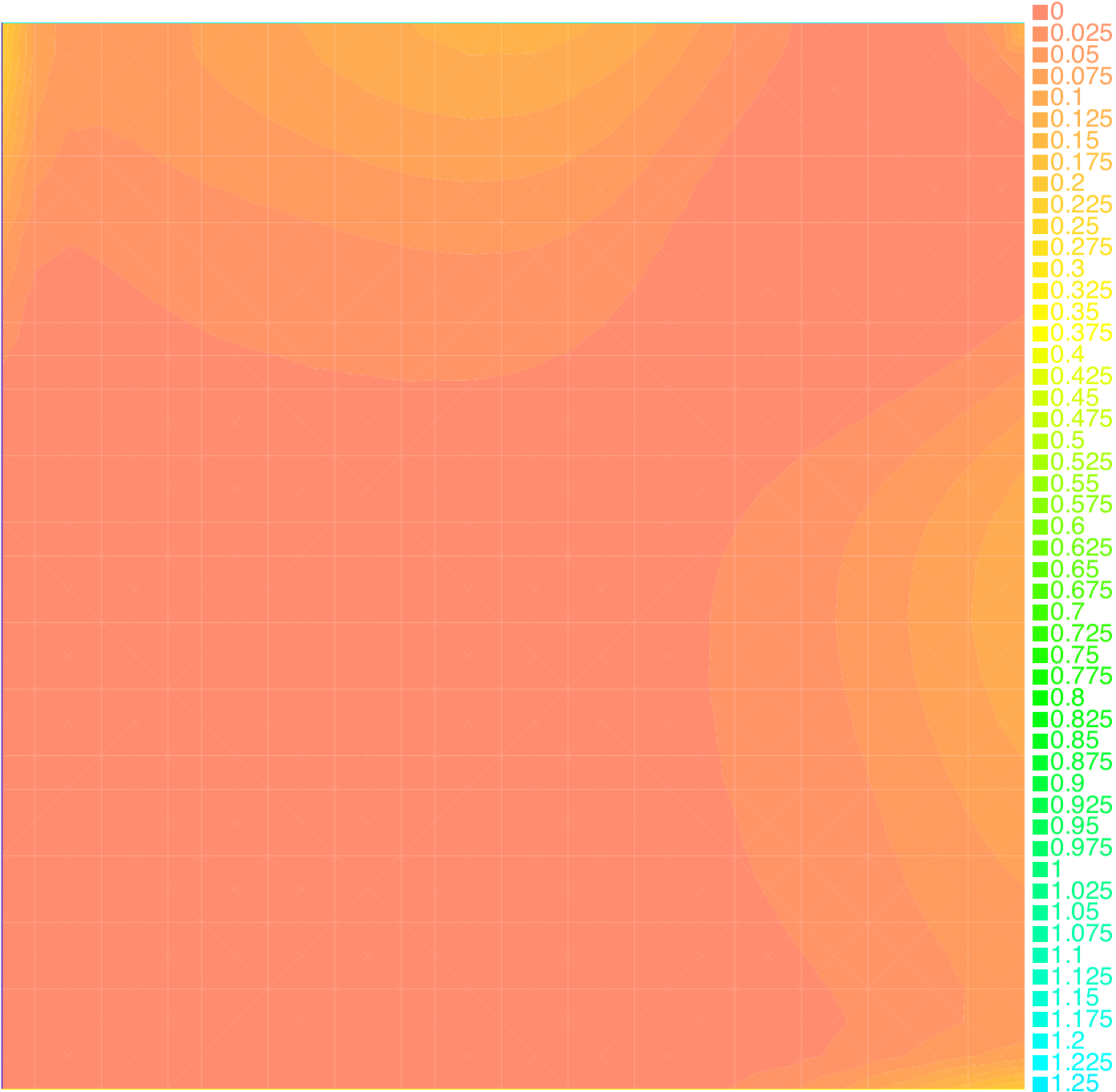}&
\includegraphics[ scale=0.15]{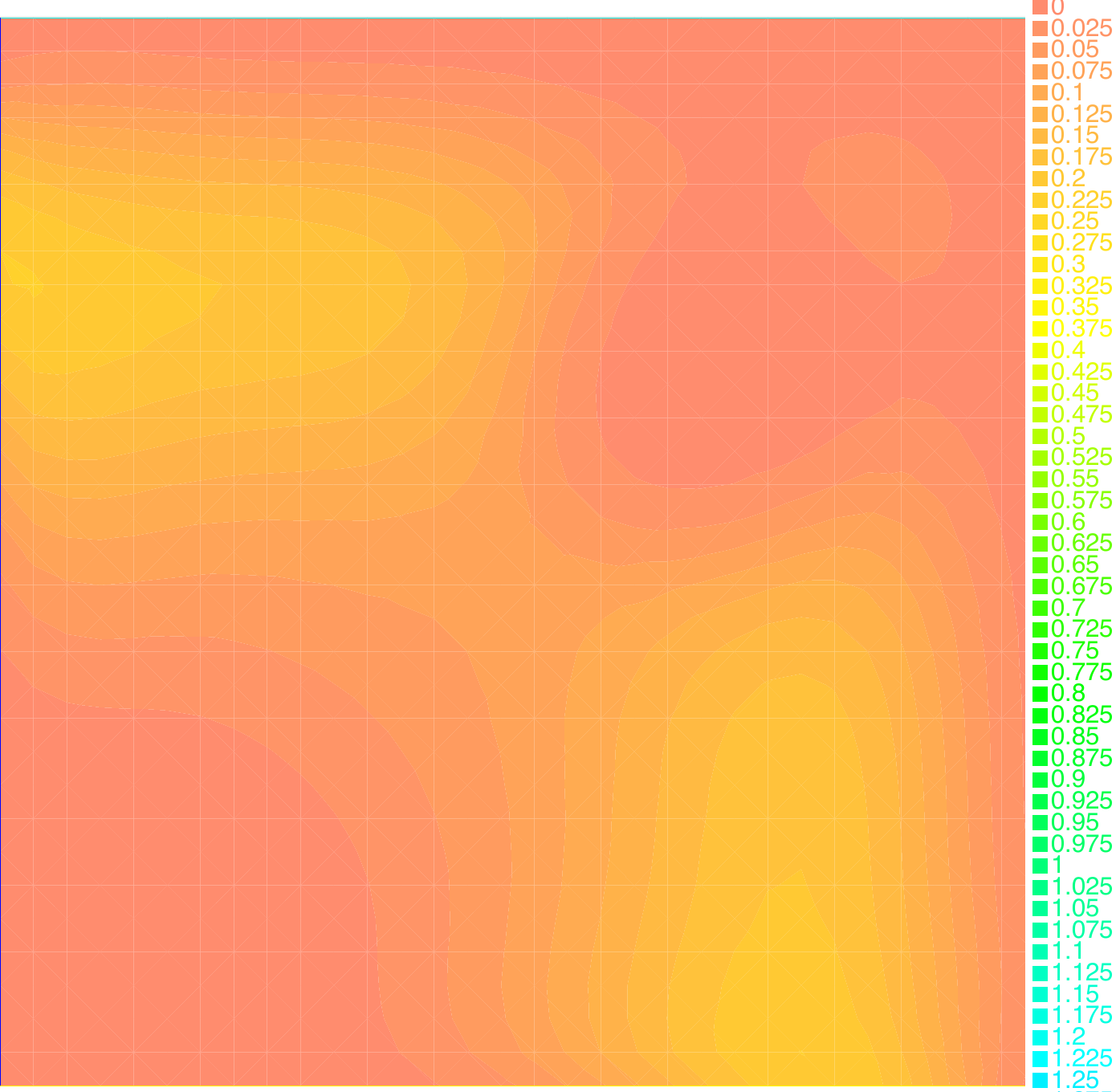}&
\includegraphics[ scale=0.15]{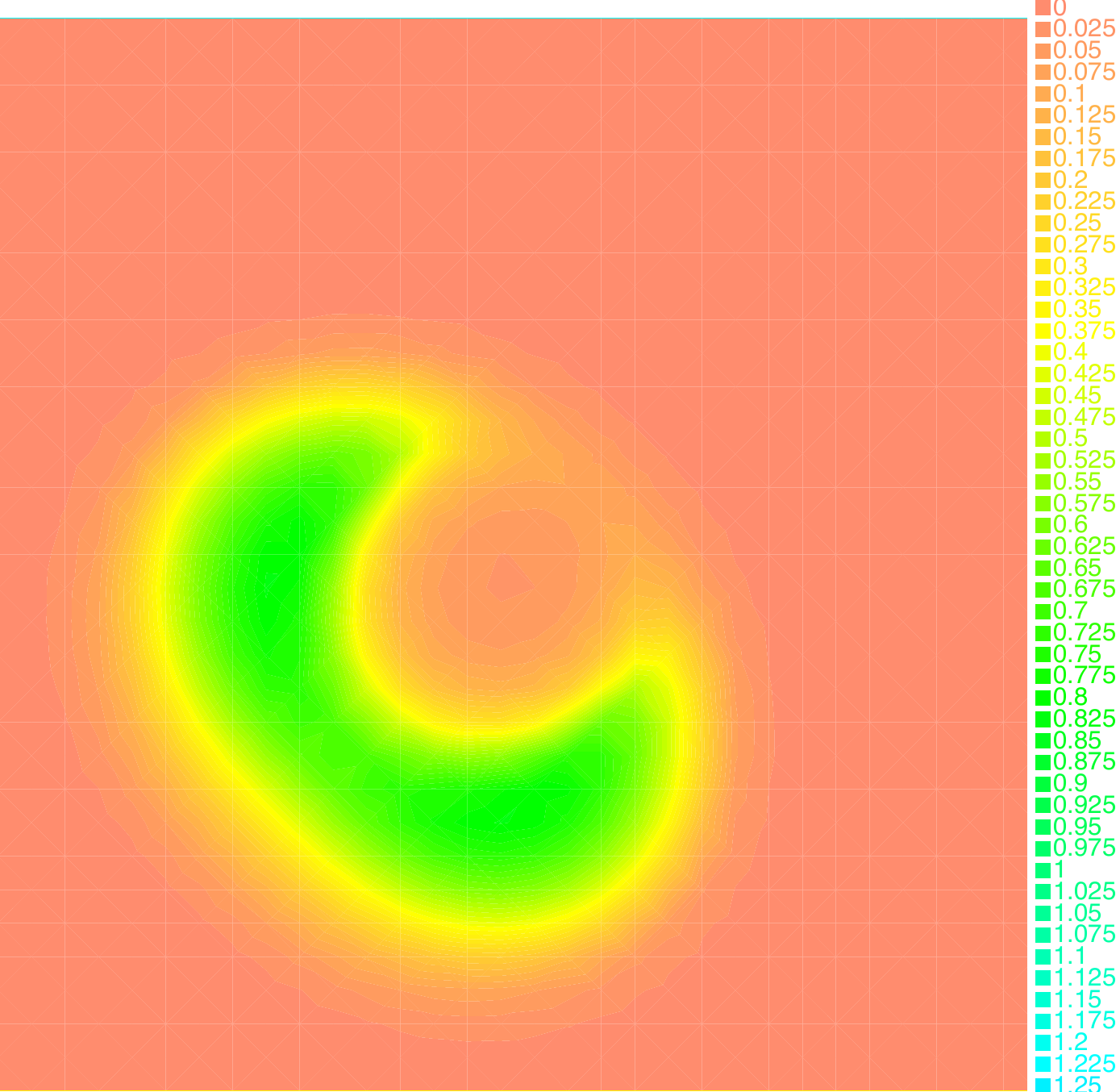}&
\includegraphics[ scale=0.15]{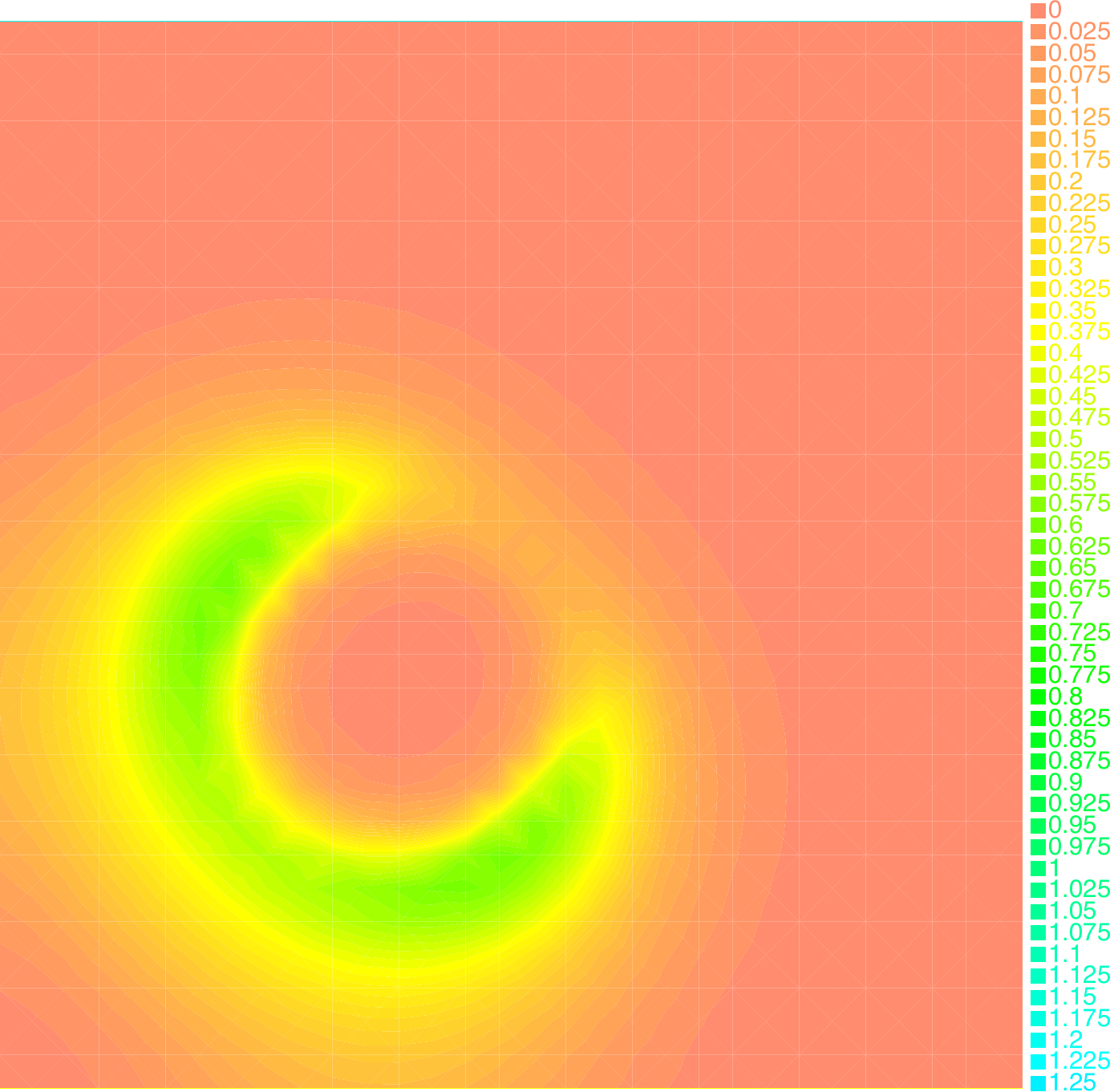}\\

\includegraphics[ scale=0.15]{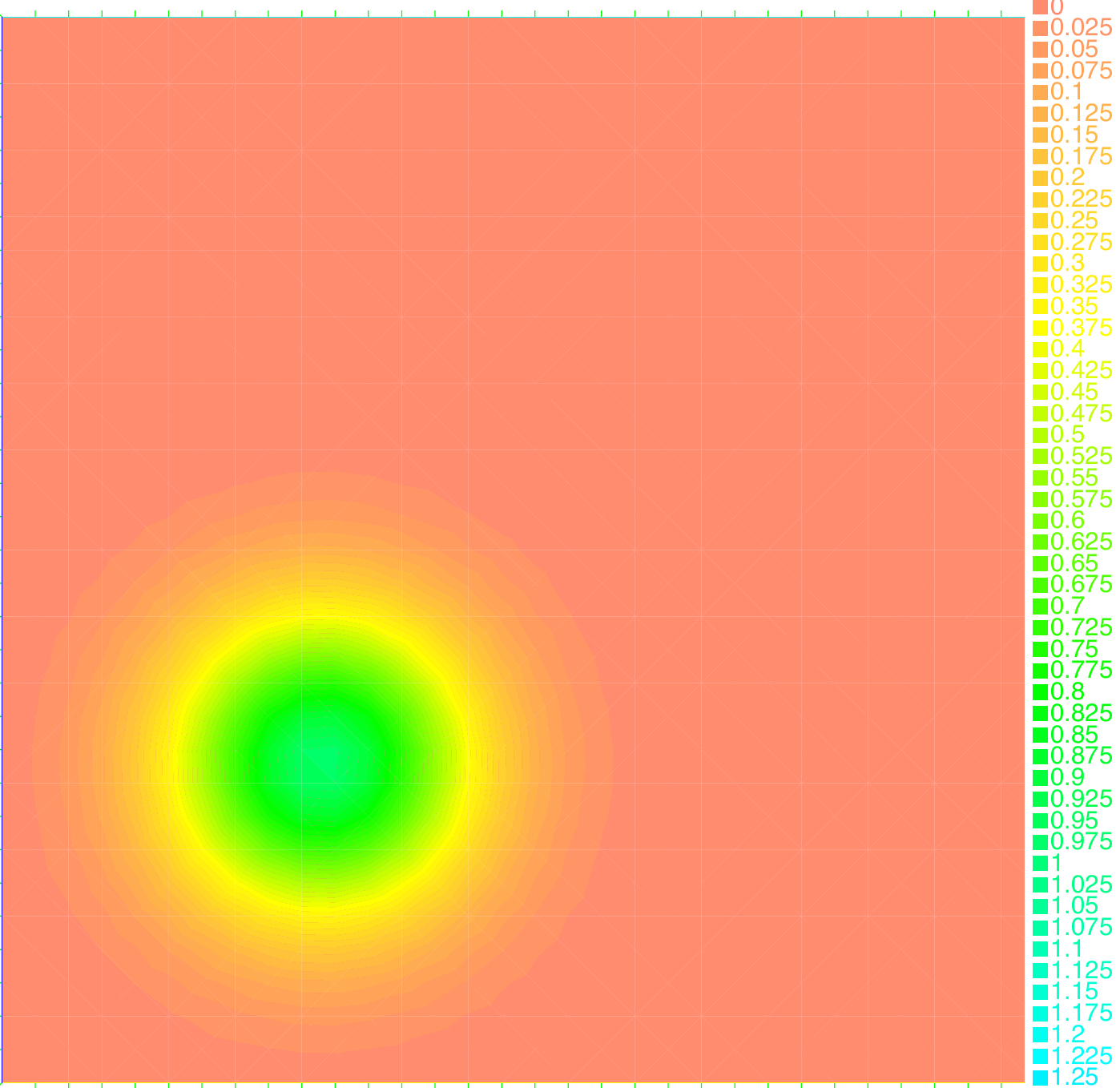}&
\includegraphics[ scale=0.15]{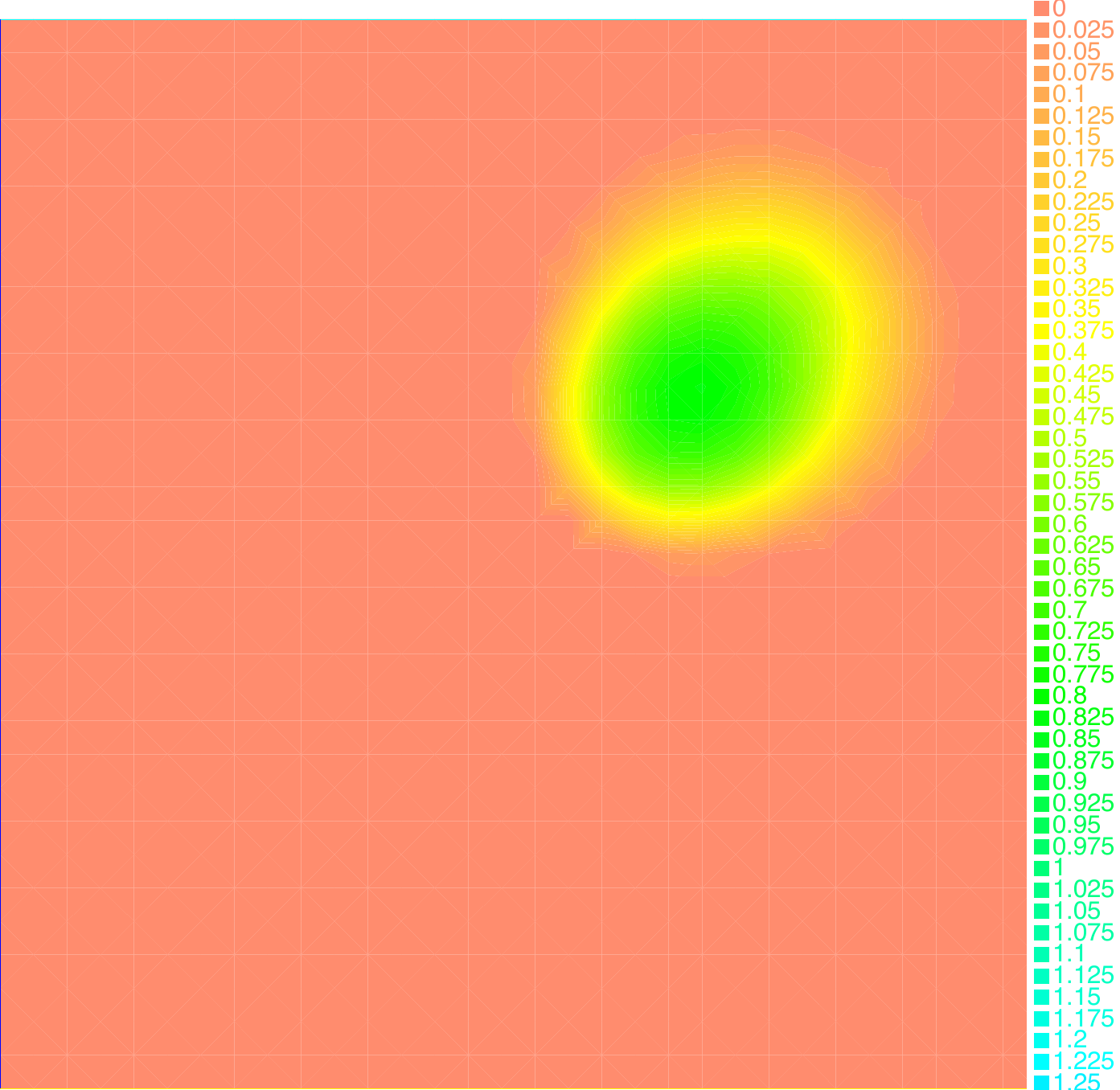}&
\includegraphics[ scale=0.15]{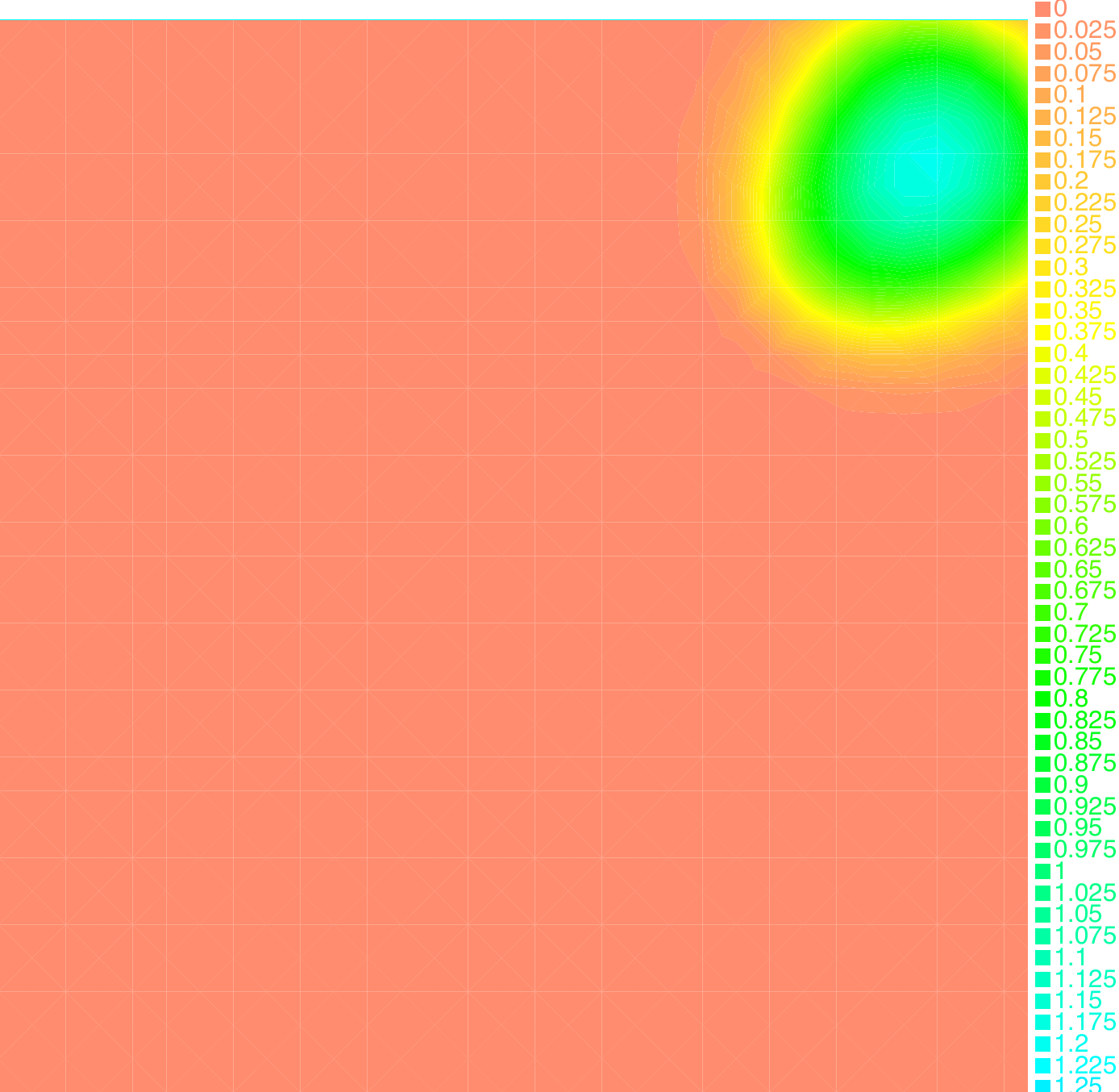}&
\includegraphics[ scale=0.15]{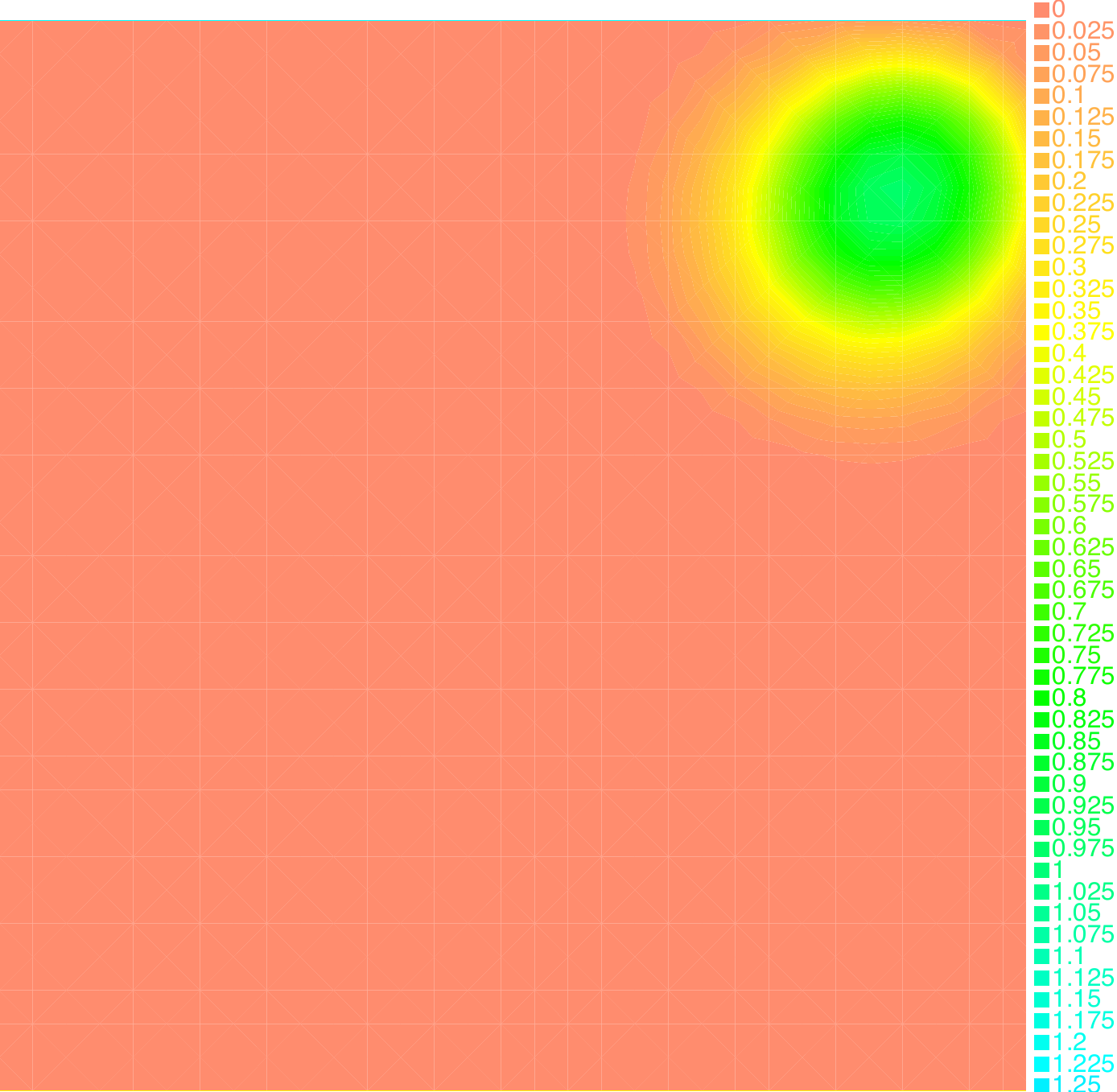}&
\includegraphics[ scale=0.15]{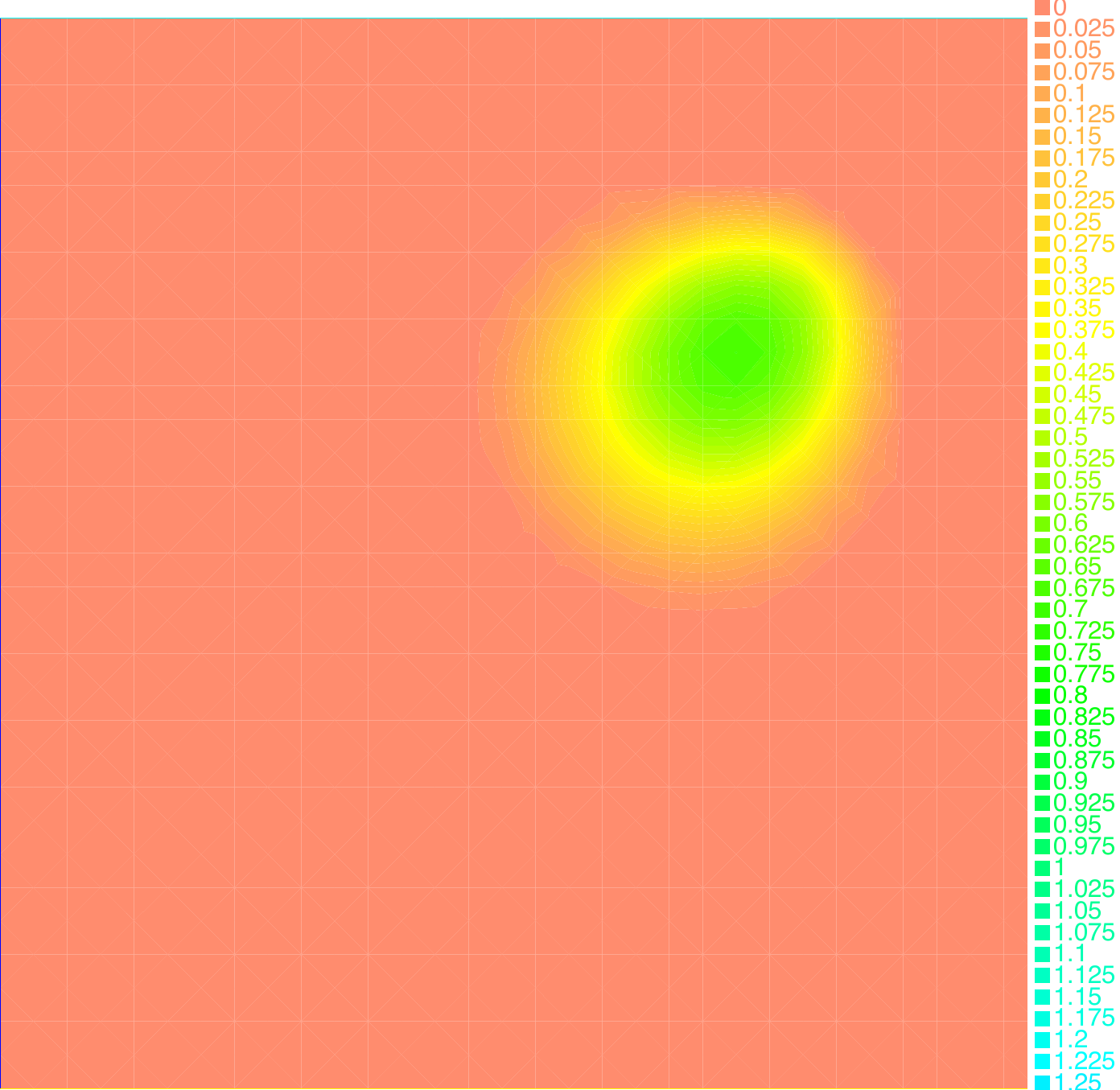}&
\includegraphics[ scale=0.15]{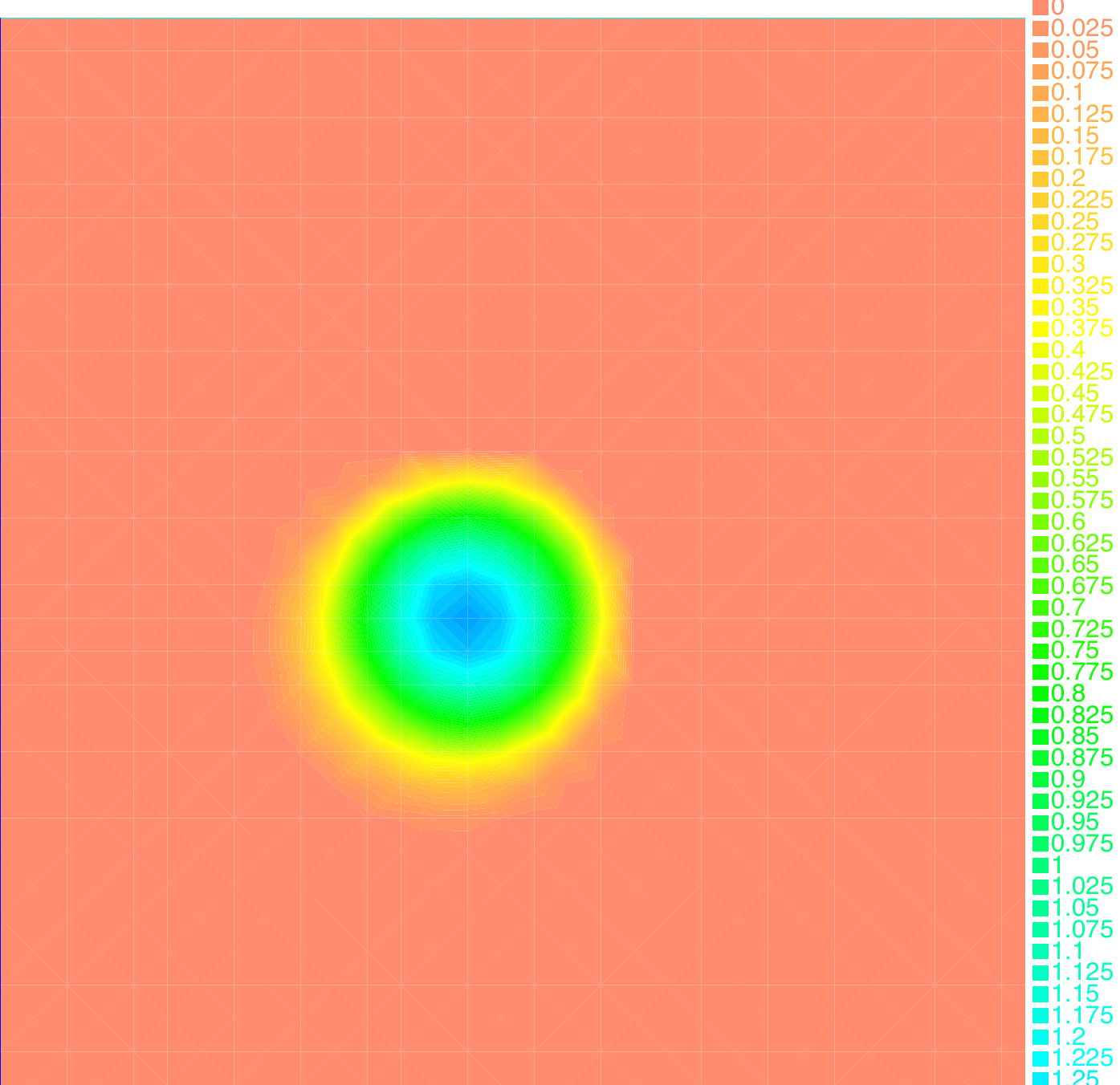}&
\includegraphics[ scale=0.15]{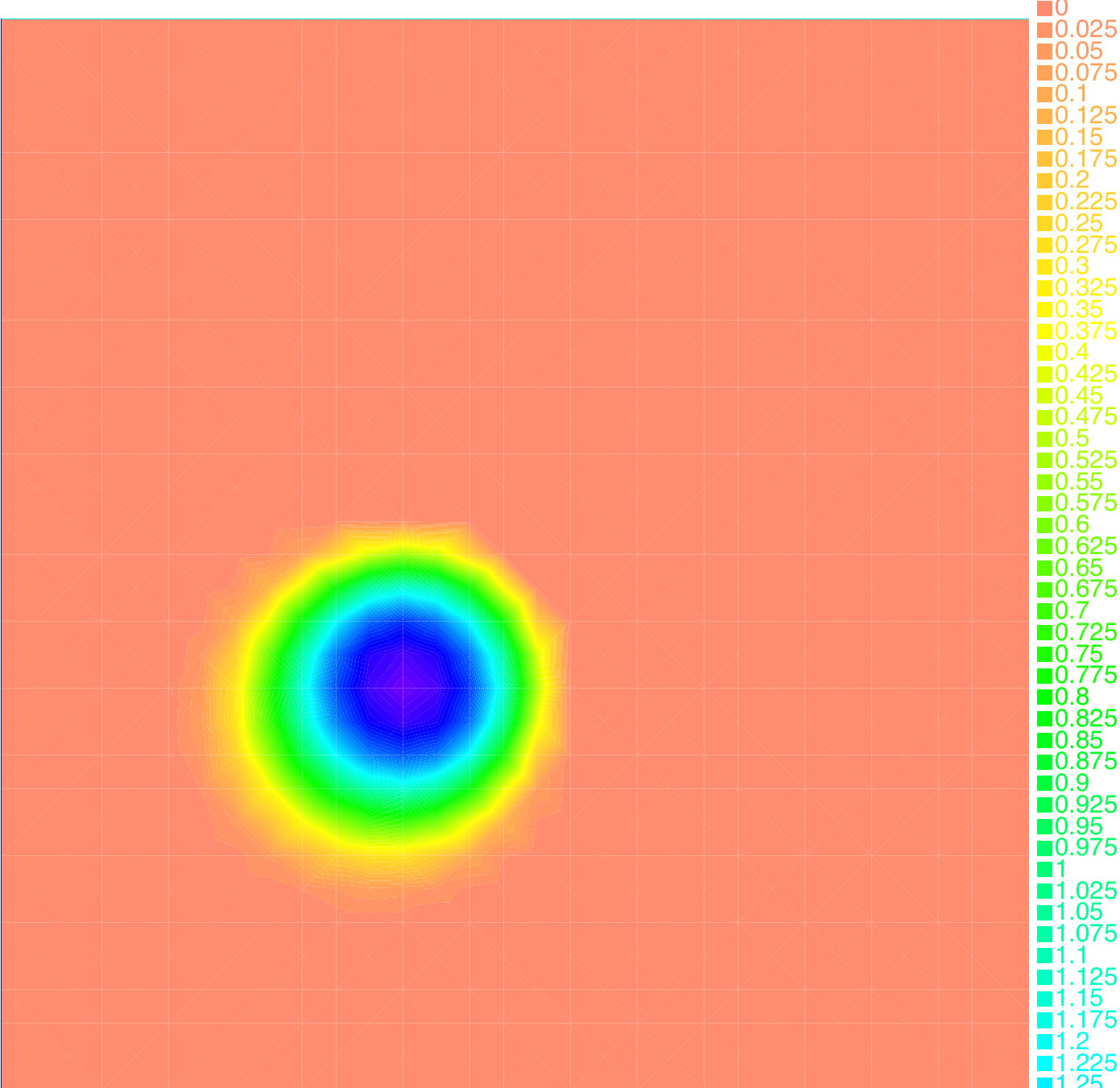}\\

\end{tabular}

\caption{\textit{Evolution of two species with prey-predator interactions. First row: display of $\rho_1+\rho_2$. Second row: display of the prey $\rho_1$. Third row: display of the predator $\rho_1$.}}
\label{alg2-figure prey-predator}
\end{figure}

\begin{figure}[h!]
\centering 
\includegraphics[ scale=0.35]{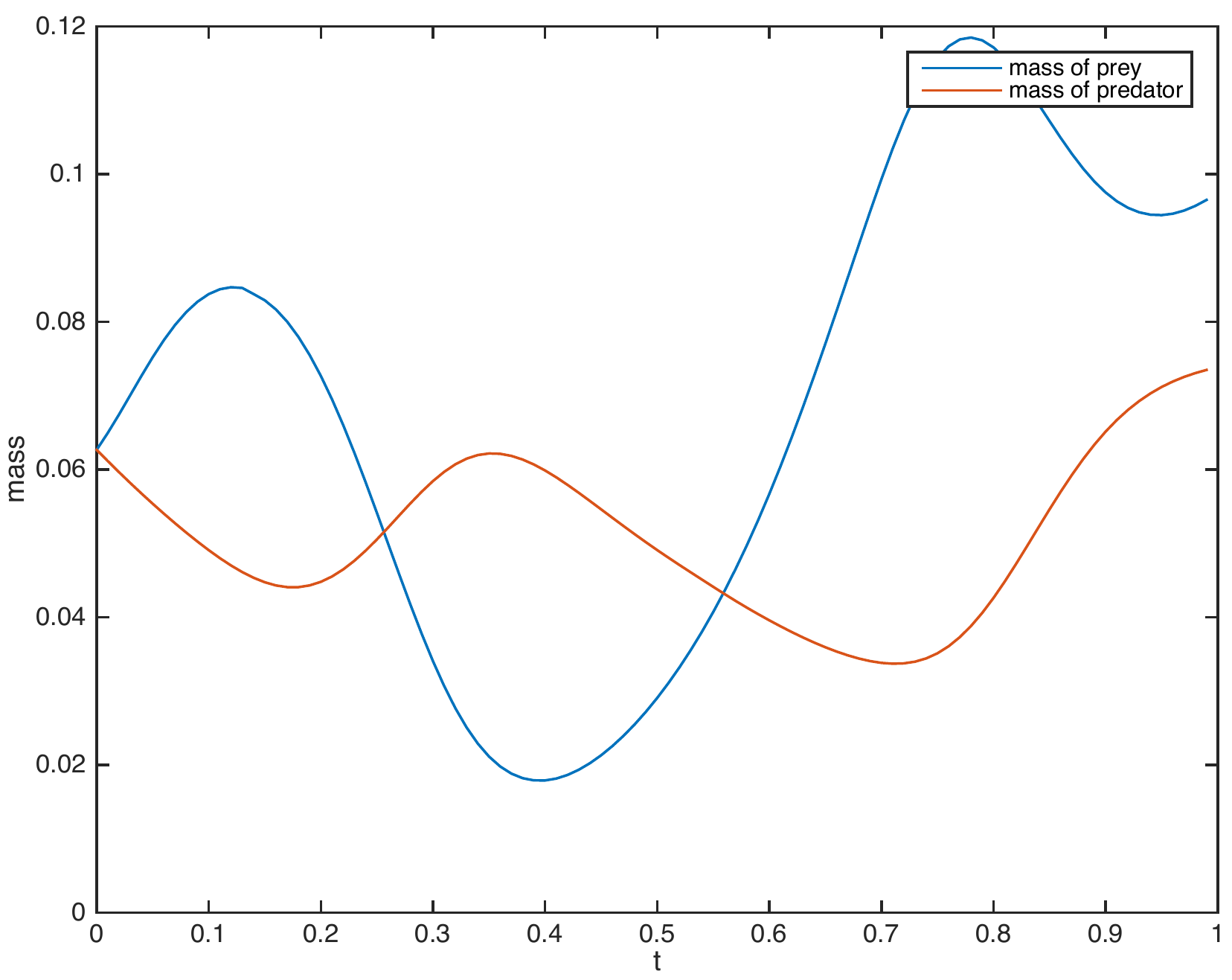}
\caption{\textit{Mass evolution for two-species prey-predator interactions.}}
\label{figure masse prey-predator}
\end{figure}

\section{Application to a tumor growth model with very degenerate enery}
\label{part4-section2HS}

In this section we take interest in the equation 
\begin{eqnarray}
\label{equation HS}
\left\{\begin{array}{l}
\partial_t \rho =\dive( \rho \nabla p)+ \rho (1-p),\\
p \geqslant 0\quad\mbox{and}\quad p(1-\rho)=0\\
0\leqslant \rho \leqslant 1,\\
\rho_{|t=0}=\rho_0.
\end{array} \right.
\end{eqnarray}
This equation is motivated by tumor growth models \cite{PQV,PTV} and exhibits a Hele-Shaw patch dynamics: if $\rho_0=\chi_{\Omega_0}$ then the solution remains an indicator $\rho(t)=\chi_{\Omega(t)}$ and the boundary moves with normal velocity $V=-\nabla p|_{\partial\Omega(t)}$, see \cite{alexander2014quasi} for a rigorous analysis in the framework of viscosity solutions.

At least formally, we remark that \eqref{equation HS} is the Wasserstein-Fisher-Rao gradient flow of the singular functional
$$
\F(\rho) :=\F_\infty(\rho) -\int_\Omega \rho,
$$
where
$$
\F_\infty(\rho):= \left\{\begin{array}{ll}
0 & \text{ if }   \rho \leqslant 1 \,\mbox{ a.e},\\
+\infty & \text{ otherwise.}
\end{array}\right.
$$
Indeed, the compatibility conditions $p\geqslant 0$ and $p(1-\rho)=0$ in \eqref{equation HS} really mean that the pressure $p$ belongs to the subdifferential $\partial \F_\infty(\rho)$, and \eqref{equation HS} thus reads as the gradient flow
$$
\partial_t\rho =\dive(\rho \nabla u)-\rho u,\qquad u=p-1\in -\partial \F(\rho).
$$
However, this functional is too singular for the previous splitting scheme to correctly capture the very degenerate diffusion.
Indeed, the naive and direct approach from section \ref{sec:KFRsplitting} would lead to
$$ \left\{\begin{array}{l}
\rho_h^{k+1/2} \in \argmin\limits_{\rho \leqslant1, \, |\rho|=|\rho_h^k|}\left\{  \frac{1}{2h} \W^2(\rho, \rho_{h}^k) -\int_\Omega \rho \right\},\\
\\
\rho_h^{k+1} \in \argmin\limits_{\rho \leqslant 1} \left\{ \frac{1}{2h} \FR^2(\rho, \rho_h^{k+1/2}) -\int_{\Omega} \rho \right\}.
\end{array}\right. $$
Since the Wasserstein step is mass-conservative by definition, the $\int\rho$ term has no effect in the first step and the latter reads as ``project $\rho_{h}^k$ on $\{\rho\leqslant 1\}$ w.r.t to the $\W$ distance''.
Since the output of the reaction step $\rho^{k+1}_h\leqslant 1$, the Wasserstein step will never actually project anything, and the diffusion is completly shut down.
As an example, it is easy to see that if the initial datum is an indicator $\rho_0=\chi_{\Omega_0}$ then the above naive scheme leads to a stationary solution $\rho^{k+1}_h=\rho^{k+1/2}_h=\rho_0$ for all $k\geqslant 0$, while the real solution should evolve according to the aforementioned Hele-Shaw dynamics $\rho(t)=\chi_{\Omega(t)}$ \cite{alexander2014quasi,PQV}.
One could otherwise try to write a semi-implicit scheme as follows: 1) keep the projection on $\{\rho\leqslant 1 \}$ in the first Wasserstein step.
As in \cite{MRCSV} a pressure term $p^{k+1/2}_h$ appears as a Lagrange multiplier in the Wasserstein projection.
2) in the $\FR$/reaction step, relax the constraint $\rho\leqslant 1$ and minimize instead $\rho^{k+1}\in\argmin \left\{\frac{1}{2h}\FR^2(\rho)+\int\rho p^{k+1/2}-\int\rho\right\}$, and keep iterating.
This seems to correctly capture the diffusion at least numerically speaking, but raises technical issues in the rigorous proof of convergence and most importantly destroys the variational structure at the discrete level (due to the fact that the reaction step becomes semi-explicit).

We shall use instead an approximation procedure, which preserves the variational structure at the discrete level: it is well-known that the Porous-Medium functional
$$
\F_m(\rho):=\left\{\begin{array}{ll}
\int_\Omega \frac{\rho^m}{m-1} & \text{ if } \rho^m \in L^1(\Omega)\\
+\infty& \text{ otherwise}
\end{array}\right. 
$$
$\Gamma$-converges to $\F_\infty$ as $m\to\infty$, see \cite{Braides}.
In the spirit of \cite{serfaty_sandier_gamma}, one should therefore expect that the gradient flow $\rho_m$ of $\F_m(\rho)-\int\rho$ converges to the gradient flow $\rho_\infty$ of the limiting functional $\F(\rho)=\F_\infty(\rho)-\int\rho$.
Implementing the splitting scheme for the regular energy functional $\F_m(\rho)-\int\rho$ gives a sequence $\rho_{h,m}$, and we shall prove below that $\rho_{h,m}$ converges to a solution of the limiting gradient flow as $m\to\infty$ and $h\to 0$.
However, it is known \cite{fleissner2016gamma} that the limit depends in general on the interplay between the time-step $h$ and the regularization parameter ($m\to\infty$ here), and for technical reasons we shall enforce the condition
$$
mh\to 0
\qquad \mbox{as }m\to\infty\mbox{ and }h\to 0.
$$
Note that \cite{PQV} already contained a similar approximation $m\to\infty$ but without exploiting the variational structure of the $m$- gradient flow, and our approach is thus different.
The above gradient-flow structure was already noticed and fully exploited in the ongoing work \cite{simone_lenaic}, where existence and uniqueness of weak solutions is proved and numerical simulations are performed needless of any splitting an using directly the $\KFR$ structure.
Here we rather emphasize the fact that the splitting does capture delicate $\Gamma$-convergence phenomena.\\

In order to make this rigorous, we fix a time step $h>0$ and construct two sequences $(\rho_{h,m}^{k+1/2})_k$ and $(\rho_{h,m}^{k})_k$, with $\rho_{h,m}^0=\rho_{0}$, defined recursively as
 \begin{eqnarray}
 \label{splitting scheme implicit}
\left\{\begin{array}{l}
\rho_h^{k+1/2} \in \argmin\limits_{ \rho\in \M^+,\,|\rho|=|\rho_h^k|} \left\{ \frac{1}{2h} \W^2(\rho, \rho_{h,m}^k) + \F_m(\rho) -\int_\Omega\rho\right\} ,\\
\\
\rho_h^{k+1} \in \argmin\limits_{\rho \in \M^+} \left\{ \frac{1}{2h} \FR^2(\rho, \rho_h^{k+1/2}) +\F_m(\rho) -\int_{\Omega} \rho \right\}.
\end{array}\right.
\end{eqnarray}

As is common in the classical theory of Porous Media Equations \cite{vazquez2007porous}, we define the pressure as the first variation
$$
p_m:=F_m'(\rho)=\frac{m}{m-1}\rho^{m-1}.
$$
We accordingly write
$$
p_{h,m}^{k+1/2}:=\frac{m}{m-1}(\rho_{h,m}^{k+1/2})^{m-1}
\qquad \text{and}\qquad 
p_{h,m}^{k+1}:=\frac{m}{m-1}(\rho_{h,m}^{k+1})^{m-1}
$$ 
for the discrete pressures.
As in section \ref{sec:KFRsplitting} we denote by $\rho_{h,m}(t), p_{h,m}(t)$ and $\trho_{h,m}(t),\tilde{p}_{h,m}(t)$ the piecewise constant interpolations of $\rho_{h,m}^{k+1}, p_{h,m}^{k+1}$ and $\rho_{h,m}^{k+1/2}, p_{h,m}^{k+1/2}$, respectively.

Our main result is
\begin{thm}
\label{existence HS implicit}
Assume that $\rho_0 \in  BV(\Omega)$, $\rho_0 \leqslant 1$, and $mh\to 0$ as $h\to 0$ and $m\to \infty$.
Then
for all $T>0$, $\rho_{h,m},\trho_{h,m}$ both converge to some $\rho$ strongly in $L^1((0,T)\times \Omega)$, the pressures $p_{h,m},\tilde{p}_{h,m}$ both converge to some $p$ weakly in $L^2((0,T) , H^1(\Omega))$, and $(\rho,p)$ is the unique weak solution of \eqref{equation HS}.
\end{thm}

Since we have a $\KFR$ gradient-flow structure, uniqueness should formally follows from the $-1$ geodesic convexity of the driving functional $\E_\infty(\rho)-\int_\Omega\rho$ with respect to the $\KFR$ distance \cite{liero2013gradient,LMS_small_15} and the resulting contractivity estimate $\KFR(\rho^1(t),\rho^2(t))\leq e^t\KFR(\rho^1_0,\rho^2_0)$.
This is proved rigorously in \cite{simone_lenaic}, and therefore we retrieve convergence of the whole sequence $\rho_{h,m}\to \rho$ in Theorem~\ref{existence HS implicit} (and not only for subsequences).
Given this uniqueness, it is clearly enough to prove convergence along any discrete (sub)sequence, and this is exactly what we show below.

The strategy of proof for Theorem~\ref{existence HS implicit} is exactly as in section \ref{sec:KFRsplitting}, except that we need now the estimates to be uniform in both in $h\to 0$ and $m\to\infty$.
\subsection{Estimates and convergences}

In this section, we improve the previous estimates from section \ref{sec:KFRsplitting}.
We start with an explicit $L^\infty$-bound:

\begin{lem}
\label{lem:borne sup inf 1}
Assume that $\rho_0 \leqslant 1$, then for all $t \in \R^+$,
$$ \| \rho_{h,m}(t,\cdot) \|_\infty,\| \trho_{h,m}(t,\cdot) \|_\infty \leqslant 1 .$$
\end{lem} 

\begin{proof}
We argue by induction at the discrete level, starting from $\rho_0=\rho^0_{h,m}\leqslant 1$ by assumption.
If $\| \rho_{h,m}^k \|_\infty \leqslant 1$, Otto's maximum principle \cite{O} implies that $\| \rho_{h,m}^{k+1/2} \|_\infty \leqslant\| \rho_{h,m}^{k} \|_\infty\leqslant 1$ in the Wasserstein step.

Assume now by contradiction that $E:=\{\rho_{h,m}^{k+1} >1\}$ has positive Lebesgue measure.
The optimality condition \eqref{opt Fisher-Rao} for the Fisher-Rao minimization step gives, dividing by $\sqrt{\rho_{h,m}^{k+1}}>0$ in $E$,
$$
\sqrt{\rho_{h,m}^{k+1}}- \sqrt{\rho_{h,m}^{k+1/2}}= \frac{h}{2}\sqrt{\rho_{h,m}^{k+1}} \left(1- \frac{m}{m-1} (\rho_{h,m}^{k+1})^{m-1} \right)
$$
Then $1- \frac{m}{m-1} (\rho_{h,m}^{k+1})^{m-1} \leqslant 1-\frac{m}{m-1}< 0$ in the right-hand side, hence the desired contradiction $\rho_{h,m}^{k+1} < \rho_{h,m}^{k+1/2} \leqslant 1$.
\end{proof}
%
Noticing that the functional $\frac{1}{m-1}\int \rho^m-\int\rho$ corresponds to taking explicitly $F_2(z)=z^{m}/m-1$ and $V_2(x)\equiv-1$ in section~\ref{sec:KFRsplitting}, it is easy to reproduce the computations from the proof of Lemma~\ref{lem encadrement} and carefully track the dependence of the constants w.r.t $m>1$ to obtain
\begin{lem}
There exists $c>0$ such that, for all $m>m_0$ large enough and all $h\leq h_0$ small enough,
 \begin{equation}
 \label{eq:encadrement_FR_HS}
 (1-ch)\rho^{k+1/2}_{h,m}(x)\leqslant \rho^{k+1}_{h,m}(x)\leqslant (1+h)\rho^{k+1/2}_{h,m}(x)\qquad 
 \mbox{a.e.}
 \end{equation}
\end{lem}
Note that this holds regardless of any compatibility such as $hm\to 0$.
The key point is here that the lower bound $c$ previously depended on an upper bound $M$ on $\rho^{k+1/2}$ in Lemma~\ref{lem encadrement}, but since we just obtained in Lemma~\ref{lem:borne sup inf 1} the universal upper bound $\rho^{k+1/2}\leqslant 1$ we end up with a lower bound which is also uniform in $h,m$.
The proof is identical to that of Lemma~\ref{lem encadrement} and we omit the details for simplicity.

Recalling that the Wasserstein step is mass-preserving, we obtain by immediate induction and for all $0\leq t\leq T$
$$
\|\rho_{h,m}(t)\|_{L^1},\,\|\trho_{h,m}(t)\|_{L^1}\leqslant e^T\|\rho_0\|_{L^1}
$$
as well as
\begin{equation}
\label{eq:comparison_p_ptilda_HS}
\|\rho_{h,m}(t)-\trho_{h,m}(t)\|_{L^1}\leqslant C_Th.
\end{equation}

Testing successively $\rho=\rho_{h,m}^k$ and $\rho = \rho_{h,m}^{k+1/2}$ in \eqref{splitting scheme implicit}, we get
\begin{equation*}
\frac{1}{2h}\left(\W^2(\rho_{h,m}^k,\rho_{h,m}^{k+1/2}) +\FR^2( \rho_{h,m}^{k+1/2},\rho_{h,m}^{k+1}) \right)  \leqslant  \F_m(\rho_{h,m}^k) -\F_m(\rho_{h,m}^{k+1}) +\int_{\Omega}(\rho_{h,m}^{k+1/2}-\rho_{h,m}^{k+1}).
\end{equation*}
Using Proposition~\ref{prop:comparison_d_W_H} to control $\KFR^2\lesssim 2(\W^2+\FR^2)$ and the lower bound in \eqref{eq:encadrement_FR_HS} yields
\begin{eqnarray*}
\frac{1}{4h}\KFR^2(\rho^{k+1}_{h,m},\rho^k_{h,m})
&\leqslant & \frac{1}{2h}\left(\W^2(\rho_{h,m}^k,\rho_{h,m}^{k+1/2}) +\FR^2( \rho_{h,m}^{k+1/2},\rho_{h,m}^{k+1}) \right)\\
&\leqslant & \F_m(\rho_{h,m}^k) -\F_m(\rho_{h,m}^{k+1}) +\int_{\Omega}(\rho_{h,m}^{k+1/2}-\rho_{h,m}^{k+1})\\
&\leqslant & \F_m(\rho_{h,m}^k) -\F_m(\rho_{h,m}^{k+1}) +c h\int_{\Omega}\rho_{h,m}^{k+1/2}\\
&\leqslant  &\F_m(\rho_{h,m}^k) -\F_m(\rho_{h,m}^{k+1}) +c he^T
\end{eqnarray*}
for all $k\leqslant N:= \lfloor T/h \rfloor$.

Summing over $k$ we get
\begin{eqnarray*}
\frac{1}{4h}\sum_{k=0}^{N-1} \KFR^2(\rho_{h,m}^k, \rho_{h,m}^{k+1}) &\leqslant & \F_m(\rho_0) - \F_m(\rho_{h,m}^N) +C_T\\
&\leqslant &\frac{1}{m-1}\int_\Omega\rho_0^m + C_T \leqslant \frac{1}{m-1}\int_\Omega\rho_0 + C_T\leqslant C_T,
\end{eqnarray*}

where we used successively $F_m\geq 0$ to get rid of $\F_m(\rho^N_{h,m})$, and $\rho_0^m\leq \rho_0$ for $\rho_0\leq 1$ and $m>1$. 

Consequently, for all fixed $T>0$ and any $t,s\in[0,T]$ we obtain the classical $\frac{1}{2}$-H\"older estimate
\begin{eqnarray}
\label{part42-holder estimates}
\left\{ \begin{array}{l}
\KFR(\rho_{h,m}(t),\rho_{h,m}(s)) \leqslant C_T|t-s+h|^{1/2},\\
\KFR(\trho_{h,m}(t),\trho_{h,m}(s)) \leqslant C_T|t-s+h|^{1/2}.
\end{array}\right.
\end{eqnarray} 
\\

Exploiting the explicit algebraic structure of $F_m(z)=\frac{1}{m-1}z^m$, compactness in space will be given here by
\begin{lem}
\label{prop:BV_estimate}
If $\rho_0 \in BV(\Omega)$ then
$$
\sup_{t \in [0,T]} \left\{\| \rho_{h,m}(t,\cdot) \|_{BV(\Omega)}, \| \trho_{h,m}(t,\cdot) \|_{BV(\Omega)} \right\} \leqslant e^T \| \rho_0 \|_{BV(\Omega)}.
$$
\end{lem}

\begin{proof}
The argument closely follows the lines of \cite[prop. 5.1]{GM}.
We first note from \cite[thm. 1.1]{DPMSV} that the $BV$-norm is nonincreasing during the Wasserstein step, 
$$
\| \rho_{h,m}^{k+1/2} \|_{BV(\Omega)} \leqslant \| \rho_{h,m}^{k} \|_{BV(\Omega)} .
$$
Using as before the implicit function theorem, we show below  that $\rho_{h,m}^{k+1} = R(\rho_{h,m}^{k+1/2})$ for some suitable $(1+h)$-Lispchitz function $R$. 
By standard $Lip\circ BV$ composition \cite{AFP} this will prove that
$$ \| \rho_{h,m}^{k+1} \|_{BV(\Omega)} \leqslant (1+h)\| \rho_{h,m}^{k+1/2} \|_{BV(\Omega)}$$
and will conclude the proof by immediate induction.

Indeed, we already know from \eqref{eq:encadrement_FR_HS} that $\rho_{h,m}^{k+1/2}$ and $\rho_{h,m}^{k+1}$ share the same support.
In this support and from \eqref{opt Fisher-Rao} it is easy to see that $\rho=\rho_{h,m}^{k+1}(x)$ is the unique positive solution of 
$f(\rho,\rho_{h,m}^{k+1/2}(x))=0$ with 
$$
f(\rho,\mu)= \sqrt{\rho}\left(1- \frac{h}{2} \left(1- \frac{m}{m-1}\rho^{m-1}  \right) \right) - \sqrt{\mu}.
$$
For $\mu >0$, the implicit function theorem gives the existence of a $\mathcal{C}^1$ map $R$ such that $f(\rho, \mu)=0\Leftrightarrow \rho =R(\mu)$, with $R(0)=0$.
An algebraic computation shows moreover that $0<\frac{d R}{d\mu}= - {\frac{\partial_\mu f}{\partial_\rho f}}_{|\rho =R(\mu)}\leqslant (1+h)$ uniformly in $m>1$, hence $R$ is $(1+h)$-Lipschitz as claimed and the proof is complete.

\end{proof}

\begin{prop}
\label{prop:CV_phm_rhohm}
Up to extraction of a discrete sequence $h\to 0,m\to\infty$, there holds
$$
\rho_{h,m},\,\trho_{h,m}\to \rho \qquad \mbox{strongly in }L^1(Q_T)
$$
$$
p_{h,m}\rightharpoonup p \quad\mbox{and}\quad\tilde p_{h,m}\rightharpoonup \tilde p \qquad\mbox{weakly in all }L^q(Q_T)
$$
for all $T>0$.
If in addition $mh\to 0$ then $p=\tilde p$.
\end{prop}
\begin{proof}
The first part of the statement follows exactly as in section \ref{sec:KFRsplitting}, exploiting the $\frac{1}{2}$-H\"older estimates \eqref{part42-holder estimates} and the space compactness from Proposition~\ref{prop:BV_estimate} in order to apply the Rossi-Savar\'e theorem \cite{RS}.
The fact that $\rho_{h,m},\trho_{h,m}$ have the same limit comes from \eqref{eq:comparison_p_ptilda_HS}.

For the pressures, we simply note from $\rho_{h,m}\leqslant 1$ and $m\gg 1$ that $p_{h,m}=\frac{m}{m-1}\rho_{h,m}^{m-1}\leqslant 2\rho_{h,m}$ is bounded in $L^1\cap L^\infty(Q_T)$ uniformly in $h,m$ in any finite time interval $[0,T]$.
Thus up to extraction of a further sequence we have $p_{h,m}\rightharpoonup p$ in all $L^q(Q_T)$, and likewise for $\tilde p_{h,m}\rightharpoonup \tilde p$.

Finally, we only have to check that $p=\tilde p$ if $hm\to 0$.
Because $\rho_{h,m},\trho_{h,m}\leqslant 1$ and $z\mapsto z^{m-1}$ is $(m-1)$-Lipschitz on $[0,1]$ we have for all fixed $t\geqslant 0$ that
\begin{eqnarray*}
\int_\Omega |p_{m,h}(t,\cdot) - \tilde{p}_{m,h}(t,\cdot) | &= & \int_\Omega \frac{m}{m-1}|\rho_{h,m}^{m-1}(t,\cdot)-\trho_{h,m}^{m-1}(t,\cdot) |\\
&\leqslant &  m\int_\Omega |{\rho_{h,m}(t)}-{\trho_{h}(t)} |
\leqslant  C_Thm \longrightarrow 0,
\end{eqnarray*}
where we used \eqref{eq:comparison_p_ptilda_HS} in the last inequality.
Hence $p=\tilde p$ and the proof is complete.

\end{proof}

In order to pass to the limit in the diffusion term $\dive(\rho\nabla p)$ we first improve the convergence of $\tilde{p}_{h,m}$:

\begin{lem}
\label{borne pression porous media}
There exists a constant $C_T$, independent of  $h$ and $m$, such that
$$\| \tilde{p}_{h,m} \|_{L^2((0,T),H^1(\Omega))} \leqslant C_T$$
for all $T>0$.
Consequently, up to a subsequence, $\tilde{p}_{h,m}$ converges weakly in $L^2((0,T),H^1(\Omega))$ to $p$.
\end{lem}

\begin{proof}
The proof is based on the flow interchange technique developed by Matthes, McCann and Savar\'e in \cite{MMCS}.
Let $\eta$ be the (smooth) solution of 
$$ \left\{ \begin{array}{l}
\partial_t \eta = \Delta \eta^{m-1} + \eps \Delta \eta,\\
\eta|_{t=0}=\rho^{k+1/2}_{h,m}.
\end{array}\right.
$$
It is well known \cite{AGS} that $\eta$ is the Wasserstein gradient flow of 
$$ 
\G(\rho):= \int_\Omega \frac{\rho^{m-1}}{m-2}  + \eps \int_\Omega \rho \log(\rho).
$$
Since $\G$ is geodesically $0$-convex, $\eta$ satisfies the Evolution Variational Inequality (EVI)
$$
\left.\frac{1}{2}{\frac{d^+}{dt}}\right|_{t=s} \W^2(\eta(s), \rho) \leqslant \G(\rho) -\G(\eta(s)),
$$
for all $s>0$ and for all $\rho \in \Paa(\Omega)$, where  $\frac{d^+}{dt}f(t):=\limsup\limits_{s\rightarrow 0^+} \frac{f(t+s) -f(t)}{s}$. 
By optimality of $\rho^{k+1/2}_{h,m}$ in \eqref{splitting scheme implicit}, we obtain that
$$
\left.\frac{1}{2}{\frac{d^+}{dt}}\right|_{t=s} \W^2(\eta(s),\rho_{h,m}^{k})\geqslant -h \left.{\frac{d^+}{dt}}\right|_{t=s} \F_m(\eta(s)).
$$
Since $\eta$ is smooth due to the regularizing $\eps\Delta$ term, we can legitimately integrate by parts for all $s>0$
\begin{eqnarray*}
\frac{d}{ds} \F_m(\eta(s))&=& \int_\Omega \frac{m}{m-1} \eta(s)^{m-1}(\Delta\eta(s)^{m-1} + \eps \Delta \eta(s) )\\
& = & -\int_\Omega \frac{m}{m-1} |\nabla \eta(s)^{m-1} |^2 -\eps \int_\Omega m\eta(s)^{m-2}|\nabla \eta(s) |^2\\
& \leqslant & -\int_\Omega \frac{m}{m-1} |\nabla \eta(s)^{m-1} |^2
=-\frac{m-1}{m} \int_\Omega \left|\nabla \left(\frac{m}{m-1}\eta(s)^{m-1}\right) \right|^2
\end{eqnarray*}

Remarking that $\frac{m}{m-1}\eta(s)^{m-1}\to\frac{m}{m-2}\rho_{h,m}^{k+1/2}=p_{h,m}^{k+1/2}$ as $s\to 0$, an easy lower semi-continuity argument gives that
$$
\int_\Omega \frac{m-1}{m} |\nabla p_{h,m}^{k+1/2}|^2= \int_\Omega \frac{m}{m-1} |\nabla (\rho_{h,m}^{k+1/2})^{m-1} |^2 \leqslant \liminf_{s\searrow 0}\left.{\frac{d^+}{dt}}\right|_{t=s} \F_m(\eta(s)).
$$

Then we have
\begin{align*}
 h\int_\Omega \frac{m-1}{m} |\nabla p_{h,m}^{k+1/2}|^2 & \leqslant \F_{m-1}(\rho_{h,m}^k) - \F_{m-1}(\rho_{h,m}^{k+1/2})\\
 & + \eps \left( \int_\Omega \rho_{h,m}^k \log(\rho_{h,m}^k) -\int_\Omega \rho_{h,m}^{k+1/2} \log(\rho_{h,m}^{k+1/2}) \right).
\end{align*}
First arguing as in Proposition \ref{prop:decr F1} to control
$$
\F_{m-1}(\rho_{h,m}^{k+1}) \leqslant \F_{m-1}(\rho_{h,m}^{k+1/2}) +C_Th,
$$
 and then passing to the limit $\eps \searrow 0$, we obtain
$$
h\int_\Omega \frac{m-1}{m} |\nabla p_{h,m}^{k+1/2}|^2  \leqslant \F_{m-1}(\rho_{h,m}^k) - \F_{m-1}(\rho_{h,m}^{k+1}) +C_Th.
$$
Summing over $k$ gives
$$
\int_0^T \int_\Omega  |\nabla \tilde{p}_{h,m}(t,x)|^2 \,dxdt \leqslant \frac{m}{m-1}( \F_{m-1}(\rho_0)- \F_{m-1}(\rho^N_{h,m})+C_T)
\leqslant 2 \F_{m-1}(\rho_0)+C_T
$$
for all $T<+\infty$.
Due to $\rho_0\leqslant 1$ and $m\gg1 $ we can bound $\F_{m-1}(\rho_0)=\frac{1}{m-2}\int\rho_0^{m-1}\leqslant \frac{1}{m-2}\int\rho_0\leqslant \|\rho_0 \|_{L^1(\Omega)}$ and the result finally follows.
\end{proof}

\subsection{Properties of the pressure $p$ and conclusion}
\label{subsec:prop_pressure_HS}
We start by showing that the limits $\rho,p$ satisfy the compatibility conditions in \eqref{equation HS}.

\begin{lem}
\label{lem:properties pressure}
There holds
$$ 
0\leqslant \rho,p\leqslant 1 \quad \text{and} \quad p(1-\rho) =0  \, \text{ a.e. in }Q_T.
$$
\end{lem}

\begin{proof}
By Lemma~\ref{lem:borne sup inf 1} it is obvious that $ 0\leqslant \rho \leqslant 1$ and $0\leqslant p \leqslant 1$ are inherited from $0\leqslant \rho_{h,m}\leqslant 1$ and $0\leqslant p_{h,m}=\frac{m}{m-1}\rho^{m-1}_{h,m}\leqslant \frac{m}{m-1}$.

In order to prove that $p(1-\rho)=0$, we first observe that
$$
p_{h,m}(1-\rho_{h,m})\to 0\qquad \mbox{a.e. in }Q_T.
$$
Indeed, since $\rho_{h,m}\to \rho$ strongly in $L^1(Q_T)$ we have $\rho_{h,m}(t,x)\to \rho(t,x)$ a.e.
If the limit $\rho(t,x)<1$ then $\rho_{h,m}(t,x)\leqslant (1-\eps)$ for small $h$ and large $m$.
Hence $p_{h,m}(t,x)=\frac{m}{m-1}\rho_{h,m}^{m-1}\leqslant \frac{m}{m-1}(1-\eps)^{m-1}\to 0$ while $1-\rho_{h,m}$ remains bounded, and therefore the product $p_{h,m}(1-\rho_{h,m})\to 0$.
Now if the limit $\rho(t,x)=1$ then the pressure $p_{h,m}=\frac{m}{m-1}\rho_{h,m}^{m-1}\leqslant \frac{m}{m-1}$ remains bounded, while $1-\rho_{h,m}(t,x)\to 0$ hence the product goes to zero in this case too.

Thanks to the uniform $L^\infty$ bounds $\rho_{h,m}\leqslant 1$ and $p_{h,m}\leqslant \frac{m}{m-1}\leqslant 2$ we can apply Lebesgue's convergence theorem to deduce from this pointwise a.e. convergence that, for all fixed nonnegative $\varphi\in\mathcal \mathcal C^{\infty}_c(Q_T)$, there holds
$$
\lim \int_{Q_T}p_{h,m}(1-\rho_{h,m})\varphi =0.
$$
On the other hand since $\rho_{h,m}\to \rho$ strongly in $L^1(Q_T)$ hence a.e, and because $0\leqslant \rho_{h,m}\leqslant 1$, we see that $(1-\rho_{h,m})\varphi\to (1-\rho)\varphi$ in all $L^q(Q_T)$.
From Proposition~\ref{prop:CV_phm_rhohm} we also had that $p_{h,m}\rightharpoonup p$ in all $L^q(Q_T)$, hence by strong-weak convergence we have that
$$
\int_{Q_T}p (1-\rho)\varphi=\lim \int_{Q_T}p_{h,m}(1-\rho_{h,m})\varphi=0
$$
for all $\varphi\geqslant 0$.
Because $p(1-\rho)\geqslant 0$ we conclude that $p(1-\rho)=0$ a.e. in $Q_T$ and the proof is achieved.

\end{proof}

We end this section with
\begin{proof}[Proof of Theorem \ref{existence HS implicit}]
We only sketch the argument and refer to \cite{GM} for the details.
Fix any $0<t_1<t_2$ and $\varphi\in \mathcal C^2_c(\R^d)$.
 Exploiting the Euler-Lagrange equations \eqref{opt Wasserstein}\eqref{opt Fisher-Rao} and summing from $k=k_1=\lfloor t_1/h \rfloor$ to $k=k_2-1=\lfloor t_2/h\rfloor-1$, we first obtain
 $$
 \int_{\R^d}\rho_{h,m}(t_2)\varphi-\rho_{h,m}(t_1)\varphi
 +\int_{k_1h}^{k_2h}\int_{\R^d}\trho_{h,m}\nabla \tilde p_{h,m}\cdot\nabla\varphi
 =-\int_{k_1h}^{k_2h}\int_{\R^d}\rho_{h,m}(1-p_{h,m})\varphi + R(h,m),
 $$
 where the remainder $R(h,m)\to 0$ for fixed $\varphi$.
 The strong convergence $\rho_{h,m},\tilde\rho_{h,m}\to \rho$ and the weak convergences $\nabla\tilde p_{h,m}\rightharpoonup \nabla \tilde p=\nabla p$ and $p_{h,m}\rightharpoonup p$ are then enough pass to the limit to get the corresponding weak formulation for all $0<t_1<t_2$.
 Moreover since the limit $\rho\in \C([0,T];\M^+_{\KFR})$ the initial datum $\rho(0)=\rho_0$ is taken at least in the sense of measures.
 This gives an admissible weak formulation of \eqref{equation HS}, and the proof is complete.
\end{proof}

\subsection{Numerical simulation}

The constructive scheme \eqref{splitting scheme implicit} naturally leads to a fully discrete algorithm, simply discretizing the minimization problem in space for each $\W,\FR$ step.
We use again the {\rm ALG2-JKO} scheme \cite{BCL} for the Wasserstein steps.
As already mentioned the Fisher-Rao step is a mere convex pointwise minimization problem, here explicitly given by: for all $x \in \Omega$,
$$ 
\rho_{h,m}^{k+1}(x) = \argmin_{\rho \geq 0}
\left\{ 4\left| \sqrt{\rho} - \sqrt{\rho_{h,m}^{k+1/2}(x)} \right|^2 + 2h \left( \frac{\rho^{m}}{m-1} -1 \right)\right\}
$$
and poses no difficulty in the practical implementation using a standard Newton method.

Figure \ref{figure m=100 } depicts the evolution of the numerical solution $\rho_{h,m}$ for $m=100$ and with a time step $h=0.005$.
We remark that the tumor first saturates the constraint ($\rho\nearrow 1$) in its initial support, and then starts diffusing outwards.
This is consistent with the qualitative behaviour described in \cite{PQV}.

\begin{figure}[h!]

\begin{tabular}{@{\hspace{0mm}}c@{\hspace{4mm}}c@{\hspace{4mm}}c@{\hspace{4mm}}c@{\hspace{4mm}}c@{\hspace{4mm}}}

\centering
\includegraphics[scale=0.19]{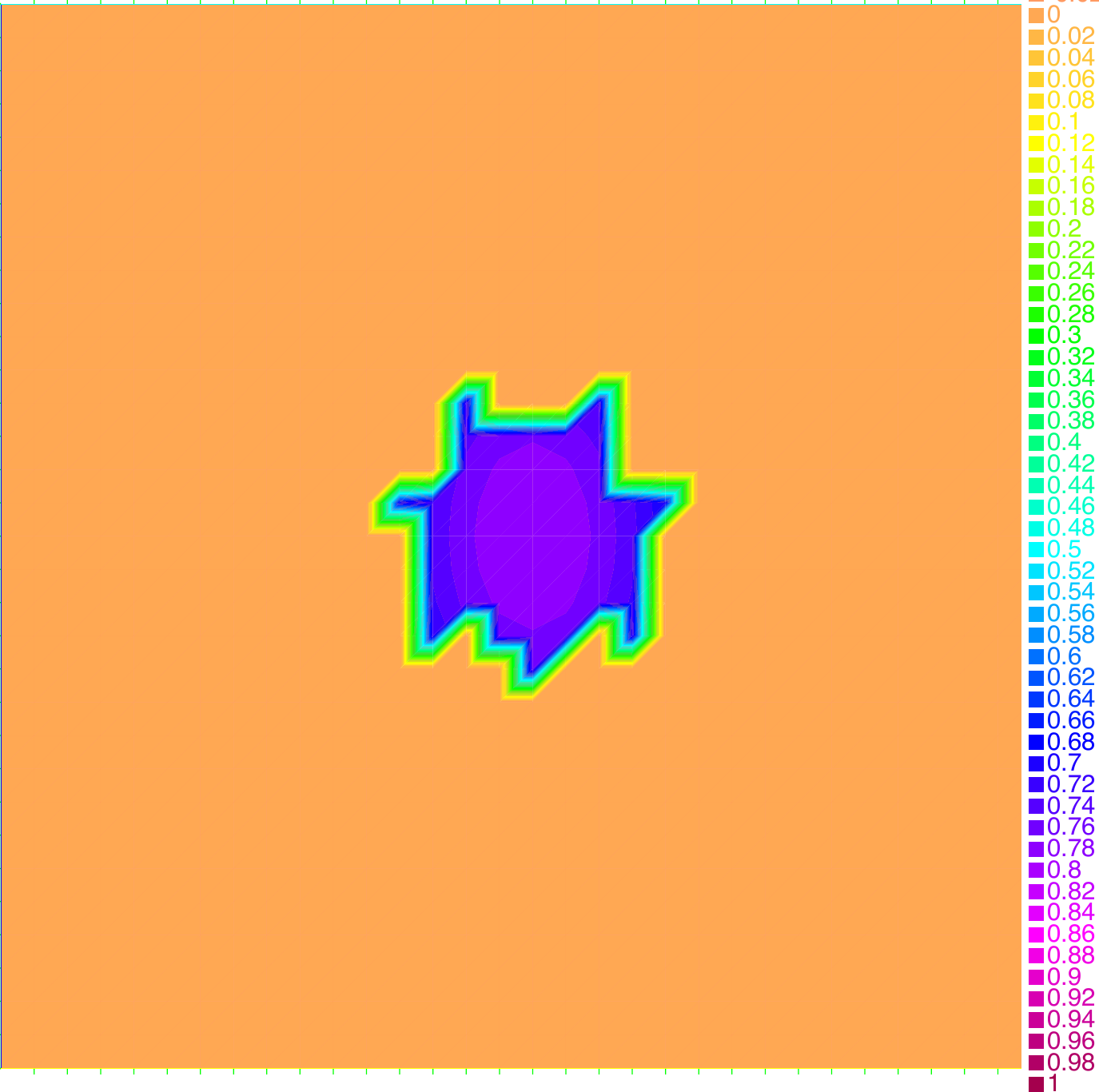}&
\includegraphics[scale=0.19]{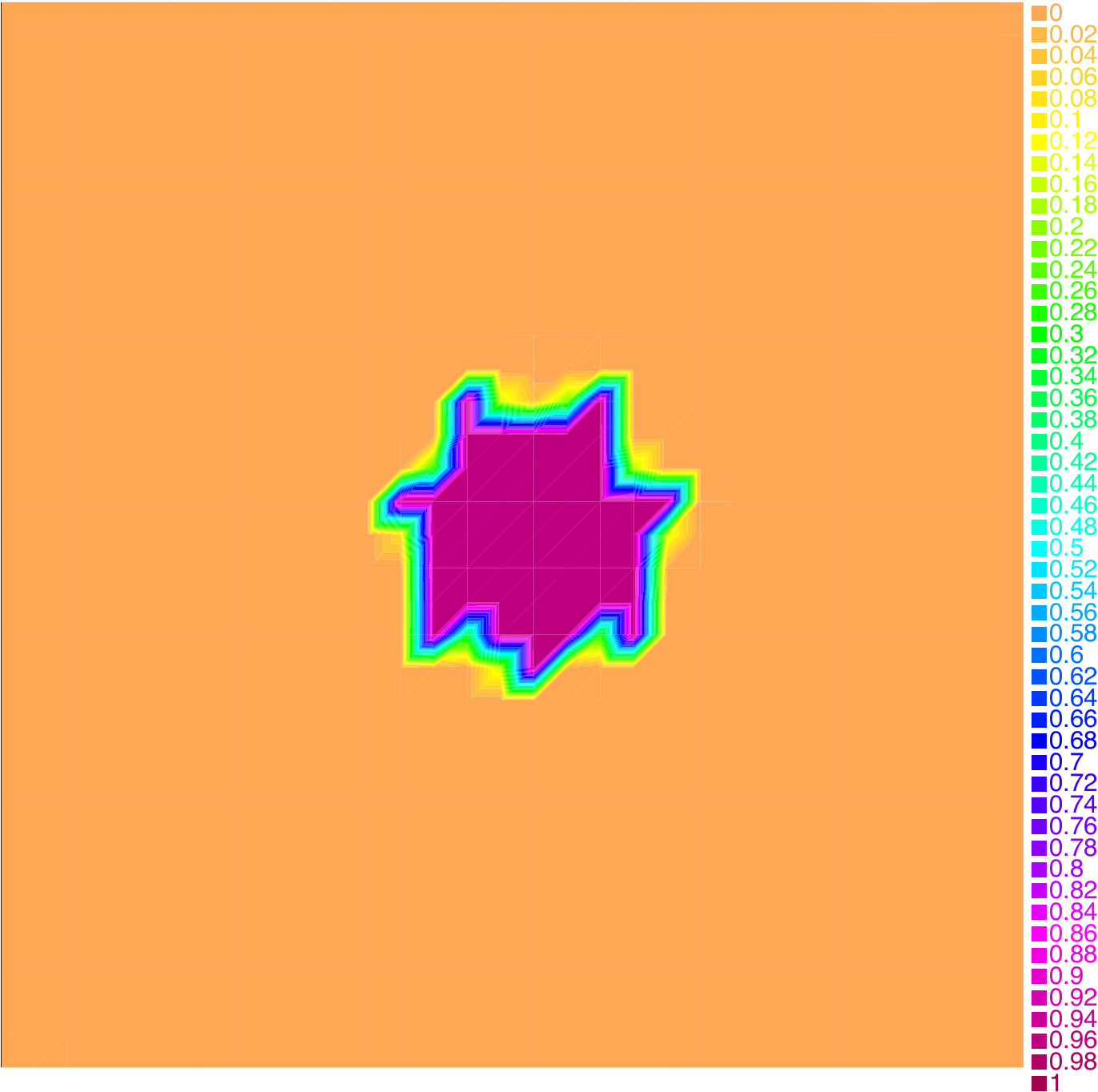}&
\includegraphics[scale=0.19]{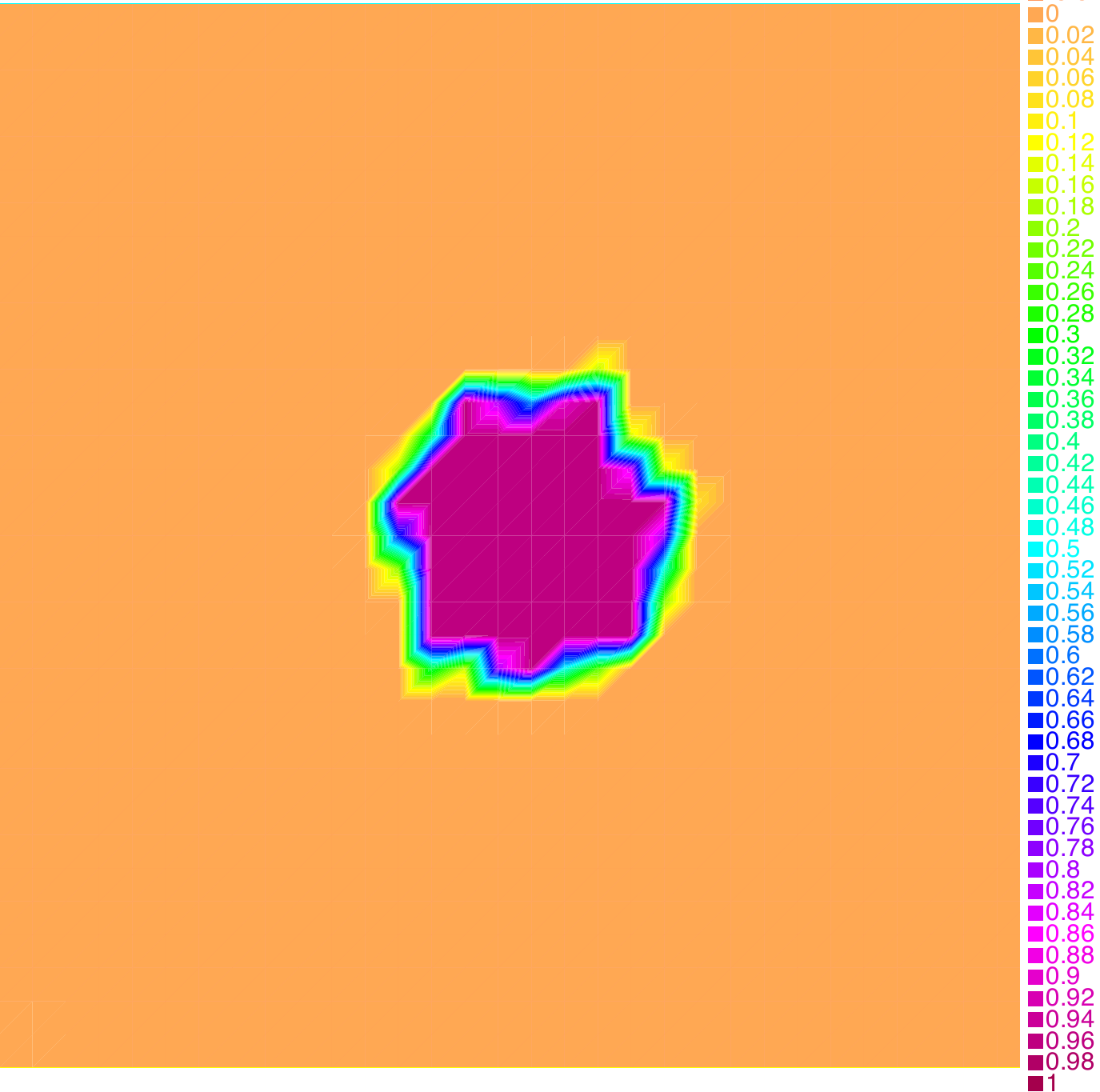}&
\includegraphics[scale=0.19]{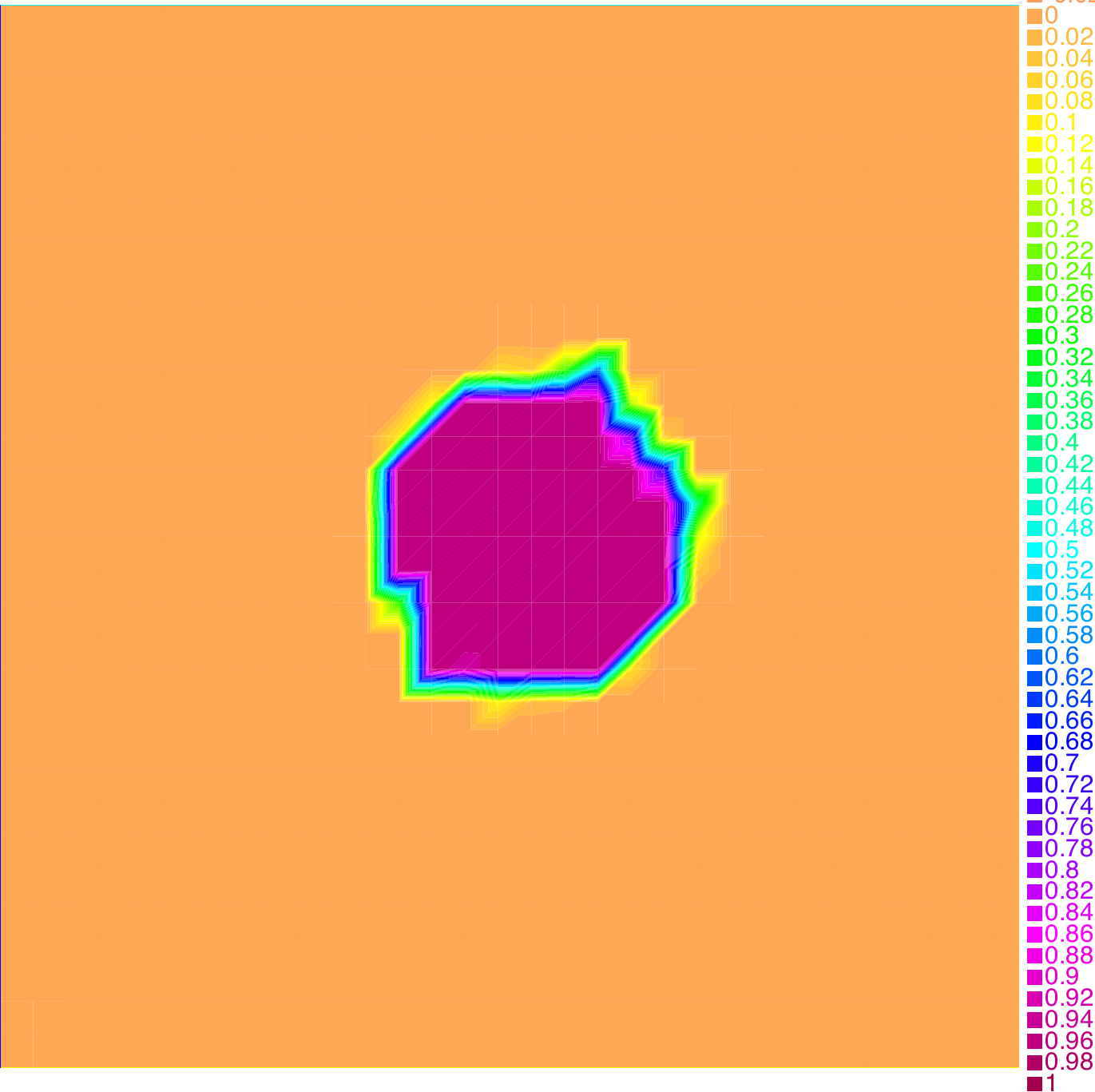}&
\includegraphics[scale=0.19]{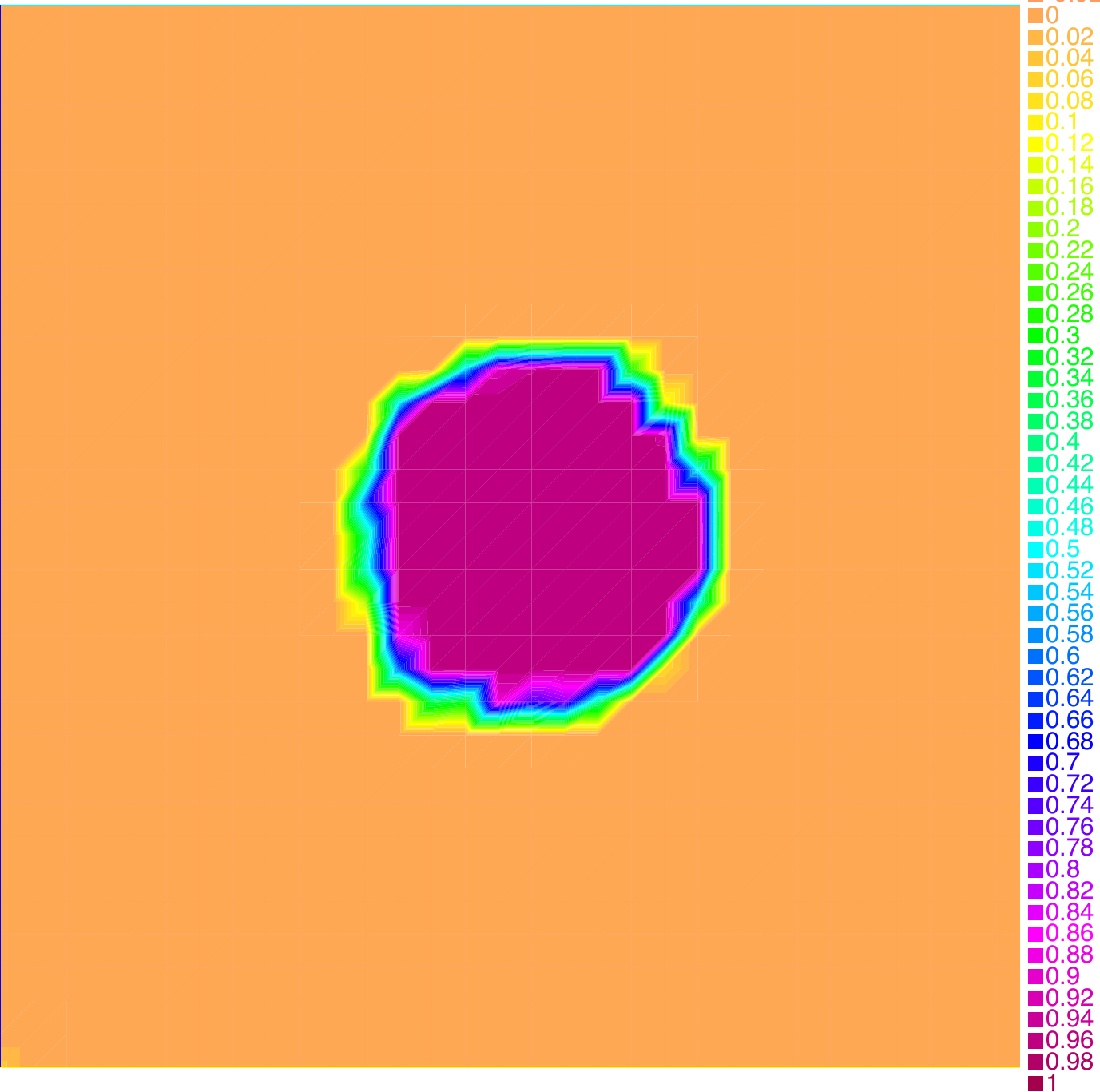}\\
$t=0$ & $t=0.3$ & $t=0.5$ & $t=0.7$ &$t=1$ \\
\end{tabular}
\caption{\textit{Snapshot of the approximate solution $\rho_{h,m}(t,.)$ to \eqref{equation HS}, with $m=100$, $h=0.005$.}}
\label{figure m=100 }
\end{figure}

\section{A tumor growth model with nutrient}
\label{part4-section2HSN}
In this section we use the same approach for the following tumor growth model with nutrients, appearing e.g. in \cite{PQV} 

\begin{eqnarray}
\label{system HS-nut}
\left\{\begin{array}{l}
\partial_t \rho -\dive( \rho \nabla p)= \rho \left((1-p)(c +  c_1) -c_2  \right),\\
\partial_t c - \Delta c = -\rho c,\\
0\leqslant \rho \leqslant 1,\\
p \geqslant 0 \mbox{  and  } p(1-\rho)=0, \\
\rho_{|t=0}=\rho_0, \, c_{|t=0}= c_0.
\end{array} \right.
\end{eqnarray}
Here $c_1$ and $c_2$ are two positive constants, and the nutrient $c$ is now diffusing in $\Omega$ in addition to begin simply consumed by the tumor $\rho$, according to the second equation.
For technical convenience we work here on a convex bounded domain $\Omega\subset\Rn$, endowed with natural Neumann boundary conditions for both $\rho$ and $c$.

Contrarily to section \ref{part4-section2HS} this is not a $\KFR$ gradient flow anymore, and we therefore introduce a semi-implicit splitting scheme.
Starting from the initial datum $\rho^0_{h,m}:=\rho_0,c_{h,m}^0:=c_0$ we construct four sequences $\rho_{h,m}^{k+1/2},\rho_{h,m}^k,c_{h,m}^{k+1/2},c_{h,m}^k$, defined recursively as
\begin{eqnarray}
 \label{splitting scheme w}
\left\{\begin{array}{l}
\rho_{h,m}^{k+1/2} \in \argmin\limits_{\rho \in \M^+, |\rho|=|\rho_{h,m}^k|} \left\{ \frac{1}{2h} \W^2(\rho, \rho_{h,m}^k) + \F_m(\rho ) \right\},\\
\\

c_{h,m}^{k+1/2} \in \argmin\limits_{c \in \M^+,|c|=|c_{h,m}^k|} \left\{ \frac{1}{2h} \W^2(c, c_{h,m}^{k}) + \E(\rho ) \right\},
\end{array}\right.
\end{eqnarray} 
and
\begin{eqnarray}
 \label{splitting scheme fr}
\left\{\begin{array}{l}
\rho_{h,m}^{k+1} \in \argmin\limits_{\rho \in \M^+} \left\{ \frac{1}{2h} \FR^2(\rho, \rho_{h,m}^{k+1/2}) +  \E_{1,m}(\rho | c_{h,m}^{k+1/2}) \right\},\\
\\
c_{h,m}^{k+1} \in \argmin\limits_{c\in \M^+} \left\{ \frac{1}{2h} \FR^2(c, c_{h,m}^{k+1/2}) + \E_2(c |\rho_{h,m}^{k+1/2}) \right\}, 

\end{array}\right.
\end{eqnarray} 
where
$$
\E(\rho):= \int_\Omega \rho \log(\rho),
$$
$$
\E_{1,m}(\rho |c):= \int_\Omega \left(c +c_1\right) \frac{\rho^m}{m-1} + \int_\Omega (c_2 - c -c_1) \rho ,
$$
and
$$ \E_2(c|\rho):= \int_\Omega \rho c.$$

As earlier it is easy to see that these sequences are well-defined (i-e there exists a unique minimizer for each step), and the pressures are defined as before as
$$
p_{h,m}^{k+1/2}:=\frac{m}{m-1}(\rho_{h,m}^{k+1/2})^{m-1}
\quad\mbox{and}\quad
p_{h,m}^{k+1}:=\frac{m}{m-1}(\rho_{h,m}^{k+1})^{m-1}.
$$
We denote again by $a_{h,m}(t),\tilde a_{h,m}(t)$ the piecewise constant interpolation of any discrete quantity $a^{k+1}_{h,m},a^{k+1/2}_{h,m}$ respectively.
Our main result reads:
\begin{thm}
\label{theo:existence system nutrient}
Assume $\rho_0 \in BV(\Omega)$ with $\rho_0 \leqslant 1$ and $c_0 \in L^\infty(\Omega) \cap BV(\Omega)$. Then $\rho_{h,m}$ and $\trho_{h,m}$ strongly converge to $\rho$ in $L^1((0,T)\times \Omega)$ and $c_{h,m}$ and $\tilde{c}_{h,m}$ strongly converge to $c$ in $L^1((0,T)\times \Omega)$ when $h\searrow 0$ and $m\nearrow+\infty$.
Moreover, if $mh\rightarrow 0$, then $p_{h,m},\tilde{p}_{h,m}$ converge weakly in $L^2((0,T) , H^1(\Omega))$ to a unique $p$, and $(\rho,p,c)$ is a solution of \eqref{system HS-nut}.
\end{thm}
Note that uniqueness of solutions would result in convergence of the whole sequence.
Uniqueness was proved in \cite[thm. 4.2]{PQV} for slightly more regular weak solutions, but we did not push in this direction for the sake of simplicity.
The method of proof is almost identical to section~\ref{part4-section2HS} so we only sketch the argument and emphasize the main differences. 

We start by recalling the optimality conditions for the scheme \eqref{splitting scheme w}-\eqref{splitting scheme fr}.
The Euler-Lagrange equations for the tumor densities in the Wasserstein and Fisher-Rao steps are
\begin{eqnarray}
\label{eq:optimality_rho_nutrients}
\left\{ \begin{array}{l}
\rho_{h,m}^{k+1/2} \nabla p_{h,m}^{k+1/2} = \frac{\nabla \varphi}{h}\rho_{h,m}^{k+1/2},\\
\sqrt{\rho_{h,m}^{k+1}} - \sqrt{\rho_{h,m}^{k+1/2}} = \frac{h}{2}\sqrt{\rho_{h,m}^{k+1}} \left( (1 - p_{h,m}^{k+1})(c_{h,m}^{k+1/2} + c_1) -c_2 \right),
\end{array}\right.
\end{eqnarray}
where $\varphi$ is a (backward) Kantorovich potential for $\W(\rho_{h,m}^{k+1/2},\rho_{h,m}^{k})$.
For the nutrient, the Euler-Lagrange equations are
\begin{eqnarray}
\label{eq:optimality_c_FR_nutrients}
\left\{ \begin{array}{l}
\nabla c_{h,m}^{k+1/2}   = \frac{\nabla \psi}{h}c_{h,m}^{k+1/2},\\
\sqrt{c_{h,m}^{k+1}} - \sqrt{c_{h,m}^{k+1/2}} =- \frac{h}{2}\sqrt{c_{h,m}^{k+1}} \rho_{h,m}^{k+1/2},
\end{array}\right.
\end{eqnarray}
with $\psi$ a Kantorovich potential for $\W(c_{h,m}^{k+1/2},c_{h,m}^{k})$.

Using the optimality conditions for the Fischer-Rao steps, we obtain directly the following $L^\infty$ bounds:

\begin{lem}
\label{borne sup inf 1 and nut}
For all $k \geqslant  0$
$$
\|c_{h,m}^{k+1} \|_{L^\infty(\Omega)} \leqslant \|c_{h,m}^{k+1/2} \|_{L^\infty(\Omega)} \leqslant \|c_{h,m}^{k} \|_{L^\infty(\Omega)},
$$
and at the continuous level
$$ 
\|c_{h,m}(t,\cdot) \|_{L^\infty(\Omega)}, \|\tilde{c}_{h,m}(t,\cdot) \|_{L^\infty(\Omega)} \leqslant \|c_0\|_{L^\infty(\Omega)}
\qquad \forall\, t\geq 0.
$$
Moreover, 
$$
\| \rho_{h,m}(t,\cdot) \|_\infty,\| \trho_{h,m}(t,\cdot) \|_\infty \leqslant 1
$$
and there exists $c_T \equiv c_T(\|c_0\|_{L^\infty}),C_T \equiv C_T(\|c_0\|_{L^\infty}) >0$ such that
\begin{equation}
\label{eq:encadrement nutrient}
\begin{array}{c}
(1-c_Th)\rho_{h,m}^{k+1/2}(x) \leqslant \rho_{h,m}^{k+1}(x) \leqslant (1+C_Th)\rho_{h,m}^{k+1/2}(x)
\qquad\mbox{a.e. in } \Omega.\\
(1-h) c_{h,m}^{k+1/2}(x) \leqslant c_{h,m}^{k+1}(x) \leqslant c_{h,m}^{k+1/2}(x) \qquad\mbox{a.e. in } \Omega.
\end{array}
\end{equation}
\end{lem} 

\begin{proof}
The proof of the estimates on $c_{h,m}$ and $\tilde{c}_{h,m}$ is obvious because one step of Wasserstein gradient flow with the Boltzmann entropy decreases the $L^\infty$-norm in \eqref{splitting scheme w} (see \cite{O1,A}), and, because the product $\sqrt{c_{h,m}^{k+1}} \rho_{h,m}^{k+1/2}$ is nonnegative in \eqref{eq:optimality_c_FR_nutrients}, the $L^\infty$-norm is also nonincreasing during the Fischer-Rao step.
The proof for $\rho_{h,m}$ and $\trho_{h,m}$ is the same as in lemma \ref{lem:borne sup inf 1}.
Using the fact that $\| \trho_{h,m}(t,\cdot) \|_\infty \leqslant 1$, we see that the term $\Phi(p_{h,m}^{k+1},c_{h,m}^{k+1/2}):=(1 - p_{h,m}^{k+1})(c_{h,m}^{k+1/2} + c_1) -c_2 $ in \eqref{eq:optimality_rho_nutrients} is bounded in $L^\infty$ uniformly in $k$.
This allows to argue exactly as in Lemma \ref{lem encadrement} to retrieve the estimate \eqref{eq:encadrement nutrient} and concludes the proof.  
\end{proof}

With these bounds it is easy to prove as in proposition \ref{décroissance F1} that
$$\begin{array}{c}
\F_m(\rho_{h,m}^{k+1}) \leqslant  \F_m(\rho_{h,m}^{k+1/2}) +C_Th,\\
\E_{1,m}(\rho_{h,m}^{k+1/2} |c_{h,m}^{k+1/2}) - \E_{1,m}(\rho_{h,m}^{k+1} |c_{h,m}^{k+1/2}) \leqslant  C_Th,\\
\E(c_{h,m}^{k+1})  \leqslant  \E(c_{h,m}^{k+1/2}) +C_Th,\\
\E_2(c_{h,m}^{k+1/2} |\rho_{h,m}^{k+1/2}) - \E_2(c_{h,m}^{k+1} |\rho_{h,m}^{k+1/2}) \leqslant  C_Th,
\end{array}.
$$
for some $C_T$ independent of $m$.
Then we obtain the usual $\frac 12$-H\"older estimates in time with respect to the $\KFR$ distance, which in turn implies that $\rho_{h,m},\trho_{h,m}$ converge to some $\rho \in L^\infty([0,T],L^1(\Omega))$ and $c_{h,m},\tilde{c}_{h,m}$ converge to some  $c\in L^\infty([0,T],L^1(\Omega))$ pointwise in time with respect to $\KFR$, see \eqref{sum telescopic FR}, Proposition~\ref{prop:1/2_holder}, and \eqref{estimate holder} for details.

As before we need to improve the convergence in order to pass to the limit in the nonlinear terms.
For $\rho_{h,m}$ and $\trho_{h,m}$, this follows from

\begin{lem}
For all $T>0$, if $\rho_0,c_0 \in BV(\Omega)$,
$$
\begin{array}{c}
\sup\limits_{t \in [0,T]} \left\{\| \rho_{h,m}(t,\cdot) \|_{BV(\Omega)}+\| c_{h,m}(t,\cdot) \|_{BV(\Omega)}   \right\} \leqslant e^{C_TT} (\| \rho_0 \|_{BV(\Omega)} +\|c_0 \|_{BV(\Omega)})\\
\sup\limits_{t \in [0,T]} \left\{\| \trho_{h,m}(t,\cdot) \|_{BV(\Omega)}+\| \tilde{c}_{h,m}(t,\cdot) \|_{BV(\Omega)}\right\} \leqslant e^{C_TT} (\| \rho_0 \|_{BV(\Omega)} +\|c_0 \|_{BV(\Omega)}).
\end{array}
$$
\end{lem}

\begin{proof}
The argument is a generalization of Lemma \ref{prop:BV_estimate}, see \cite[remark 5.1]{GM}. First, the $BV$-norm is nonincreasing during the Wasserstein step, \cite[thm. 1.1]{DPMSV},
$$ \| \rho_{h,m}^{k+1/2} \|_{BV(\Omega)} \leqslant  \| \rho_{h,m}^{k} \|_{BV(\Omega)} \text{  and  } \| c_{h,m}^{k+1/2} \|_{BV(\Omega)} \leqslant  \| c_{h,m}^{k} \|_{BV(\Omega)}.$$ 
Arguing as in Lemma \ref{prop:BV_estimate}, we observe that, inside $\supp \rho_{h,m}^{k+1/2}=\supp \rho_{h,m}^{k+1}$, the minimizer $\rho=\rho_{h,m}^{k+1}(x)$ is the unique positive solution of $f(\rho, \rho_{h,m}^{k+1/2}(x),c_{h,m}^{k+1/2}(x))=0$, with
$$
f(\rho, \mu , c) = \sqrt{\rho}\left( 1-\frac{h}{2}\left( \left( 1- \frac{m}{m-1}\rho^{m-1}\right)(c+c_1) -c_2 \right)\right) -\sqrt{\mu}.
$$
For $\mu >0$ the implicit function theorem gives as before a $\mathcal{C}^1$ map $R$ such that $f(\rho,\mu,c)=0 \Leftrightarrow \rho=R(\mu,c)$. 
An easy algebraic computation and \eqref{eq:encadrement nutrient} then gives $0 <\partial_\mu R(\mu,c) \leqslant (1+C_Th)$ and $|\partial_c  R(\mu,c) | \leqslant C_Th$ for some constant $C_T>0$ independent of $h,m,k$.
This implies that 
\begin{eqnarray*}
\| \rho_{h,m}^{k+1} \|_{BV(\Omega)} & \leqslant & (1+C_Th)\| \rho_{h,m}^{k+1/2} \|_{BV(\Omega)} + C_Th\| c_{h,m}^{k+1/2} \|_{BV(\Omega)}\\
&\leqslant & (1+C_Th)\| \rho_{h,m}^{k} \|_{BV(\Omega)} + C_Th\| c_{h,m}^{k} \|_{BV(\Omega)}.
\end{eqnarray*}

The same argument shows that
$$
\| c_{h,m}^{k+1} \|_{BV(\Omega)} \leqslant (1+C_Th)\| c_{h,m}^{k} \|_{BV(\Omega)} + C_Th\| \rho_{h,m}^{k} \|_{BV(\Omega)},
$$
and a simple induction allows to conclude.
\end{proof}

\begin{prop}
Up to extraction of a discrete sequence $h \to 0, m\to +\infty$,
$$
\rho_{h,m},\,\trho_{h,m}\to \rho \qquad \mbox{strongly in }L^1(Q_T)
$$
$$
p_{h,m}\rightharpoonup p \mbox{  and   }\tilde p_{h,m}\rightharpoonup \tilde p \qquad\mbox{weakly in all }L^q(Q_T)
$$
for all $T>0$.
If in addition $mh\to 0$ then $p=\tilde p \in L^2((0,T),H^1(\Omega))$ and $(\rho,p)$ satisfies
$$ 0 \leqslant \rho, p \leqslant 1 \quad \text{and}  \quad p(1-\rho)=0  \qquad \text{a.e. in } Q_T.
$$ 
\end{prop}

\begin{proof}
The proof is the same as Proposition \ref{prop:CV_phm_rhohm}, Lemma \ref{borne pression porous media}, and Lemma \ref{lem:properties pressure}.
\end{proof}

In order to conclude the proof of Theorem \ref{theo:existence system nutrient} we only need to check that $\rho,p,c$ satisfy the weak formulation of \eqref{system HS-nut}: the strong convergence of $\rho_{h,m},c_{h,m}$ and the weak convergence of $p_{h,m}$ are enough to take the limit in the nonlinear terms as in section \ref{subsec:prop_pressure_HS}, and we omit the details.

\subsection*{Acknowledgements}
We warmly thank G. Carlier for fruitful discussions and suggesting us the problem in section~\ref{sec:KFRsplitting}


\bibliographystyle{plain}
\addcontentsline{toc}{chapter}{Bibliography}
\bibliography{Laborde_bib}

\end{document}